\documentclass
{article}
\usepackage[T2A]{fontenc}
\usepackage[cp1251]{inputenc}
\usepackage[english,russian]{babel}
\usepackage{amsmath}
\usepackage{amsfonts,amssymb,mathrsfs,amscd}
\usepackage{verbatim}
\usepackage{eucal}
\usepackage{graphicx}
\usepackage{enumerate}
\usepackage{esint}
\usepackage{upgreek}
\let\No=\textnumero

\usepackage[hyper]{izv2e}
\JournalName{}

\numberwithin{equation}{section}

\theoremstyle{plain}
\newtheorem{theorem}{Теорема}
\newtheorem{maintheorem}{Основная теорема}
\newtheorem{corollary}{Следствие}

\newtheorem{theo}{Теорема} 

\newtheorem{lemma}{Лемма}
\newtheorem{propos}{Предложение}

\theoremstyle{definition}
\newtheorem{proof}{Доказательство}
\newtheorem{remark}{Замечание}
\newtheorem{example}{Пример}

\renewcommand{\leq}{\leqslant} 
\renewcommand{\geq}{\geqslant}
\newcommand{\RR}{\mathbb{R}} 
\newcommand{\CC}{\mathbb{C}} 
\newcommand{\NN}{\mathbb{N}} 
\newcommand{\ZZ}{\mathbb{Z}}
\newcommand{\DD}{\mathbb{D}} 

\DeclareMathOperator{\rad}{rad}
\DeclareMathOperator{\bal}{bal}
\DeclareMathOperator{\Bal}{Bal}

\DeclareMathOperator{\Zero}{Zero} 
\DeclareMathOperator{\supp}{supp} 
\DeclareMathOperator{\type}{type} 
\DeclareMathOperator{\strip}{str} 
\DeclareMathOperator{\ord}{ord}
\DeclareMathOperator{\rh}{rh}
 \DeclareMathOperator{\lh}{lh}
 \DeclareMathOperator{\up}{up}
\DeclareMathOperator{\dd}{d}
\renewcommand{\Re}{\operatorname{Re}}
\renewcommand{\Im}{\operatorname{Im}}

\begin{document} 

\title{Теорема Мальявена\,--\,Рубеля о малых целых функциях экспоненциального типа с заданными нулями: 60 лет спустя
}
\author[B.\,N.~Khabibullin]{Б.\,Н.~Хабибуллин}
\address{Башкирский государственный университет, Институт математики с вычислительным центром УФИЦ РАН}
\email{khabib-bulat@mail.ru} 

\date{19.04.2022}
\udk{517.538 : 517.574}

\maketitle
	
	\begin{fulltext}

\begin{abstract} 

 Пусть ${\mathrm Z}$ и $\mathrm W$ --- распределения точек на комплексной плоскости $\mathbb C$. 
Следующая задача восходит  к исследованиям Ф.~Карлсона, Т.~Карлемана, Л. Шварца, А.~Ф.~Леонтьева, Б.~Я.~Левина, Ж.-П.~Кахана и др.
При  каких  ${\mathrm Z}$ и $\mathrm W$   для целой функции $g\neq 0$ экспоненциального типа, обращающейся в нуль на  $\mathrm W$, найдётся целая  функция  $f\neq 0$ экспоненциального типа, обращающаяся в нуль на  ${\mathrm Z}$, для которой  $|f|\leq |g|$ на мнимой оси? Классическая теорема Мальявена\,--\,Рубела начала 1960-х гг.  полностью решает эту задачу для  <<положительных>>
 ${\mathrm Z}$ и $\mathrm W$, лежащих только на положительной полуоси. Ряд обобщений этого критерия  были установлены нами в конце 1980-х гг.  для <<комплексных>> ${\mathrm Z} \subset \mathbb C$ и ${\mathrm W}\subset \mathbb C$, отделённых углами от мнимой оси,  
с некоторыми  продвижениями  в 2020-е.  В настоящей статье решаются   более жёсткие задачи в  обобщающем субгармоническом обрамлении
для распределений масс на $\mathbb C$. Все предшествующие упоминавшиеся результаты могут быть получены из основных результатов статьи в гораздо более сильной  форме даже для  исходной постановки с распределениями точек ${\mathrm Z}$ и  ${\mathrm W}$ и целыми функциями $f$ и  $g$ экспоненциального типа. Часть результатов статьи тесно связана со знаменитыми теоремами Бёрлинга\,--\,Мальявена о радиусе полноты и мультипликаторе.

Библиография:  60 наименований

\end{abstract}
	

\markright{Распределения корней и масс   целых и субгармонических функций}

\footnotetext[0]{Исследование выполнено при финансовой поддержке Российского научного фонда, проект №~22-21-00026.}


\section{Введение}\label{s10}

\subsection{Постановки задач}\label{Sspz}
Голоморфную  на всей {\it комплексной плоскости\/} $\CC$, или\/ {\it целую\/} функцию $f$, для которой в обозначении $\ln^+x:=\max\bigl\{0. \ln x\bigr\}$
её тип 
\begin{equation}\label{efet}
\type\bigl[\ln |f|\bigr]:=\limsup_{z\to \infty}\frac{\ln^+|f(z)|}{|z|} 
\end{equation}
при порядке $1$ конечен,  называем {\it целой функцией\/  $f$   экспоненциального типа,\/} 
\cite{Boas}, \cite{Levin96}, \cite{RC},  \cite{Khsur}, хотя в отечественной литературе   ранее был широко распространён
и термин <<целая функция конечной степени>> \cite{Levin56}. 
Именно целые функции экспоненциального типа   чаще всего используются в разнообразных приложениях теории целых функций, к примеру, как реализующие сопряжённые пространства к функциональным пространствам  на подмножествах в $\CC$.

{\it Общая  задача\/} ---  для заданного распределения  точек ${\mathrm Z}$ на $\CC$  
найти условия существования целой функции экспоненциального типа   $f\not\equiv 0$, обращающейся в нуль на ${\mathrm Z}$ с учётом кратности (пишем $f({\mathrm Z})=0$),  при предписанных ограничениях сверху на модуль $|f|$  вдоль фиксированной прямой. Если учитывается лишь взаимное расположение этой прямой и распределения точек ${\mathrm Z}$, то  выбор прямой не имеет значения. В качестве такой прямой чаще всего рассматривалась {\it вещественная ось\/} $\RR\subset \CC$
или {\it мнимая ось\/} $i\RR\subset \CC$.
Далее в основном  придерживаемся последнего  выбора, что продиктовано исходной для  настоящей  статьи теоремой  Мальявена\,--\,Рубела \cite{MR},  \cite[гл.~22]{RC}, \cite[3.2]{Khsur} об условиях  существовании  целой функции экспоненциального типа  
\begin{equation}\label{fgiR}
f\not\equiv 0, \quad f({\mathrm Z})=0, \quad  \bigl|f(iy)\bigr|\leq  \bigl|g(iy)\bigr|\quad \text{\it при всех $y\in \RR$},
\end{equation}
где $g\not\equiv 0$ ---   целая функция экспоненциального типа  с $g({\mathrm W})=0$ для заданного распределении точек ${\mathrm W}$. 
 Решение этой задачи для распределений точек из ${\mathrm Z}$ и ${\mathrm W}$, полностью лежащих на положительной полуоси 
$\RR^+:=\bigl\{x\in \RR\bigm| x\geq 0\bigr\}$,   давалось в  \cite{MR} и  \cite[гл.~22]{RC}  в терминах соотношений между распределениями точек  ${\mathrm Z}$  и  ${\mathrm W}$.  В настоящей статье  по задаче существования целой функции $f$ экспоненциального типа   из \eqref{fgiR}  будем рассматривать, как правило,  соотношения между распределениями точек   ${\mathrm Z}$ на $\CC$ и ростом вдоль $i\RR$ заданной  целой функции $g$ экспоненциального типа.  Ещё более содержательна    {\it субгармоническая версия\/} задачи 
о существовании целой функции экспоненциального типа 
\begin{equation}\label{fgiRM}
f\not\equiv 0, \quad f({\mathrm Z})=0, \quad   \ln \bigl|f(iy)\bigr|\leq  M(iy)\quad \text{\it при всех $y\in \RR\setminus  E$},
\end{equation}
где $M\not\equiv -\infty$ --  {\it субгармоническая функция конечного типа\/}  
\begin{equation}\label{typeM}
\type[M]:=\limsup_{z\to \infty}\frac{M^+(z)}{|z|}\in \RR^+, \quad M^+(z):=\max\bigl\{0,M(z)\bigr\},  
\end{equation}
 при порядке $1$ на $\CC$,   а $E$ --- достаточно малое  исключительное множество.  Эту задачу также естественно решать   в терминах соотношений между распределением точек из ${\mathrm Z}$ {\it с одной стороны\/}  и 
поведением функции $M$, а также размерами исключительного множества $E$  {\it с другой.\/}
Основные рассмотренные  в настоящей статье  задачи охватывают и более общие и жёсткие субгармонические версии исходных задач \eqref{fgiR}--\eqref{fgiRM}. В них для  заданного распределения масс $\nu$ на $\CC$
с некоторыми ограничениями на него  вблизи мнимой оси и для 
{\it произвольной\/} субгармонической функции $M\not\equiv -\infty$ конечного типа даются одновременно необходимые и достаточные условия, при которых для  числа $b\in \RR^+$  с соответствующей   {\it вертикальной открытой\/} или {\it  замкнутой полосой 
\begin{equation}\label{{strip}c}
 \strip_b:=\Bigl\{z\in \CC\Bigm| |\Re z|< b\Bigr\}, \quad 
\overline \strip_b:=\Bigl\{z\in \CC\Bigm| |\Re z|\leq b\Bigr\}
\end{equation}
ширины\/ $2b$ со средней линией $i\RR$} найдётся  субгармоническая функция $U\not\equiv -\infty$ конечного типа  с {\it распределением масс  Рисса  не меньшим, чем $\nu$,\/} с тождеством  
\begin{equation}\label{UM=}
U(z)\equiv    M(z)
\quad\text{\it при всех  $z\in \strip_b$ или $z\in \overline \strip_b$}. 
\end{equation}
Рассматривается и  специальный   выбор такой функции $U=v+\ln |h|$ конечного типа 
с субгармонической функцией $v$ с распределением масс Рисса, равным  $\nu$,  и  целой функцией  $h\not\equiv 0$ с неравенствами
\begin{equation}\label{UMbbul}
v(z)+\ln \bigl|h(z)\bigr|\leq   M^{\bullet r}(z)
\quad\text{\it при всех  $z\in \overline \strip_b$}.
 \end{equation}
где $M^{\bullet r}(z)$ --- интегральные средние  функции $M$ по кругам с центрами $z$ и очень  быстро убывающими  к нулю   радиусами $r(z)$ при  $z\to \infty$, а также  
\begin{equation}\label{{fgiRMu}l}
v(z)+\ln \bigl|h(z)\bigr|\leq  M(z)\quad\text{\it для всех  $z\in \overline \strip_b\setminus E$},
\end{equation}
где исключительное множество $E\subset \CC$  очень мал\'о. При этом  некоторые  ограничения на распределение масс $\nu$ вблизи и на мнимой оси  неизбежны, если искать решения в простых геометрических терминах 
{\it логарифмических функций интервалов и субмер\/ 
\eqref{df:dDlm+}--\eqref{df:dDlLm} на\/} $\RR^+$ для  $\nu$, восходящих к 
логарифмическим характеристикам  \eqref{ZWR}  распределений точек на $\RR^+$ из \cite{MR}, \cite[гл. 22]{RC}.  

Здесь можно перейти и  к п.~\ref{prr1_2},  а к  п.~\ref{11def}  обращаться по мере необходимости.  

\subsection{Некоторые  обозначения, определения, соглашения}\label{11def} 
 Одноточечные множества $\{a\}$ часто записываем без фигурных скобок, т.е. просто как $a$. Так,   $\NN_0:=0\cup \NN=\{0,1, \dots\}$ для  множества $\mathbb N:=\{1,2, \dots\}$ {\it натуральных чисел,\/}  
$\CC_{\infty}:=\CC\cup \infty$ и $\overline \RR:=-\infty \cup \RR\cup +\infty$ --- {\it расширенные\/}
комплексная плоскость и вещественная ось с $-\infty:=\inf \RR\notin \RR$, $+\infty:=\sup \RR\notin \RR$, неравенствами $-\infty\leq x\leq +\infty$ для любого $x\in \overline \RR$
 и естественной порядковой топологией. По определению 
$\sup \varnothing:=-\infty$ и $\inf \varnothing:=+\infty$ для {\it пустого множества\/} $\varnothing$.
Символом  $0$, кроме нуля,   могут обозначаться {\it нулевые\/} функции,  меры  и пр.

Для $x\in X\subset \overline \RR$ его {\it положительную часть\/} обозначаем как $x^+:=\sup\{0,x \}$, $X^+:=\bigl\{x^+\bigm|  x\in X\bigr\}$. {\it Расширенной числовой функции\/} $f\colon S\to \overline \RR$ сопоставляем её {\it положительную часть\/} $f^+\colon s\underset{s \in S}{\longmapsto} (f(s))^+\in \overline{\RR}^+$ 
и {\it отрицательную часть\/} $f^-:=(-f)^+\colon S\to \overline{\RR}^+ $. Как обычно, пишем $f\not\equiv c$, если функция $f$ принимает хотя бы одно значение, отличное от $c$, в области её определения. 

Для $x_0\in \RR$ и расширенной числовой функции $m\colon [x_0,+\infty) \to \overline \RR$ определим  
\begin{equation}\label{senu0:a} 
\ord[m]:=\limsup_{x\to +\infty} \frac{\ln \bigl(1+m^+(x)\bigr)}{\ln x}\in
\overline \RR^+
 \end{equation} 
{\it --- порядок\/}  (роста) функции $m$ (около $+\infty$), а  для $p\in \RR^+$
\begin{equation} 
\type_p[m]:=\limsup_{x\to +\infty} \frac{m^+(x)}{x^p}\in \overline \RR^+
\label{typevf}
 \end{equation} 
{\it --- тип\/} (роста) функции $m$ {\it при порядке\/} $p$ (около $+\infty$) \cite{Boas}, \cite{Levin56}, \cite{Levin96}, 
\cite{Kiselman}, \cite[2.1]{KhaShm19}, а для произвольной 
 функции $u\colon \CC\to \overline \RR$ с  {\it радиальный функцией роста\/}
\begin{equation} 
{\mathrm M}_u \colon r\underset{r\in \RR^+}{\longmapsto}
\sup\bigl\{u(z)\bigm| |z|=r\bigr\}
\label{u}
\end{equation} 
по определению $\ord[u]:=\ord[{\mathrm M}_u]$ и $\type_p[u]:=\type_p[{\mathrm M}_u]$ 
--- соответственно {\it порядок\/} и 
{\it тип функции $u$ при порядке $p$} \cite{Boas}, \cite{Levin96}, \cite{Kiselman}, \cite[Замечание 2.1]{KhaShm19}.  Функции $u$ конечного типа 
$\type_1[u]\in \RR^+$ при порядке $p=1$ называем просто функциями {\it конечного типа,\/} не упоминая  и не указывая 
порядок $1$ в $\type[u]:=\type_1[u]$,  как это делалось выше  для целых функций экспоненциального типа в \eqref{efet} 
и для субгармонических функций конечного типа в \eqref{typeM}. 

{\it Распределением масс\/} --- это  {\it положительная   мера Радона\/}  \cite{EG},  \cite[Дополнение A]{Rans}, \cite[гл.~3]{HK}, 
а {\it распределение зарядов\/} ---  разность распределений масс \cite{Landkof}.
Для распределений масс   или зарядов  {\it на\/} $\CC$, как правило, не указываем, где  они заданы.   
Для  субгармонической  в области из $\CC$ функции $u\not\equiv -\infty$ действие  на неё   {\it оператора Лапласа\/} ${\bigtriangleup}$  в смысле теории обобщённых функций  определяет её {\it распределение масс Рисса\/}
\begin{equation}\label{Riesz}
\frac{1}{2\pi}{\bigtriangleup}u=:\varDelta_u
\end{equation}
 в этой области  \cite{HK}, \cite{Rans}, \cite{Az}. Для обозначения распределения масс Рисса функции $u$ используем  как первую  форму записи $\frac{1}{2\pi}{\bigtriangleup}u$ из \eqref{Riesz}, так и вторую  $\varDelta_u$.
 Далее $D_z(r):=\bigl\{w \in \CC \bigm| |w-z|<r\bigr\}$ и  $\overline{D}_z(r):=\bigl\{w \in \CC_{\infty} \bigm| |w-z|\leq r\bigr\}$,
а также $\partial \overline{D}_z(r):=\overline{D}_z(r)\setminus {D}_z(r)$
---  соответственно {\it открытый} и   {\it замкнутый круги,\/} а также  {\it окружность  радиуса\/ $r\in \overline \RR^+$
 с центром\/} $z\in \CC$, а   $\DD:=D_0(1)$ и  $\overline \DD:=\overline D_0(1)$, а также  
$\partial \overline \DD:=\partial \overline{D}_0(1)$ --- соответственно {\it открытый\/} и {\it замкнутый 
единичные круги,\/} а также {\it единичная окружность} в  $\CC$. 

Через $\CC_{ \rh}:=\bigl\{z\in \CC \bigm| \Re z>0\bigr\}$ и $\CC_{\overline  \rh}:=\CC_{ \rh}\cup i\RR$, а также
$\CC_{ \lh}:=-\CC_{ \rh}$ и $\CC_{\overline \lh}:=-\CC_{\overline  \rh}$ обозначаем 
соответственно {\it правые открытую} и   {\it замкнутую полуплоскости,\/} а также  
{\it левые открытую} и   {\it замкнутую  полуплоскости\/} в  $\CC$. 

Для распределения зарядов  $\nu$ на $S\subset \CC$ через $\nu^+:=\sup\{\nu,0\}$, $\nu^-:=(-\nu)^+$
и $|\nu|:=\nu^++\nu^-$ обозначаем соответственно {\it верхнюю, нижнюю\/}
и {\it полную вариации\/} распределения зарядов  $\nu$, а $\supp \nu=\supp |\nu|$ --- его {\it носитель,} но  распределение зарядов  $\nu$ {\it сосредоточено на $\nu$-измеримом подмножестве $S_0\subset S$,\/} если полная вариация $|\nu|$ дополнения $S\setminus S_0$ множества $S$ равна нулю. 

{\it Сужение\/} функции $f$ на  $S\subset \CC$ обозначаем как   $f{\lfloor}_S$. Аналогично через $\nu{\lfloor}_S$
обозначается  обозначается и   {\it сужение\/}  положительной меры Бореля или  распределения зарядов $\nu$ на  $\nu$-измеримое  $S\subset \CC$.
При $r\in \overline \RR^+$ для таких   $\nu$  через 
\begin{equation}\label{df:nup} 
\nu_z^{\rad} (r):=\nu \bigl(\,\overline D_z(r)\bigr),\quad \nu^{\rad}(r):=\nu_0^{\rad}(r)
=\nu\bigl(r\overline \DD\bigr)
\end{equation} 
обозначаем  {\it радиальные непрерывные справа считающие  функции  распределения зарядов  $\nu$ с
центрами\/}  соответственно {\it в точке\/ $z\in \CC$} и {\it в нуле.\/}

{\it Верхняя плотность распределения зарядов\/ $\nu$ при порядке\/} $p\in \RR^+$  равна 
\begin{equation}\label{typenu}
\type_p[\nu]:=\type_p\bigl[|\nu|\bigr]
\overset{\eqref{typevf}}{:=} \limsup_{0<r\to +\infty} \frac{|\nu|(r\overline \DD)}{r^p}
\overset{\eqref{df:nup}}{=} \limsup_{0<r\to +\infty} \frac{|\nu|^{\rad}(r)}{r^p}
\in \overline \RR^+,
\end{equation} 
и при $p=1$ упоминание о порядке  опускаем.   В частности, {\it  распределение зарядов\/} $\nu$
{\it конечной верхней плотности,\/} если $\type[\nu]:=\type_1[\nu]<+\infty$. {\it Порядок распределения зарядов\/}
$\nu$ определяется как $\ord[\nu]\overset{\eqref{senu0:a}}{:=}\ord\bigl[|\nu|^{\rad}\bigr]$
через \eqref{df:nup}. 

Всюду далее для  {\it распределений,\/} вообще говоря,  повторяющихся {\it точек на\/} $\CC$  предполагается, 
что в каждом круге $r\DD$ при $r\in \RR^+$ содержится конечное число точек из этого распределения точек,  т.е. рассматриваются только   {\it локально конечные в\/} $\CC$ распределения точек. {\it Распределение зарядов\/} называем {\it целочисленным,\/} если он принимает только целые значения из $\ZZ:=\NN_0\cap (-\NN)$ на ограниченных множествах. 
{\it Распределение точек\/} ${\mathrm Z}$ можно трактовать как целочисленное  распределение масс, для которого масса каждого ограниченного в $\CC$ множества  равна числу попавших в него точек из ${\mathrm Z}$.  Для этого  целочисленного распределения масс сохраняем то же самое обозначение ${\mathrm Z}$. Всюду  далее ${\mathrm Z}$ и ${\mathrm W}$ --- распределения точек на $\CC$.
Таким образом, каждое целочисленное распределение масс однозначно 
определяет локально конечное распределение точек, и наоборот, 
а  равенство ${\mathrm Z}={\mathrm W}$ понимается как равенство соответствующих распределений масс  \cite[0.1.2]{Khsur}. 
Все вводимые в статье  понятия и обозначения для распределений зарядов и масс  переносятся и на распределения точек. 
Так,  ${\mathrm Z}{\lfloor}_{S}$ --- {\it сужение\/} распределения точек ${\mathrm Z}$ на $S\subset \CC$.
Если $\supp {\mathrm Z}\subset S$, то для краткости часто пишем просто ${\mathrm Z}\subset S$.

Для целой функции $f\not\equiv 0$ через $\Zero_f$ обозначаем её {\it распределение корней,\/} представляющее собой распределение точек, в котором каждая точка $z\in \CC$ повторяется столько раз, какова кратность корня функции $f$ в точке  $z$. При этом  $\Zero_f=\frac{1}{2\pi}{\bigtriangleup}\ln |f|$ --- целочисленное распределение масс Рисса субгармонической функции $\ln|f|$
\cite[теорема 3.7.8]{Rans}. 
Целая функция $f\not\equiv 0$ {\it обращается в нуль на\/} ${\mathrm Z}$ и пишем $f({\mathrm Z})=0$, если имеет место неравенство ${\mathrm Z}\leq \Zero_f$ для целочисленных распределений масс  ${\mathrm Z}$ и $\Zero_f$.  

Распределение точек  ${\mathrm Z}$ при его нумерации номерами из  $N\subset \ZZ$ можно рассматривать как некоторую {\it последовательность  ${\mathrm Z}=({\mathrm z}_n)_{n\in N}$ комплексных чисел,} где каждое число  $z_n=z\in \CC$ встречается ровно ${\mathrm Z}(z)$ раз,  т.е. столько же раз, сколько раз точка $z\in \CC$ повторяется  в распределении точек ${\mathrm Z}$.
При этом  
\begin{equation}\label{sumZ}
\sum_{\stackrel{z\in S}{z\in {\mathrm Z}}}f(z):=\int_Sf\dd {\mathrm Z}=\sum_{{\mathrm z}_n\in S} f({\mathrm z}_n)
\end{equation}
при нумерации   ${\mathrm Z}=({\mathrm z}_n)_{n\in N}$ для $S\subset \CC$, когда для расширенной числовой функции $f$ на $\supp {\mathrm Z}$  интеграл и сумма  справа в \eqref{sumZ}
корректно определены. 

Если  при некоторой нумерации  ${\mathrm Z}=({\mathrm z}_n)_{n\in N}$   на $\CC$ можно так  подобрать {\it последовательность попарно различных целых чисел $({\mathrm m}_n)_{n\in N}$} и $c\in \mathbb{R}^+$,  что  
\begin{equation*}
\sum_{k\in N}\Bigl|\frac{1}{{\mathrm z}_n}-\frac{c}{i{\mathrm m}_n}\Bigr|<+\infty, 
\end{equation*}
то   {\it внешняя плотность Редхеффера от ${\mathrm Z}$ вдоль\/ $i\mathbb{R}$} 
не превышает $c$ \cite{Red77}, \cite{Kra89}, \cite{Kha94}, \cite[2.1.1]{Khsur}, \cite{Sal22}, а сама она равна 
точной нижней грани таких  $c\in \RR^+$.

\subsection{Предшествующие результаты для целых функций}\label{prr1_2}
 Задача существования целой функции $f$ экспоненциального типа со свойствами   \eqref{fgiR} 
для любой целой функции $g\not\equiv 0$ экспоненциального типа, обращающейся в нуль  на заданном распределении точек ${\mathrm W}$,
в терминах соотношений между ${\mathrm Z}$ и ${\mathrm W}$  была полностью решена в начале 1960-х гг.  в совместной работе  П.~Мальявена и Л.\,А.~Рубела \cite[теорема 4.1]{MR}, но  лишь для пар  распределений 
{{\it положительных\/}} точек ${\mathrm Z}$ и ${\mathrm W}$ на $\RR^+$. 
 Это решение без каких-либо существенных дополнений  изложено и  в  одной  из основных глав монографии Л.\,А.~Рубела с Дж.\,Э.~Коллиандром \cite[гл.~22]{RC} 1996 г.

\begin{theo}[{Мальявена\,--\,Рубела}] Для   распределений точек ${\mathrm Z}\subset \RR^+\setminus 0$
 и  ${\mathrm W}\subset \RR^+\setminus 0$ конечной верхней плотности   эквивалентны три утверждения:
\begin{enumerate}[{\rm I.}]
\item\label{IMR}  Для любой целой функции   $g\not\equiv 0$ экспоненциального типа с  $g({\mathrm W})=0$  найдётся 
целая функция  $f$ экспоненциального типа  со свойствами  \eqref{fgiR}.
\item\label{IIMR} Существует $C\in \RR$, для которого в обозначениях \eqref{sumZ}
\begin{equation}\label{ZWR}
\sum_{\stackrel{r<z\leq R}{z\in {\mathrm Z}}}\frac{1}{z}\leq
\sum_{\stackrel{r<w\leq R}{w\in {\mathrm W}}}\frac{1}{w}+ C\quad \text{при всех\/ $0<r<R<+\infty$.}
\end{equation}

\item\label{IIIMR} Существуют целая функция   $g$ экспоненциального типа с  распределением корней  
$\Zero_g$ с  сужением  $\Zero_g{\lfloor}_{\CC_{\rh}}={\mathrm W}$ на $\CC_{ \rh}$  
и целая функция   $f$ экспоненциального типа, для которой имеет место  \eqref{fgiR}.
\end{enumerate}
\end{theo}

 В наших работах 1988--89 гг.   были  получены подобные результаты для  
 {\it комплексных\/}   ${\mathrm Z} $ и  ${\mathrm W}$ на $\CC$ в форме, близкой к   \eqref{fgiR}, но с некоторой  малой добавкой к правой   части неравенства из  \eqref{fgiR}.  Так, при $g\equiv 1$ имеет место 
\begin{theorem}[{(\cite[основная теорема]{Kha88})}]\label{theK0}
Для распределения точек ${\mathrm Z}\subset \CC$ при любом $\varepsilon \in \RR^+\setminus 0$ существует целая функция  $f\not\equiv 0$ 
экспоненциального типа с $f({\mathrm Z})=0$ и $\ln |f(iy)|\leq \varepsilon |y|$ при всех $y\in \RR$, если и только если 
${\mathrm Z}$ конечной верхней плотности и для любого $\varepsilon \in \RR^+\setminus 0$ найдётся $C_{\varepsilon}\in \RR$, для которого 
\begin{equation}\label{Kh88e}
\sum_{\stackrel{r<|z|\leq R}{z\in {\mathrm Z}}}\Bigl|\Re \frac{1}{z}\Bigr|
\leq  \varepsilon \ln\frac{R}{r}+C_{\varepsilon} 
\quad\text{при всех\/ $0< r<R<+\infty$.}
\end{equation}
\end{theorem}
Доказательство теоремы \ref{theK0} в \cite[основная теорема]{Kha88} наряду с  теоремой Мальявена\,--\, Рубела существенно  использовало    специальный случай ${\mathrm Z}\subset i\RR$ из статьи И.\,Ф.~Кра\-с\-и\-ч\-к\-о\-ва-Те\-р\-н\-о\-в\-с\-к\-о\-го  1972 г. \cite[Теоремы 8.3, 8.5]{Kra72} по спектральному синтезу. Этот специальный случай ${\mathrm Z}\subset i\RR$ был значительно усилен в нашей статье     \cite[основная теорема]{Kha01l} 2001 г. со взаимно обратным балансом на рост целой функции экспоненциального типа  вдоль $\RR$ и $i\RR$, что в некотором  частном  случае достаточно полно  отражает следующий результат. 
\begin{theorem}[{\cite[теорема 1]{Kha01l}, \cite[теорема 3.3.8]
{Khsur}}]\label{thK1-1}
Пусть  ${\mathrm Z}\subset \CC$ --- распределение точек   конечной верхней плотности 
$\type[{\mathrm Z}]\overset{\eqref{typenu}}{<}d\in \RR^+$ и $\lim\limits_{{\mathrm Z}\ni z\to\infty}{|\Re z|/|z|}=0$. Тогда для  любого числа  $\varepsilon \in (0,1)$
 существует целая функция экспоненциального типа  $f\not\equiv 0$ с $f({\mathrm Z})=0$, для которой одновременно выполнены неравенства 
\begin{equation}\label{xy}
\begin{cases}
\ln\bigl|f(iy)\bigr|&\leq  \frac{\varepsilon}{250 d}|y| \quad\text{при всех $y\in \RR$},
\\
\ln\bigl|f(x)\bigr|&\leq \frac{250 d}{\varepsilon}|x|\quad\text{при всех $x\in \RR$.}
\end{cases}
\end{equation}
\end{theorem}
Для  случая  $\ln |g(iy)|+ \varepsilon |y|$ по-преж\-н\-е\-му со сколь угодно малым числом  $\varepsilon >0$ и  целой функцией  $g\neq 0$ экспоненциального типа  
в 1989 г. установлена 

\begin{theorem}[{\rm (\cite[основная теорема]{KhaD88}, \cite[основная теорема]{Kha89}, \cite[теорема 3.2.1]{Khsur})}]\label{theKh}
Для  любых  распределений точек ${\mathrm Z}\subset \CC$
 и  ${\mathrm W}\subset \CC_{\rh}$ конечной верхней плотности   эквивалентны следующие три утверждения:

\begin{enumerate}[{\rm I.}]
\item\label{IKh}  Для любой целой функции   $g\not\equiv 0$ экспоненциального типа, обращающейся в нуль на  ${\mathrm W}$,  и  любого 
$\varepsilon \in \RR^+\setminus 0$  найдутся  целая функция   $f\not\equiv 0$  экспоненциального типа, обращающаяся в нуль на  ${\mathrm Z}$,  и борелевское  множество $E\subset \RR$ конечной линейной меры Лебега  ${\mathfrak m}_1(E)<+\infty$, для которых 
\begin{equation}\label{fgiK}
\ln\bigl|f(iy)\bigr|\leq \ln\bigl|g(iy)\bigr|+\varepsilon |y|\quad\text{при всех $y\in \RR\setminus E$.}
\end{equation}
\item\label{IIKh} Для логарифмической субмеры распределения точек ${\mathrm Z}$, определённой как
\begin{equation}\label{{Kh89e}m}
\ell_{{\mathrm Z}}(r,R)\overset{\eqref{sumZ}}{:=}\max \left\{ \sum_{\stackrel{r<|z|\leq R}{z\in {\mathrm Z}}}\Re^+ \frac{1}{z},
\sum_{\stackrel{r<|z|\leq R}{z\in {\mathrm Z}}}\Re^- \frac{1}{z}\right\},
\end{equation}
при любом  $\varepsilon \in \RR^+\setminus 0$ существует $C_{\varepsilon}\in \RR$, для которого 
\begin{align}
\ell_{{\mathrm Z}}(r,R)\leq \ell_{{\mathrm W}}(r,R)
+\varepsilon \ln\frac{R}{r}+C_{\varepsilon} \quad\text{при всех\/ $0< r<R<+\infty$.}
\label{{Kh89e}g}
\end{align}

\item\label{IIIKh} Существует некоторая  целая функция   $g\not\equiv 0$ экспоненциального типа с   $\Zero_g{\lfloor}_{\CC_{\rh}}={\mathrm W}$, для которой  при любом $\varepsilon \in \RR^+\setminus 0$  найдутся  такие целая функция   $f\not\equiv 0$ экспоненциального типа с $f({\mathrm Z})=0$  и борелевское  подмножество $E\subset \RR$, что\/    ${\mathfrak m}_1(E)<+\infty$  и имеет место \eqref{fgiK}.
\end{enumerate}
\end{theorem}

При переходе  к более жёсткому, по сравнению с  теоремами \ref{theK0}, \ref{thK1-1} и \ref{theKh}, требованию $\varepsilon =0$
 как в задачах вида \eqref{fgiR}, так  и вида   \eqref{fgiRM} естественным образом  
возникает необходимость уже  различать случаи сходимости или расходимости 
  {\it логарифмических  интегралов\/} 
 \cite{Koosis88}, \cite{Koosis92}, \cite{Koosis96}, \cite[1.3]{BaiTalKha16},  \cite [4.6, (4.16)]{KhaShmAbd20} 
\begin{equation}\label{J2}
\int_{-\infty}^{+\infty}\frac{\ln^+ \bigl|g(iy)\bigr|}{1+y^2}\dd y ,  
\quad 
\int_{-\infty}^{+\infty}\frac{M^+ (iy)}{1+y^2}\dd y.   
\end{equation}
В случае конечности  первого  интеграла  в \eqref{J2}  целая функция $g$ экспоненциального типа называется 
{\it функцией класса Картрайт вдоль $i\RR$.\/}  При этом задача в версии  \eqref{fgiR} может быть легко вписана в 
знаменитые теоремы Бёрлинга\,--\,Маль\-я\-в\-е\-на о мультипликаторе и о радиусе полноты  1960-х гг. \cite{BM62}, \cite{BM67}, \cite{Kah62}, \cite{Red77}, \cite{Mal79}, \cite{Kra89}, \cite{Kha94}, \cite{HJ94}, \cite{MasNazHav05}, \cite[гл.~2]{Khsur} даже с ограничениями на 
величины типов целых функций. Возможна объединяющая 
формулировка теорем Бёрлинга\,--\,Маль\-я\-в\-е\-на, мотивированная формулировкой теоремы Рубела\,--\,Мальявена. 
\begin{theorem}\label{thBM1} При любом    $c\in \RR^+\setminus 0$ для любого 
 распределения  точек ${\mathrm Z}\subset \CC$  эквивалентны следующие четыре  утверждения: 
\begin{enumerate}[{\rm I.}]
\item \label{Ieg0BM} Для любой целой функции $g\not\equiv 0$ экспоненциального типа и любого $b\in \RR^+$
найдётся  целая функция   $f\not\equiv 0$ экспоненциального типа 
\begin{equation}\label{tfgc}
\type\bigl[\ln|f|\bigr]< \pi c+\type[\ln|g|],
\end{equation} 
обращающаяся в нуль на ${\mathrm Z}$, для которой  $\bigl| f(z)\bigr|\leq \bigl| g(z)\bigr|$ при всех $z\in \overline \strip_{b}$.  

\item\label{llR0BM} Внешняя плотность Редхеффера от ${\mathrm Z}$ вдоль $i\RR$  строго меньше, чем  $c$. 

\item\label{IIg0BM} Существует  ограниченная на мнимой оси $i\RR$ 
 целая функция  $f\not\equiv 0$  экспоненциального типа $\type\bigl[\ln|f|\bigr]<\pi c$, обращающаяся в нуль на ${\mathrm Z}$.   

\item\label{IVg0BM} Существует  ненулевая целая функция  $f\not\equiv 0$ класса Картрайт вдоль $i\RR$
и типа $\type\bigl[\ln|f|\bigr]<\pi c$, обращающаяся в нуль на ${\mathrm Z}$.  
\end{enumerate}
\end{theorem}
Элементарные рассуждения с обоснованием теоремы \ref{thBM1} приведены в конце \S~\ref{mainres}.

Решение задачи в версии  \eqref{fgiRM} 
для  субгармонических  функций $M\not\equiv -\infty$ конечного типа при конечности второго интеграла в \eqref{J2} с ещё более жёстким
ограничениями  $\type\bigl[\ln |f|\bigr]\leq \pi c+ \type[M]$  в  \eqref{tfgc} и $\type\bigl[\ln |f|\bigr]\leq \pi c$ в утверждении   \ref{IIg0BM}  
для типа целой функции $f\not\equiv 0$ экспоненциального типа   из   \eqref{fgiRM} может быть получено  из результатов наших совместных с  Т.\,Ю.~Байгускаровым, Г.\,Р.~Талиповой и  Ф.\,Б. Хабибуллиным статей \cite{KhaTalKha14}  и \cite{BaiTalKha16} 2014-16 гг.

Наиболее сильные на начало  1990-х гг. результаты  в версии  \eqref{fgiR}  для случая  {\it целой функции\/ $g\not\equiv 0$ экспоненциального типа \/}  с, вообще говоря,  {\it расходящимся\/} первым логарифмическим  интегралом в \eqref{J2}
были получены  при условии расположения  как  ${\mathrm Z}\subset \CC$, так и  
$ {\Zero_g} \subset \CC$ вне какой-нибудь {\it пары\/} соответственно  {\it открытых\/} или  {\it замкнутых вертикальных   углов\/}
\begin{equation}\label{Xe}
{\mathrm X}_a:=\bigl\{z\in \CC \bigm| |\Re z|<a |z|\bigr\},\; 
\overline {\mathrm X}_a:=\bigl\{z\in \CC \bigm| |\Re z|\leq a |z|\bigr\}, \;  a\in [0,1],
\end{equation}
 {\it с биссектрисой $i\RR$,\/} где  ${\mathrm X}_0=\varnothing$, $\overline {\mathrm X}_0=i\RR$,  ${\mathrm X}_1=\CC\setminus \RR$,
$\overline {\mathrm X}_1=\CC$, когда  {\it раствор\/} углов  $2\arcsin a$ углов \eqref{Xe} при $0<a\to 0$ может быть сколь угодно мал.
 \begin{theorem}[{(\cite[основная теорема]{Kha91AA})}]\label{thKhAM}
Пусть для распределения точек  ${\mathrm Z}$ на $\CC$ и целой  функции  $g\not\equiv 0$ экспоненциального типа существует $a\in (0,1)$, для которого ${\mathrm Z}\subset \CC\setminus \overline {\mathrm X}_a$ и $\Zero_g\subset \CC\setminus \overline {\mathrm X}_a$. Для существования целой функции   $f$ экспоненциального типа с свойствами из  \eqref{fgiR} необходимо и достаточно, чтобы в обозначении 
\eqref{{Kh89e}m} для некоторого  $C\in \RR$  выполнялись неравенства 
\begin{equation}\label{ZWRKh}
\ell_{\mathrm Z}(r,R)\leq \frac{1}{2\pi}\int_r^R\frac{\ln\bigl|g(iy)g(-iy)\bigr|}{y^2}\dd y +C
\quad\text{при всех $1\leq r<R<+\infty$.}
\end{equation}
\end{theorem}
Из последних результатов 2020--22 гг. можно отметить определённые дополнительные продвижения по задаче в версии  \eqref{fgiR}  исключительно в рамках целых функций экспоненциального типа  в наших  с А. Е. Салимовой работах \cite{SalKha20U}, \cite{SalKha20}, \cite{SalKha21}, \cite{Sal22}. Ряд  утверждений из этих работ с модификациями по существу используются ниже в доказательствах. 
Б\'ольшую часть основных результатов из этих работ мы не приводим, поскольку они, после некоторой подготовки и корректировки, могут быть получены как  частные случаи результатов настоящей статьи, что в некоторой мере  прокомментировано в замечаниях. 
Часть основных результатов настоящей статьи докладывалась на Международной конференции по комплексному анализу памяти 
А.~А.~Гончара и А.~Г.~Витушкина  в октябре 2021 г., а сам  доклад доступен в полном объёме по ссылке \cite{Kha21t}.

 \subsection{Специальные   версии  основных результатов}
В формулировках приведённых ниже двух очень частных случаях результатов из \S~\ref{mainres} и \S~\ref{mainresMR} приоритет отдаётся   простоте и краткости, а приводятся  они лишь для целых функций экспоненциального типа без участия субгармонических функций. Но уже и  эти облегчённые варианты в достаточной мере  иллюстрирует существенное усиление  
предшествующих результатов даже  в рамках постановки лишь задачи \eqref{fgiR},  не затрагивающей субгармонические версии 
из \eqref{fgiRM} и  \eqref{UM=}--\eqref{{fgiRMu}l}.

Следующий результат ---  развитие теоремы \ref{thKhAM} с формулировкой, близкой по форме к  теореме  Мальявена\,--\,Рубела и теореме \ref{theKh}.

\begin{theorem}\label{th1_5} Пусть   $g\not\equiv 0$ ---  целая функция экспоненциального типа  и для распределения точек ${\mathrm Z}$ 
существует число $a\in (0,1)$, для которого внешняя плотность Редхеффера от сужения ${\mathrm Z}{\lfloor}_{\overline {\mathrm X}_a}$
 вдоль $i\RR$ конечна.

Тогда эквивалентны следующие три утверждения: 
\begin{enumerate}[{\rm I.}]
\item \label{Ieg0}  
 При любом  $b\in \RR^+$ найдётся  целая функция   $f\not\equiv 0$ экспоненциального типа, для которой $f({\mathrm Z})=0$ и 
$\bigl|f(z)\bigr|\leq \bigl|g(z)\bigr|$  при всех  $z\in \overline \strip_b$. 
\item\label{II2g0} При  значениях $n$ и $N$, пробегающих   соответственно\/ $\NN_0$ и\/ $\NN$,  имеем 
\begin{equation}\label{ellZMg}
\limsup\limits_{N\to  \infty}\sup\limits_{0\leq n<N}
\biggl(\ell_{\mathrm Z}(2^n,2^N)-\frac{1}{2\pi}\int_{2^n}^{2^N}\frac{\ln \bigl|g(iy)g(-iy)\bigr|}{y^2}\dd y\biggr)<+\infty.
\end{equation}

\item\label{IIg0} Для некоторого числа  $p\in [0,1)$ существуют  целая функция   $f\not\equiv 0$ экспоненциального типа с $f({\mathrm Z})=0$ и ${\mathfrak m}_1$-измеримое  подмножество  $E\subset \RR^+$ с функцией $r\underset{r\in \RR^+}{\longmapsto} {\mathfrak m}_1\bigl(E\cap [0,r]\bigr)$  порядка меньше единицы, для которых 
\begin{equation}\label{ineqfgE}
\ln\bigl|f(iy)\bigr|\leq 
\frac{1}{2\pi}\int_0^{2\pi}\ln \bigl|g\bigl(iy+|y|^pe^{i\theta}\bigr)\bigr|\dd \theta
+|y|^p \quad\text{при  $|y|\in \RR^+\setminus E$}.
\end{equation}
\end{enumerate}
\end{theorem}

Следующий результат  --- прямое обобщение теоремы  Мальявена\,--\,Рубела.

\begin{theorem}\label{th1_5MR}  Пусть   ${\mathrm Z}\subset \CC$ распределение точек точно такое же, что и  в теореме\/  {\rm \ref{th1_5}},
а\/  ${\mathrm W}\subset \CC$ ---  распределение точек конечной верхней плотности и 
\begin{equation}\label{ellZMgMR}
\limsup\limits_{N\to  \infty}\sup\limits_{0\leq n<N}
\Bigl(\ell_{\mathrm W}(2^n,2^N)-\ell_{{\mathrm W}{\lfloor}_{\CC_{ \rh}}}(2^n,2^N)\Bigr)<+\infty,
\end{equation}
где ${\mathrm W}{\lfloor}_{\CC_{ \rh}}$ --- сужение ${\mathrm W}$ на $\CC_{ \rh}$.  
В частности, \eqref{ellZMgMR}, очевидно, выполнено,  если\/ ${\mathrm W}\subset \CC_{\overline\rh}$. 
Тогда эквивалентны  следующие три утверждения:  
\begin{enumerate}[{\rm I.}]
\item \label{Ieg0MR}  
Для любой целой функции $g\not\equiv 0$ экспоненциального типа с  $g({\mathrm W})=0$ 
при любом  $b\in \RR^+$ найдётся  целая функция   $f\not\equiv 0$  экспоненциального типа, 
обращающаяся в нуль на ${\mathrm Z}$, для которой 
$\bigl|f(z)\bigr|\leq \bigl|g(z)\bigr|$   при всех  $z\in \overline \strip_b$. 
\item\label{II2g0MR} При  значениях $n$ и $N$, пробегающих   соответственно\/ $\NN_0$ и\/ $\NN$,  имеем 
\begin{equation}\label{ellZMgMRl}
\limsup\limits_{N\to  \infty}\sup\limits_{0\leq n<N}
\Bigl(\ell_{\mathrm Z}(2^n,2^N)-\ell_{\mathrm W}(2^n,2^N)\Bigr)<+\infty.
\end{equation}

\item\label{IIg0MR} Для некоторого числа  $p\in [0,1)$ существуют  
целая функция $g\not\equiv 0$ экспоненциального типа с  $\Zero_g {\lfloor}_{\CC_{\rh}}={\mathrm W}{\lfloor}_{\CC_{\rh}}$, а также 
целая функция   $f\not\equiv 0$ экспоненциального типа с $f({\mathrm Z})=0$ и множество  $E\subset \RR^+$ такие же, как 
в утверждении\/ {\rm \ref{IIg0}} теоремы\/ {\rm \ref{th1_5}}, для которых имеет место  \eqref{ineqfgE}.
\end{enumerate}
\end{theorem}

Теорема \ref{th1_5}, выведенная  из основных результатов в конце \S~\ref{mainres}, теорема  
\ref{th1_5MR}, являющаяся частным случаем следствия \ref{corefZ} из \S~\ref{mainresMR}, 
и, тем более, основные результаты статьи из  \S~\ref{mainres} и \S~\ref{mainresMR} 
усиливают и обобщают   предшествующие результаты в нескольких различных направлениях. 

Во-первых, ранее во всех предшествующих результатах, включая и последние   \cite[основная теорема]{SalKha20}, \cite[теоремы 2.1 и 4.3]{SalKha21}, \cite[теоремы 2, A и B, замечания 1 и 2]{Sal22} 2020--22 гг., при требованиях вида \eqref{fgiR} накладывались различные  довольно жёсткие ограничения на распределение корней  $\Zero_g$ в паре углов  $\overline{\mathrm X}_a\overset{\eqref{Xe}}{\supset} i\RR$ при некотором $a>0$, в то времы как в теореме \ref{th1_5}   целая функция $g\not\equiv 0$ экспоненциального типа  {\it произвольная,\/} 
а в теореме \ref{th1_5MR} на распределение точек ${\mathrm W}$ наложено лишь  одно  условие \eqref{ellZMgMR},  
по которому в правой полуплоскости $\CC_{ \rh}$ точек из ${\mathrm W}$ в некотором смысле не меньше, чем в левой.
 
Во-вторых, ограничения на рост функции $|f|\leq |g|$ рассматривались ранее  только на $i\RR\overset{\eqref{{strip}c}}{=}\overline \strip_0$,  а в   теоремах  \ref{th1_5} и \ref{th1_5MR}  эти ограничения рассматриваются на полосе $\overline \strip_b$ сколь угодно большой  ширины  $2b\geq 0$.

В-третьих, ограничения на  распределение точек ${\mathrm Z}$ в  \eqref{Kh88e}, \eqref{{Kh89e}m}--\eqref{{Kh89e}g},  \eqref{ZWRKh} предполагают проверку неравенств по множеству мощности континуум всех интервалов $(r,R]\subset \RR^+$ с $r\geq 1$, в то время как в утверждении \ref{II2g0} теоремы \ref{th1_5} и утверждении \ref{II2g0MR} теоремы \ref{th1_5MR} требуется проверка единственного условия \eqref{ellZMg} и лишь по счётному множеству интервалов $(2^n,2^N]\subset \RR^+$, $0\leq n<N\in \NN$.  

Наконец, ослабление исходной посылки в импликациях  \ref{IIIMR}$\Rightarrow$\ref{IMR} теоремы Мальявена\,--\,Рубела, а также 
\ref{IIIKh}$\Rightarrow$\ref{IKh} теоремы \ref{theKh}, как и достаточного условия \eqref{ZWRKh} в теореме \ref{thKhAM}
с сохранением соответствующих эквивалентностей означает усиление этих теорем.  Именно ослабления исходной посылки в трёх направлениях предложены  в утверждении   \ref{IIg0} теоремы \ref{th1_5} и утверждении \ref{IIg0MR} теоремы \ref{th1_5MR}, где  неравенства $|f(iy)|\leq |g(iy)|$ для всех $y\in \RR$ заменены на  гораздо более слабые 
прологарифмированные неравенства  \eqref{ineqfgE} с интегральными средними по расширяющимся окружностям от $\ln |g|$ с аддитивной 
возрастающей добавкой в правой части вне довольно массивного исключительного множества $E$.   Значительно  более слабые версии 
этих утверждений  представлены в \S~\ref{mainres} в субгармонической версии   в утверждениях \ref{II_3} основной теоремы и  
 \ref{II} из теоремы \ref{th2_1}, а также в \S~\ref{mainresMR} в утверждениях \ref{II_3mu}  теоремы \ref{th_mu} и \ref{IIZ} из  теоремы \ref{th2_1Z}.

В статье не излагаются применения наших основных результатов к вопросам нетривиальности весовых пространств целых функций экспоненциального типа, полноты экспоненциальных систем в пространствах функций, к теоремам единственности для целых функций экспоненциального типа, к существованию му\-ль\-т\-и\-п\-л\-и\-к\-а\-т\-о\-р\-ов, гасящих рост целой функции вдоль прямой, к представлению мероморфных функций в виде отношения целых функций экспоненциального типа  с ограничениями на рост этих  функций  вдоль прямой, к аналитическому продолжению рядов, к задачам спектрального анализа-синтеза в пространствах голоморфных функций и пр. подобно тому, как это проделано или проиллюстрировано в  \cite{MR}, \cite{Kra72},
\cite{Kha88}, \cite{KhaD88},  \cite{Kha89}, \cite{Kha91AA}, \cite{Kha01l}, \cite{Kha09}, \cite{Khsur}, \cite{SalKha20}. 
Эти применения предполагается изложить отдельно в ином месте несколько  позже.

\section{Основные результаты}\label{mainres}

\subsection{Формулировка основной теоремы}\label{mainres1}
 Для  распределения зарядов $\nu$   
\begin{align}
\ell_{\nu}^{\rh}(r, R)&
:=\int_{r < | z|\leq R} \Re^+ \frac{1}{ z} \dd \nu(z)\in \RR, \quad 0< r < R < +\infty ,
\label{df:dDlm+}\\
\ell_{\nu}^{\lh}(r, R)&
:=\int_{r< |z|\leq R}\Re^- \frac{1}{z} \dd \nu(z)\in \RR,  \quad 0< r < R <  +\infty ,
\label{df:dDlm-}
\end{align}
---  соответственно {\it правая\/} и {\it левая логарифмические функции интервалов\/} $(r,R]$ на $\RR^+$.  В случае  {\it распределения масс\/} $\mu$ они порождают его 
{\it правую\/} $\ell_{\mu}^{\rh}$ и {\it левую $\ell_{\mu}^{\lh}$ логарифмические меры\/} на $\overline \RR^+\setminus 0$, допуская  и $R=+\infty$  в \eqref{df:dDlm+}--\eqref{df:dDlm-} с возможными  значениями $+\infty$ для $\ell_{\mu}^{\rh}(r, +\infty)$ 
и $\ell_{\mu}^{\lh}(r, +\infty)$,  а также его двустороннюю  {\it логарифмическую субмеру} на $\RR^+\setminus 0$, определённую как  
\begin{equation}\label{df:dDlLm}
\ell_{\mu}(r, R):=\max \bigl\{ \ell_{\mu}^{\lh}(r, R), \ell_{\mu}^{\rh}(r,R)\bigr\}\in \overline \RR^+, \quad 0< r < R \leq +\infty. 
\end{equation}
Для распределения точек ${\mathrm Z}$ это даёт в точности  $\ell_{\mathrm Z}(r, R)$ из  \eqref{{Kh89e}m}.

Развивая \cite[II]{Carleson}, \cite{Rodgers}, \cite{Federer},    \cite[гл. 2]{EG}, \cite{Eid07}, \cite[5.2]{VolEid13}
\cite[определение 3]{Kha22}, для $d\in  \RR^+$, функции  $r\colon \CC\to   \overline \RR^+\setminus 0$ и {\it гамма-функции\/} $\Gamma$ внешнюю   меру 
 \begin{equation}\label{mr}
{\mathfrak m}_d^r\colon S\underset{S\subset \CC}{\longmapsto}  
\inf \Biggl\{\sum_{k} \dfrac{\pi^{d/2}}{\Gamma (1+d/2)}r_k^d\biggm| S\subset \bigcup_{k} 
\overline D_{z_k}(r_k), \, z_k\in \CC, \, r_k \leq  r(z_k)\Biggr\}
\end{equation}
 называем {\it $d$-мерным обхватом Хаусдорфа переменного радиуса обхвата\/ $r$.\/} При этом  
через {\it постоянные функции\/} $r>0$ определяется {\it $d$-мерная    мера Хаусдорфа\/}
\begin{equation}\label{hH}
{\mathfrak m}_d\colon S\underset{S\subset \CC}{\longmapsto}  \lim_{0<r\to 0} {\mathfrak m}_d^r(S)
\underset{r>0}{\geq} {\mathfrak m}_d^r(S)\geq {\mathfrak m}_d^\infty(S),
\end{equation}
являющаяся регулярной  мерой Бореля. По определению \eqref{mr}, очевидно, 
\begin{equation}\label{mts}
{\mathfrak m}_d\geq {\mathfrak m}_d^{r}\geq {\mathfrak m}_d^{t}\geq {\mathfrak m}_d^{\infty}
\quad\text{для любых пар  функций $r\leq t$.} 
\end{equation}

В частности,   ${\mathfrak m}_2$ ---  это {\it плоская мера  Лебега на\/} $\CC$,  а для любой  липшицевой кривой $L$ в $\CC$  сужение   ${\mathfrak m}_1{\lfloor}_L$ --- это мера длины дуги на липшицевой 
кривой $L$ \cite[3.3.4A]{EG}, что вполне согласовано с предшествующим обозначением  ${\mathfrak m}_1$ для линейной меры  Лебега на $\RR$.
Кроме того,  $0$-мерная   мера Хаусдорфа ${\mathfrak m}_0$ множества -- это число элементов в нём, а также 
${\mathfrak m}_d={\mathfrak m}_d^r=0$ при любом    $d>2$.

Как и  в \cite[предисловие]{EG}, расширенная числовая функция {\it интегрируема\/} по мере  Радона на множестве, если интеграл от неё по этой мере  корректно определён значением из $\overline \RR$. Интегрируемая функция {\it суммируема,\/}  если соответствующий  интеграл от неё конечен, т.е. принимает значения из $\RR$.

Для функций  $r\colon S\to \RR^+$  на подмножестве $S\subset \CC$ будут использованы  {\it интегральные средние 
 с переменным радиусом $r$  по  окружностям}
\begin{equation}
v^{\circ r}\colon z\underset{z\in S}{\longmapsto}\frac{1}{2\pi} \int_{0}^{2\pi}  v\bigl(z+r(z)e^{i\theta}\bigr) \dd \theta 
\label{vpC}
\end{equation}
для  ${\mathfrak m}_1$-интегрируемых  функций  $v$  на окружностях $\partial D_z\bigl(r(z)\bigr)$ при $z\in S$, а также 
{\it интегральные средние  с переменным радиусом $r$   по кругам\/}
\begin{equation}
v^{\bullet r}\colon z\underset{z\in S}{\longmapsto}
\frac{1}{\pi (r(z))^2} \int_{\overline  D_z(r(z))}  v \dd \mathfrak m_2  
\label{vpD}
\end{equation}
для $\mathfrak m_2$-интегрируемых  функций $v$ на кругах $\overline  D_z\bigl(r(z)\bigr)$ при $z\in S$. 
Для любой функции $u$, субгармонической  на открытой окрестности объединения кругов 
\begin{equation}\label{Sdrcup}
 S^{\cup r}:=\bigcup_{z\in S}\overline  D_z\bigl(r(z)\bigr)\subset \mathbb C 
\end{equation}
из  \eqref{vpD}, имеем неравенства \cite[теорема 2.6.8]{Rans}
\begin{equation}\label{vbc}
u(z)\leq u^{\bullet r}(z)\leq u^{\circ r}(z) \leq \sup_{D_z(r(z))}u
\quad\text{при всех }z\in S^{\cup r}.
\end{equation}
Для $\mathfrak m_1$-измеримого  подмножества $E\subset \RR^+$ будет использована характеризующая  его 
 {\it возрастающая функция\/} \cite[лемма 1]{SalKha20U}
\begin{equation}\label{qE}
q_E\colon r\underset{r\in \RR^+}{\longrightarrow} \mathfrak m_1 \bigl(E\cap [0,r]\bigr)\ln \frac{er}{\mathfrak m_1 \bigl(E\cap [0,r]\bigr)}
\underset{r\in \RR^+}{\leq} r.
\end{equation}
Задачи \eqref{UM=} и \eqref{UMbbul}--\eqref{{fgiRMu}l} решает следующий результат.
\begin{maintheorem}\label{th2main} 
Пусть   $\nu$ --- распределение масс на $\CC$, для которого
\begin{equation}\label{nubstr-}
s:= \inf_{z\in \supp \nu}|\Re z|>0, 
\qquad \liminf_{\stackrel{z\to \infty}{z\in \supp \nu}}\frac{|\Re z|}{|z|}>0. 
\end{equation}
Тогда для любой субгармоническая функции $M$ конечного  типа 
следующие пять  утверждений\/ {\rm \ref{I}--\ref{II_3}} эквивалентны:  
\begin{enumerate}[{\rm I.}]
\item\label{I}   Для  любого  $b\in [0,s)$ существует субгармоническая функция $U\not\equiv -\infty$ конечного типа 
с  распределением масс Рисса 
$\frac{1}{2\pi}{\bigtriangleup}U\geq \nu$, для которой 
\begin{equation}\label{UeqM}
U(z)\equiv M(z)\quad\text{при всех $z\in \overline\strip_b$}.
\end{equation}

\item\label{II_2} При  значениях $n$ и $N$, пробегающих   соответственно\/ $\NN_0$ и\/ $\NN$,  имеем 
\begin{equation}\label{lJ2}
\limsup\limits_{N\to  \infty}\sup\limits_{0\leq n<N}
\biggl(\ell_{\nu}(2^n,2^N)-\frac{1}{2\pi}\int_{2^n}^{2^N}\frac{M(iy)M(-iy)}{y^2}\dd y\biggr)<+\infty.
\end{equation}

\item\label{I_2}  
Для  любой пары субгармонических функций $v$ и $m$ с распределениями масс Рисса соответственно 
$\frac{1}{2\pi}{\bigtriangleup} v=\nu$ и  $\frac{1}{2\pi}{\bigtriangleup}m=\frac{1}{2\pi}{\bigtriangleup}M{\lfloor}_{\strip_{s}}$
при   каждом $b\in [0,s)$ найдётся целая функция $h\not\equiv 0$, с которой  сумма\/ $v+m+\ln |h|$ ---  субгармоническая функция
  конечного типа и 
\begin{equation}\label{umM}
v(z)+m(z)+\ln\bigl|h(z)\bigr|\leq M(z) \quad\text{\it при  всех $z\in \overline\strip_b$.} 
\end{equation} 

\item\label{I_3} 
Для  произвольной субгармонической функции   $v$ с распределением масс Рисса $\frac{1}{2\pi}{\bigtriangleup} v=\nu$ при любых  $b\in [0,s)$, $d\in (0,2]$ и функции $r\colon \CC\to (0,1]$ с 
\begin{equation}\label{qr}
\inf_{z\in \CC} \frac{\ln r(z)}{\ln(2+ |z|)}>-\infty
\end{equation}
найдутся целая функция $h\not\equiv 0$ и подмножество $E_b\subset \CC$, для которых   
$v+\ln |h|$ ---  субгармоническая функция  конечного типа  и 
\begin{align}
 v(z)+\ln\bigl|h(z)\bigr|&\leq M^{\bullet r}(z) \quad\text{при всех $z\in \overline\strip_b$, а также}
\label{ubullet}
\\
v(z)+\ln\bigl|h(z)\bigr|&\leq M(z)\quad\text{при всех $z\in \overline\strip_b\setminus E_b$, где}
\label{ubull1}
\\
{\mathfrak m}_d^r(E_b\cap S)&\leq \sup_{z\in S}r(z)\quad\text{для любого $S\subset \CC$}.
\label{ubull}
\end{align}

\item\label{II_3} Существуют\/ ${\mathfrak m}_1$-измеримая функция $q_0\colon \RR\to \overline \RR^+$, 
непрерывная функция $q\colon \RR\to \RR^+$  конечного типа с  убывающей при некоторых $P\in \RR^+$ и $t_0\in \RR^+$   
функцией  $t\underset{t\geq t_0}{\longmapsto} \dfrac{q(t)+q(-t)}{t^P}$, 
субгармоническая  функция  $U\not\equiv -\infty$ конечного типа
с распределением масс Рисса\/  $\frac{1}{2\pi}{\bigtriangleup}U\geq \nu$ и\/ $\mathfrak m_1$-из\-м\-е\-р\-и\-мое множество  $E\subset \RR^+$ с функцией $q_E$ из \eqref{qE}, для которых  одновременно 
\begin{gather}
U(iy)+U(-iy)\underset{y\in \RR^+\setminus E}{\leq} M^{\circ q}(iy)+M^{\circ q}(-iy)+q_0(y)+q_0(-y),
\label{{UM0}M}\\
\int_1^{+\infty}\bigl(q_0(t)+q_0(-t)+q(t)+q(-t)+q_E(t)\bigr)\frac{\dd t}{t^2}<+\infty.
\label{{UM0}E}
\end{gather}
\end{enumerate}
\end{maintheorem}

\begin{remark}\label{rem0}
Условия на $\supp \nu$ из \eqref{nubstr-} эквивалентны одному условию:  
\begin{enumerate}
\item[{\rm [$\upnu$]}] {\it Существуют числа $s\in \RR^+$ и  $a\in (0,1)$, для которых 
\begin{equation}\label{nubstr}
  \supp \nu\overset{\eqref{Xe}}{\subset} \CC\setminus \bigl(\overline {\mathrm X}_a\cup \overline \strip_b\bigr)
\quad\text{для любого  $b\in [0,s)$}.
\end{equation}
}
\end{enumerate}
\end{remark}

\begin{example}\label{ex2_1} Степенные функции $q_0(y)\equiv q(y)\equiv |y|^p$ степени  $p\in [0,1)$
обладают  всеми свойствами, требуемыми в утверждении\/ {\rm \ref{II_3}.}  
Если при этом ещё и порядок функции $r\underset{r\in \RR^+}{\longmapsto} {\mathfrak m}_1\bigl(E\cap [0,r]\bigr)$
меньше единицы, то найдутся числа $C\geq 1$, $p_E\in [0,1)$ и $r_E\geq 1$,   для которых ${\mathfrak m}_1\bigl(E\cap [0,r]\bigr)\leq Cr^{p_E}\leq r$
при  всех $r\geq r_E$. Отсюда в силу возрастания  функции  $x\underset{x\in [0,r]}{\longmapsto} x\ln\frac{er}{x}$  имеем
\begin{equation*}
q_E(r)\overset{\eqref{qE}}{\leq} Cr^{p_E}\ln \frac{er}{Cr^{p_E}}
\leq Cr^{p_E}\ln er
\quad\text{при всех $r\geq r_E$.}
\end{equation*} 
Очевидно, для таких функций $q_0$, $q$ и $q_E$  выполнено \eqref{{UM0}E}. 
\end{example}

\begin{remark}\label{rem3}
Условие   \eqref{qr} для  $r\colon \CC\to (0,1]$ эквивалентно существованию достаточно малого   $c\in \RR^+\setminus 0$ и больших   
$R\in \RR^+$ и $P\in \RR^+$, для которых 
\begin{equation}\label{nubstrr}
 r(z)\geq\begin{cases}
 c>0&\text{ при всех $z\in  R\overline \DD$},\\ 
\dfrac{1}{(1+|z|)^P}&\text{ при всех $z\in \CC\setminus R\overline \DD$}.
\end{cases}
\end{equation}
В частности, при такой  функции $r$  для исключительного множества $E_b\subset \CC$ из утверждения 
\ref{I_3} соотношение \eqref{ubull} при выборе $S:=\CC\setminus t\overline \DD$ 
даёт 
$$
{\mathfrak m}_d^r(E_b\setminus t\overline \DD) \overset{\eqref{mts}}{=}O\Bigl(\frac{1}{t^P}\Bigr)\quad\text{при $t\to +\infty$}.
$$
Отсюда, например,   при выборе $d:=1$ для  любого сколь угодно большого $P>0$  исключительное множество $E_b$ 
в утверждении \ref{I_3} может быть выбрано так, что его часть, лежащая на вертикальных прямых  
$x+i\RR$ с  $x\in [-b,b]$,  не только конечной линейной меры Лебега ${\mathfrak m}_1$,  но и быстро уменьшается по мере ${\mathfrak m}_1$ со степенной  скоростью $t^{-P}$ вне отрезков $x+i[-t,t]$ при $t\to +\infty$.  
\end{remark}

\begin{remark}\label{remmt}
Доказательство основной теоремы  проводится по схеме 
\begin{equation}\label{impl}
{\rm\ref{I_3}} \Rightarrow    {\rm\ref{II_3}}\Rightarrow  {\rm \ref{II_2}}\Rightarrow {\rm\ref{I}} \Rightarrow {\rm\ref{I_2}} \Rightarrow{\rm\ref{I_3}}, 
\end{equation}
где в  доказательствах импликаций, предшествующих  {\rm\ref{II_2}} в \eqref{impl}, и в доказательстве  последней импликации \ref{I_2}$\Rightarrow$\ref{I_3}  не будут использоваться  ограничения \eqref{nubstr-} на 
распределение масс $\nu$. По  замечанию \ref{rem2_3}  импликации \ref{I}$\Rightarrow$\ref{II_3} и  \ref{I_2}$\Rightarrow$\ref{II_3} также верны без ограничений на $\nu$. Следовательно, импликации \ref{I_3}$\Rightarrow$\ref{II_3}$\Rightarrow$\ref{II_2}, 
\ref{I}$\Rightarrow$\ref{II_3} и  \ref{I_2}$\Rightarrow$\ref{II_3}  истинны для произвольного распределением масс $\nu$.   
Доказательства же импликаций \ref{II_2}$\Rightarrow$\ref{I}$\Rightarrow$\ref{I_2} использует 
ограничения \eqref{nubstr-} или эквивалентное им условие [$\upnu$] с \eqref{nubstr} из замечания \ref{rem0}  на распределение масс $\nu$. 
В то же время контрпример из \cite[пример 3.3.2]{Khsur} при  некотором  его  развитии   показывает, что 
импликации \ref{II_2}$\Rightarrow$\ref{I}  и \ref{II_2}$\Rightarrow$\ref{I_2} могут  быть неверны 
и при значительно более слабом требовании-неравенстве вида $\leq M(iy)+o\bigl(|y|\bigr)$ лишь на $i\RR$ при $|y|\to +\infty$, пропуская довольно массивное неограниченное исключительное подмножество  $E\subset \RR^+$, если не накладывать каких бы то ни было дополнительных ограничений на распределение  масс $\nu$  вблизи $i\RR$. Это позволяет с уверенностью сказать, что в рамках  логарифмической субмеры $\ell_{\nu}$ для распределения масс $\nu$ из  утверждения \ref{II_2} основная теорема даёт исчерпывающее решение задач из п.~\ref{Sspz}. 
Возможные подходы  к  снятию ограничений   \eqref{nubstr-} 
или эквивалентного им условия  [$\upnu$] с \eqref{nubstr} из замечания \ref{rem0}  на распределение масс $\nu$
 --- это  использование дополнительных характеристик распределений масс или точек, что обсуждается в заключительном комментарии к статье. 
 \end{remark}

\subsection{Доказательства импликаций, предшествующих  {\rm\ref{II_2}} в \eqref{impl}}

\begin{proof}[импликации \ref{I_3}$\Rightarrow$\ref{II_3}]  По утверждению {\rm\ref{I_3}}  при  $b:=0$ существует субгармоническая  функция    $U:=v+\ln|h|\not\equiv -\infty$ конечного типа в \eqref{ubull1}--\eqref{ubull} с распределением масс Рисса $\frac{1}{2\pi}{\bigtriangleup}U\geq \nu$, удовлетворяющая неравенству $U(iy)\leq M(iy)$  для всех $iy\in (\CC\setminus E_0)\cap i\RR$, где при $d:=1$ и  выборе $r\equiv 1$ в \eqref{qr} имеем ${\mathfrak m}_1^1(E_0) <+\infty$. Последнее по определению обхвата Хаусдорфа  \eqref{mr} радиуса $1$ означает, что некоторое объединение отрезков на $i\RR$ с конечной суммой длин покрывает 
$iE':=E_0\cap i\RR$ и  по утверждению {\rm\ref{I_3}}
\begin{equation}\label{UMER}
U(iy)\leq M(iy)\quad\text{для всех $y\in \RR\setminus E'$, где ${\mathfrak m}_1(E') <+\infty$}.
\end{equation} 
Отсюда    для  $E:=\bigl(E'\cup (-E')\bigr)\cap \RR^+$ по-прежнему 
${\mathfrak m}_1(E) <+\infty$ и 
\begin{equation}\label{UME}
U(iy)+U(-iy)\leq M(iy)+M(-iy)\quad\text{для всех $y\in \RR^+\setminus E$.}
\end{equation} 
 Для такого подмножества $E\subset \RR^+$ в силу возрастания  функции  $x\underset{x\in [0,r]}{\longmapsto} x\ln\frac{er}{x}$ 
по определению функции $q_E$ из \eqref{qE} при $r\geq {\mathfrak m}_1(E)$ следует 
\begin{equation}\label{qEr1}
q_{E}(r) \overset{\eqref{qE}}{\leq}
{\mathfrak m}_1(E)\ln\frac{er}{{\mathfrak m}_1(E)}\underset{r\to +\infty}{=}O(\ln r).
\end{equation}
Отсюда  при выборе $q_0=q=0$ получаем как конечность интеграла из \eqref{{UM0}E}, так и 
согласно \eqref{UME} неравенства \eqref{{UM0}M} при всех $y\in \RR^+\setminus E$. 
\end{proof}

\begin{remark}\label{rem2_3}  Доказательства импликаций  {\rm\ref{I}}$\Rightarrow${\ref{II_3}} и {\rm\ref{I_2}}$\Rightarrow${\ref{II_3}} ещё проще, поскольку выбор функции $U$ в первом случае тот же, что и в утверждении {\rm\ref{I}}, во втором  можем положить $U:=v+m+\ln|h|$, $E:=\varnothing$ для  обоих случаев, а рассуждения, начиная с \eqref{UME} те же, что и при доказательстве \ref{I_3}$\Rightarrow$\ref{II_3}. Кроме того, основную теорему 
можно дополнить ещё одним, эквивалентным  утверждениям  \ref{I}--\ref{II_3}, но  значительно более простым и кратким по сравнению с \ref{II_3}   утверждением 
\begin{enumerate}
\item[{\rm {VI.}}]\label{VIa} {\it Существуют субгармоническая  функция  $U\not\equiv -\infty$ конечного типа с\/  $\frac{1}{2\pi}{\bigtriangleup}U\geq \nu$ и\/ $\mathfrak m_1$-из\-м\-е\-р\-и\-мое подмножество  $E'\subset \RR$, для которых  
\begin{equation}\label{UME0}
U(iy)\leq M(iy)\quad\text{при всех $y\in \RR\setminus E'$, где ${\mathfrak m}_1(E')<+\infty$}. 
\end{equation}
}
\end{enumerate}
Действительно, импликация  \ref{I_3}$\Rightarrow${VI} доказывается дословно так же, как 
импликация  \ref{I_3}$\Rightarrow$\ref{II_3} до \eqref{UMER} включительно, 
а импликация  {VI}$\Rightarrow$\ref{II_3} так же, как  
импликация  \ref{I_3}$\Rightarrow$\ref{II_3} от \eqref{UMER} до конца доказательства. 
\end{remark}

\begin{proof}[импликации {\ref{II_3}}$\Rightarrow${\rm\ref{II_2}}]
Утверждение \ref{II_3} до \eqref{{UM0}M} включительно  --- это в точности посылка основного результата из \cite[теорема 1]{SalKha20U}  с той лишь разницей, что не предполагается убывание функции  
\begin{equation}\label{ubP}
t\underset{t\geq t_0}{\longmapsto} \frac{q(t)+q(-t)}{t^P}
\quad\text{при некоторых $P\in \RR^+$ и $t_0\in \RR^+$.}
\end{equation}
Даже по существенно ослабленному заключению из \cite[теорема 1, формулы (1.10)--(1.11)]{SalKha20U} 
для  любых  чисел $r_0>0$ и $N\in \RR^+$ найдётся $C\in \RR^+$, для которого 
\begin{multline}\label{rRuM}
\ell_{\nu}(r,R)\leq \frac{1}{2\pi}\int_{r}^{R}\frac{M(iy)M(-iy)}{y^2}\dd y
+C\frac{1}{2\pi}\int_{r}^{R}\frac{q_0(t)+q_0(-t)+2q_E(t)}{t^2}\dd t
 \\+C \int_r^R t^N\sup_{s\geq t} \frac{q(s)+q(-s)}{s^{2+N}}\dd t+C
\quad\text{при всех $r_0\leq r<R<+\infty$.}
\end{multline}
При условии убывания функции \eqref{ubP} с $P:=2+N$ при некотором  достаточно большом $t_0\in \RR^+$
 найдутся \cite[предложение 1]{SalKha20U} числа $r_1> 0$ и $C_1\in \RR^+$, для которых  
последний интеграл в \eqref{rRuM} оценивается сверху как 
\begin{equation*}
\int_r^R t^N\sup_{s\geq t} \frac{q(s)+q(-s)}{s^{2+N}}\dd t\leq C_1\int_r^R \frac{q(t)+q(-t)}{t^2}\dd t
\text{ при всех $r_1\leq r<R<+\infty$}.  
\end{equation*}
Последняя оценка вместе с  \eqref{rRuM} даёт \cite[предложение 1, формула (3.4)]{SalKha20U}
\begin{multline*}
\ell_{\nu}(r,R)\leq \frac{1}{2\pi}\int_{r}^{R}\frac{M(iy)M(-iy)}{y^2}\dd y
\\
+C_2\frac{1}{2\pi}\int_{r}^{R}\frac{q_0(t)+q_0(-t)+2q_E(t)+q(t)+q(-t)}{t^2}\dd t
+C_2
\end{multline*}
для некоторого $C_2\in \RR^+$ при всех $r_2:=\max\{1, r_0, r_1\}\leq r<R<+\infty$, откуда согласно 
 \eqref{{UM0}E} для некоторого $C_3\in \RR$ получаем 
\cite[следствие 1]{SalKha20U}
\begin{equation*}
\ell_{\nu}(r,R)\leq \frac{1}{2\pi}\int_{r}^{R}\frac{M(iy)M(-iy)}{y^2}\dd y
+C_3\quad\text{при всех $r_2\leq r<R<+\infty$}.  
\end{equation*}
Здесь  ввиду локальной ${\mathfrak m}_1$-суммируемости  субгармонических  функций на каждой прямой и неравенства 
субаддитивности    $\ell_{\nu}(r,R)\overset{\eqref{df:dDlLm}}{\leq} \ell_{\nu}(1,r_2)+\ell_{\nu}(r_2,R)$ при $r\in [1,r_2)$
\cite[3.2]{SalKha20U} фиксированное значение $r_2$ можно заменить на $1$, возможно увеличивая $C_3$.
После этого, перекидывая интеграл из правой части с противоположным знаком в левую  часть неравенства и применяя к полученной 
разности слева  сначала операцию $\sup\limits_{r\in [1,R)}$, а затем $\limsup\limits_{R\to +\infty}$, получаем 
\begin{equation}\label{lJ2r}
\limsup\limits_{R\to +\infty} \sup\limits_{1\leq r<R}\biggl(\ell_{\nu}(r,R)- \frac{1}{2\pi}\int_{r}^{R}\frac{M(iy)M(-iy)}{y^2}\dd y\biggr)<+\infty.
\end{equation}
Здесь левая часть  не меньше левой части  \eqref{lJ2}, что даёт утверждение \ref{II_2}.
\end{proof}

Доказательство ключевой в основной теореме импликации \ref{II_2}$\Rightarrow$\ref{I} 
потребует  предварительной подготовки 
и будет дано в \S~\ref{S7}. Доказательства импликаций  {\rm\ref{I}}$\Rightarrow$\ref{I_2}$\Rightarrow${\ref{I_3}}, завершающих доказательство основной теоремы,   приводятся  в \S~\ref{S9}.

\subsection{Случай  распределения точек  $\mathrm Z$ в роли распределения масс}\label{Ss2_2} 
В следующей теореме \ref{th2_1} рассматривается задача \eqref{fgiRM}, когда в роли распределения масс $\nu$  из основной теоремы выступает распределение точек ${\mathrm Z}\subset \CC$. 

\begin{theorem}\label{th2_1} Пусть    $M\not\equiv -\infty$ --- субгармоническая функция   конечного типа,  а     
${\mathrm Z}\subset \CC$ --- распределение точек, для которого при некотором $a\in (0,1)$ 
внешняя плотность Редхеффера от сужения  ${\mathrm Z}{\lfloor}_{\overline{\mathrm X}_a}$  вдоль $i\RR$ конечна.  

Тогда следующие четыре  утверждения\/ {\rm \ref{Ie}--\ref{II}} эквивалентны:
\begin{enumerate}[{\rm I.}]
\item \label{Ie}  
Для любых  $0\leq b<s\in \RR^+$  и субгармонической функции $m$ с распределением масс Рисса 
$\frac{1}{2\pi}{\bigtriangleup}m=\frac{1}{2\pi}{\bigtriangleup}M{\lfloor}_{\strip_s}$
найдётся  целая функция $f\not\equiv 0$, обращающаяся в нуль на   ${\mathrm Z}$,  для которой  субгармоническая функция 
$\ln |f|+m$ конечного типа и выполняются неравенства 
\begin{equation}\label{umMe}
\ln \bigl|f(z)\bigr|+m(z)\leq M(z) \quad\text{\it при  всех $z\in \overline \strip_b$.} 
\end{equation} 

\item\label{II2} При  значениях $n$ и $N$, пробегающих   соответственно\/ $\NN_0$ и\/ $\NN$,  имеем 
\begin{equation}\label{ellZM}
\limsup\limits_{N\to  \infty}\sup\limits_{0\leq n<N}
\biggl(\ell_{\mathrm Z}(2^n,2^N)-\frac{1}{2\pi}\int_{2^n}^{2^N}\frac{M(iy)M(-iy)}{y^2}\dd y\biggr)<+\infty.
\end{equation}

\item \label{IIe}  
 При любых  $b\in \RR^+$, $d\in (0,2]$ и функции $r\colon \CC\to (0,1]$ с ограничением \eqref{qr}
найдутся целая функция $f\not\equiv 0$  экспоненциального типа  с  $f({\mathrm Z})=0$ и   $E_b\subset \CC$, для которых 
 $\ln\bigl|f(z)\bigr|\leq M^{\bullet r}(z)$ 
{при всех $z\in \overline\strip_b$, а также}
 $\ln\bigl|f(z)\bigr|\leq M(z)$ при всех $z\in \overline\strip_b\setminus E_b$, где для $E_b$
выполнено \eqref{ubull}.

\item\label{II} 
Для целочисленного распределения масс $\nu:={\mathrm Z}$ 
выполнено утверждение\/ {\rm \ref{II_3}}   основной теоремы  с соотношениями \eqref{{UM0}M} и \eqref{{UM0}E}. 
\end{enumerate}
 \end{theorem}
\begin{proof} 
Для произвольной субгармонической фу\-н\-к\-ц\-ии $v$ с целочисленным распределением масс $\nu:={\mathrm Z}$  по одной из классических теорем Вейерштрасса найдётся  целая функция $f_{\mathrm Z}\not\equiv 0$, для которой  
\begin{equation}\label{vlnf}
\Zero_{f_{\mathrm Z}}={\mathrm Z}, \quad  v=\ln |f_{\mathrm Z}|.
\end{equation}
Доказательство теоремы \ref{th2_1}, которое проводим по схеме \ref{IIe}$\Rightarrow$\ref{II}$\Rightarrow$\ref{II2}$\Rightarrow$\ref{Ie}$\Rightarrow$\ref{IIe}, будет полностью основываться на основной теореме.   

Пусть выполнено утверждение \ref{IIe}.  Для целой  функции $f\not\equiv 0$ с  $f({\mathrm Z})=0$  из утверждения \ref{IIe} имеет место представление 
$f\overset{\eqref{vlnf}}{=}f_{\mathrm Z}h$. Отсюда по равенствам  $\ln |f|=\ln |f_{\mathrm Z}|+\ln|h|\overset{\eqref{vlnf}}{=}v+\ln|h|$ утверждение  \ref{IIe} даёт  утверждение 
\ref{I_3} основной теоремы.  Из импликации \ref{I_3}$\Rightarrow$\ref{II_3} основной теоремы, 
истинной по замечанию \ref{remmt} без ограничений на $\nu$,  следует, что истинна 
импликация \ref{IIe}$\Rightarrow$\ref{II}.

Пусть выполнено  утверждение  \ref{II}.    
По замечанию \ref{remmt} из импликации  \ref{II_3}$\Rightarrow$\ref{II_2} основной теоремы следует,
что для $\nu:={\mathrm Z}$ выполнено \eqref{lJ2}, что даёт в точности 
\eqref{ellZM} для ${\mathrm Z}=\nu$. Это доказывает истинность импликации \ref{II}$\Rightarrow$\ref{II2}.

Для доказательства импликации \ref{II2}$\Rightarrow$\ref{Ie} потребуется 

\begin{lemma}\label{lem2_1} 
Пусть   ${\mathrm Z}\subset \CC$ --- какое-либо распределение  точек 
внешней  плотности Редхеффера вдоль $i\RR$ строго меньшей, чем  $c\in \RR^+\setminus 0$. 

Тогда для любого $s\in \RR^+$ найдётся целая функции $f_s\not\equiv 0$ экспоненциального типа $\type\bigl[\ln |f_s|\bigr]<\pi c$, для которой 
$f_s({\mathrm Z})=0$ и $\bigl|f_s(z)\bigr|\leq 1$ при всех $z\in \overline \strip_{s}$.
\end{lemma}
\begin{proof}[леммы \ref{lem2_1}]
Теорема Бёрлинга\,--\,Мальявена о радиусе полноты  \cite{BM67}  в версии Р. М. Редхеффера \cite{Red77}, \cite[теорема 2.1.10]{Khsur}, после её двойственной переформулировки  в терминах целых функций экспоненциального типа и  перехода с помощью поворота на прямой угол  от $\mathbb{R}$ к $i\mathbb{R}$, обеспечивает существование целой функции $h_s\not\equiv 0$ экспоненциального типа 
\begin{equation}\label{fs1}
\type\bigl[\ln|h_s|\bigr]<\pi c, \quad   h_s({\mathrm Z})=0, \quad \sup\limits_{y\in \mathbb{R}}\bigl|h_s(iy)\bigr|\leq 1.
\end{equation}
Для ограниченной на мнимой оси $i\RR$ субгармонической функции $\ln|h_s|$ конечного типа 
интеграл Пуассона от  функции $\ln|h_s|$ по мнимой оси  для некоторого числа $A\in \RR^+$ задаёт гармонические мажоранты на $\CC_{\rh}$
и на $\CC_{\lh}$ 
\begin{equation*}
z\underset{z\in \CC\setminus i\RR}{\longmapsto} \frac{1}{\pi}\int_{-\infty}^{+\infty} \frac{|\Re z|\ln\bigl| h_s(iy)\bigr|}{(\Re z)^2+(\Im z-y)^2}
\dd y+A|\Re z|\overset{\eqref{fs1}}{\underset{z\in \CC\setminus i\RR}{\leq}} A|\Re z|.
\end{equation*}
Таким образом, выполнены  неравенства 
\begin{equation}\label{fastb}
\ln \bigl|h_s(z)\bigr|\underset{z\in \CC}{\leq} A |\Re z|\leq As\quad\text{при всех $z\in \overline\strip_s$.}
\end{equation}
Тогда  целая функция $f_s:=h_se^{-As}$ экспоненциального типа 
$$
\type\bigl[\ln|f_s|\bigr]=\type\bigl[\ln|h_s|\bigr]\overset{\eqref{fs1}}{<}\pi c,
$$
согласно \eqref{fs1}  обращается в нуль на ${\mathrm Z}$ и  
  $$
\ln\bigl|f_s(z)\bigr|\leq \ln\bigl|h_s(z)\bigr| -As\overset{\eqref{fastb}}{\leq} As-As=0\quad\text{при всех $z\in \overline \strip_{s}$},
$$
что завершает доказательство  леммы   \ref{lem2_1}. 
\end{proof}

Из  утверждения  \ref{II2} при $0\leq b<s\in \RR^+$ выведем утверждение  \ref{Ie}.   

Сужение ${\mathrm Z}{\lfloor}_{{\overline{\mathrm X}_a}\cup \overline\strip_ s}$, добавляет к сужению  ${\mathrm Z}{\lfloor}_{\overline{\mathrm X}_a}$  не более чем конечное число точек. Поэтому, очевидно, 
и распределение точек 
${\mathrm Z}{\lfloor}_{{\overline{\mathrm X}_a}\cup \overline\strip_ s}$ также конечной внешней плотности Редхеффера вдоль $i\RR$. 
По лемме \ref{lem2_1}  с ${\mathrm Z}{\lfloor}_{{\overline{\mathrm X}_a}\cup \overline\strip_ s}$ в роли ${\mathrm Z}$ 
существует  целая  функции $f_s$ экспоненциального типа, для которой 
\begin{equation}\label{faZa}
f_s({\mathrm Z}{\lfloor}_{{\overline{\mathrm X}_a}\cup \overline\strip_ s})=0, \quad 
\bigl|f_s(z)\bigr|\leq 1\quad\text{при всех $z\in \overline \strip_{s}$.}
\end{equation}
Наряду с распределением точек ${\mathrm Z}{\lfloor}_{{\overline{\mathrm X}_a}\cup \overline\strip_ s}$
рассмотрим оставшуюся часть распределения точек $\mathrm Z$,
для которой, при трактовке  её как целочисленной меры 
$\nu:={\mathrm Z}{\lfloor}_{\CC\setminus ({\overline{\mathrm X}_a}\cup \overline\strip_ s)}\leq {\mathrm Z}$,
очевидно, выполнено  условие {\rm [$\upnu$]} с  \eqref{nubstr} из замечания \ref{rem0}, эквивалентное  
условиям \eqref{nubstr-} основной теоремы, а из \eqref{ellZM} следует \eqref{lJ2}. 
Для  целочисленного  распределения масс $\nu$ существует целая функция $F\not\equiv 0$ с распределением корней 
$\Zero_F={\mathrm Z}{\lfloor}_{\CC\setminus ({\overline{\mathrm X}_a}\cup \overline\strip_ s)}\leq {\mathrm Z}$. Другими словами, функция  
$v:=\ln|F|$ субгармоническая  с распределением масс Рисса, равным  $\nu$.   
Из импликации  \ref{II_2}$\Rightarrow$\ref{I_2}  основной теоремы  для  этой 
субгармонической функции $v=\ln|F|$ найдётся целая функция $h\not\equiv 0$, для которой
\begin{equation}\label{vmf}
v+\ln |h|+m=\ln |Fh|+m
\end{equation}
 --- субгармоническая функция конечного типа и выполнено \eqref{umM}. Отсюда
\begin{equation}\label{Fm}
\ln \bigl|F(z)h(z)\bigr|+m(z) \overset{\eqref{vmf}}{\underset{z\in \CC}{\equiv}} v(z)+m(z)+\ln\bigl|h(z)\bigr|
\underset{z\in \overline\strip_b}{\overset{\eqref{umM}}{\leq}} M(z), 
\end{equation}
где целая функция $Fh$ обращается в нуль на  ${\mathrm Z}{\lfloor}_{\CC\setminus ({\overline{\mathrm X}_a}\cup \overline\strip_ s)}=\Zero_F$.
Тогда для целой функции $f:=f_sFh$ субгармоническая функция 
$\ln|f|+m=\ln |Fh|+m$ конечного типа как сумма таковых, и  при этом 
$$
\ln|f(z)|+m(z)=\ln|f_s(z)|+\ln |F(z)h(z)|+m(z)
\overset{\eqref{faZa},\eqref{Fm}}{\underset{z\in \overline\strip_b}{\leq}} M(z), 
$$
где  $f\not\equiv 0$ обращается в нуль на ${\mathrm Z}$, поскольку 
для сомножителя  $f_s$ имеем   \eqref{faZa}, а сомножитель $Fh$ обращается в нуль  на  ${\mathrm Z}{\lfloor}_{\CC\setminus ({\overline{\mathrm X}_a}\cup \overline\strip_ s)}$. Таким образом,    
установлено, что утверждение   \ref{II2} влечёт за собой утверждение  \ref{Ie}.   

Наконец, пусть выполнено утверждение \ref{Ie}. Для вывода из него утверждения  \ref{IIe} при заданном $b\geq 0$  выберем $s>b$ и вновь  
  рассмотрим произвольную субгармоническую  функцию $v$  с целочисленным распределением масс Рисса $\frac{1}{2\pi}{\bigtriangleup}v:=\nu:={\mathrm Z}$, как в  \eqref{vlnf}. Целая функция $f\not\equiv 0$
из утверждения \ref{Ie} представляется в виде $f\overset{\eqref{vlnf}}{=}f_{\mathrm Z}h$, где $h\not\equiv 0$ --- целая функция.
При этом  $v+m+\ln |h|=\ln |f_{\mathrm Z}h|+m=\ln |f|+m$. Отсюда по утверждению  \ref{Ie} субгармоническая функция 
$v+m+\ln |h|$ конечного типа, а из  \eqref{umMe} получаем 
\begin{equation*}
v(z)+m(z)+\ln |h(z)|\underset{z\in \CC}{\equiv} \ln \bigl|f(z)\bigr|+m(z)\overset{\eqref{umMe}}{\leq} M(z) \quad\text{\it при  всех $z\in \overline \strip_b$.} 
\end{equation*} 
Это означает, что выполнено утверждение \ref{I_2} основной теоремы для $\nu={\mathrm Z}$. 
Из импликации   \ref{I_2}$\Rightarrow$\ref{I_3} основной теоремы, истинной   по замечанию \ref{remmt} без ограничений на $\nu$,  
для функции $v=\ln|f_{\mathrm Z}|$ найдётся  некоторая целая функция  $h\not\equiv 0$, с которой выполнены 
неравенства \eqref{ubullet}--\eqref{ubull} с левой частью $v+\ln|h|=\ln |f_{\mathrm Z}h|$, являющейся субгармонической функцией конечного типа. 
Это означает, что для распределения масс ${\mathrm Z}=\nu$ нашлась целая функция $f:=f_{\mathrm Z}h$,
 обращающейся в нуль на ${\mathrm Z}$, со свойствами, требуемыми  в утверждении   \ref{IIe}.
Таким образом, импликация  \ref{Ie}$\Rightarrow$\ref{IIe} истинна, а доказательство теоремы 
\ref{th2_1} завершено. 
\end{proof}

Следующее следствие \ref{coref} позволяет объединить утверждения \ref{Ie} и \ref{IIe}  теоремы \ref{th2_1} в одно более простое и полностью избавится  от интегральных усреднений $M^{\bullet}$ и исключительного множества $E_b$ в утверждении \ref{IIe}  теоремы \ref{th2_1}.

\begin{corollary}\label{coref} 
Если в условиях теоремы\/ {\rm \ref{th2_1}} при некотором $s\in \RR^+\setminus 0$ сужение 
$\frac{1}{2\pi}{\bigtriangleup} M{\lfloor}_{\strip_s}$
--- целочисленное распределение масс, то в теореме\/ {\rm \ref{th2_1}} 
пару утверждений\/ {\rm \ref{Ie}} и\/ {\rm \ref{IIe}} можно заменить на одно более сильное  утверждение 
\begin{enumerate}
\item[{\rm \ref{Ie}$\cap$\ref{IIe}.}]  Для любого $b\in [0,s)$ 
найдётся целая функция   $f\not\equiv 0$ экспоненциального типа, обращающаяся в нуль на  ${\mathrm Z}$, для которой
\begin{equation}\label{umMes}
\ln \bigl|f(z)\bigr|\leq M(z) \quad\text{\it при  всех $z\in \overline \strip_b$} 
\end{equation} 
\end{enumerate}
с сохранением эквивалентности утверждениям\/ {\rm \ref{II2}} и\/ {\rm \ref{II}} теоремы\/ {\rm \ref{th2_1}}.
\end{corollary}

\begin{proof} Для любой субгармонической функции $m$ с целочисленным распределением масс Рисса 
$\frac{1}{2\pi}{\bigtriangleup}m=\frac{1}{2\pi}{\bigtriangleup}M{\lfloor}_{\strip_s}$
из утверждения \ref{Ie}  теоремы \ref{th2_1} найдётся такая целая функция $l\not\equiv 0$, что $m=\ln |l|$,    
а $\ln |fl|=\ln|f|+\ln|l|=\ln|f|+m$ --- субгармоническая функция конечного типа, 
для которой 
\begin{equation}\label{umMe+}
\ln \bigl|f(z)l(z)\bigr|\underset{z\in\CC}{\overset{\eqref{umMe}}{\equiv}}
 \ln \bigl|f(z)\bigr|+m(z)\leq M(z) \quad\text{при  всех $z\in \overline \strip_b$.} 
\end{equation} 
Таким образом, для целой  функции $fl$ экспоненциального типа имеет место  \eqref{umMe+}. Если переобозначить   произведение 
$fl$ как целую функцию $f$ экспоненциального типа, то утверждения \ref{Ie}  теоремы \ref{th2_1}  становиться в точности утверждением 
\ref{Ie}$\cap$\ref{IIe}, которое, очевидно, сильнее утверждения    \ref{IIe} теоремы \ref{th2_1}.
\end{proof}

\begin{corollary}\label{th2_2} Для любой целой функции\/  $g\not\equiv 0$ экспоненциального типа
и такого же, как в  теореме\/ {\rm \ref{th2_1}}, распределения точек ${\mathrm Z}$ каждое из утверждений\/  {\rm \ref{Ieg0}} и\/ {\rm \ref{II2g0}}  теоремы\/ {\rm \ref{th1_5}} эквивалентны утверждению\/ {\rm \ref{II}} теоремы\/ {\rm \ref{th2_1}}. 
\end{corollary}
\begin{proof} Положим $M:=\ln |g|\not\equiv -\infty$. Тогда  для любого числа $s>0$
сужение $\frac{1}{2\pi}{\bigtriangleup} M{\lfloor}_{\strip_s}=\Zero_g{\lfloor}_{\strip_s}$, очевидно, целочисленное распределение масс, 
а  утверждения  {\rm \ref{Ieg0}}  и\/ {\rm \ref{II2g0}}  теоремы\/ {\rm \ref{th1_5}} --- 
это в точности  соответственно утверждение \ref{Ie}$\cap$\ref{IIe} из следствия \ref{coref}
и утверждение  {\rm \ref{II2}}   теоремы\/ {\rm \ref{th2_1}}. 
Таким образом, следствие \ref{coref} влечёт за собой следствие  \ref{th2_2}.  
\end{proof}

\begin{proof}[теоремы \ref{th1_5}]  Утверждение \ref{Ieg0} при $E:=\varnothing$, очевидно, влечёт за собой \ref{IIg0}, 
которое согласно   примеру   \ref{ex2_1}  --- частный случай  утверждения  \ref{II}  теоремы \ref{th2_1}. 
Таким образом,  следствие \ref{th2_2} влечёт за собой  теорему \ref{th1_5}.
\end{proof} 

\begin{proof} [теоремы \ref{thBM1}] 
Эквивалентность \ref{llR0BM}$\Leftrightarrow$\ref{IIg0BM} --- двойственный вариант теоремы Бёрлинга\,--\,Мальявена о радиусе полноты в сочетании с классической теоремой Пэли\,--\,Винера, а эквивалентность \ref{IIg0BM}$\Leftrightarrow$\ref{IVg0BM}
--- одна из форм теоремы Бёрлинга\,--\,Мальявена о  мультипликаторе. Импликация   \ref{Ieg0BM}$\Rightarrow$\ref{IIg0BM} очевидна при выборе $g\equiv 1$. Наконец,  при выполнении утверждения   \ref{llR0BM} согласно   лемме \ref{lem2_1} при любом  $s\in \RR^+$ найдётся целая функции $f_s\not\equiv 0$ экспоненциального типа $\type\bigl[\ln |f_s|\bigr]<\pi c$, для которой 
$f_s({\mathrm Z})=0$ и $\bigl|f_s(z)\bigr|\leq 1$ при всех $z\in \overline \strip_{s}$.
Тогда для произвольной целой функции $g\not\equiv 0$ экспоненциального типа целая функция $f:=f_sg\not\equiv 0$
экспоненциального типа  
$$
\type \bigl[\ln|f|\bigr]\leq \type \bigl[\ln|f_s|\bigr]+\type \bigl[\ln|g|\bigr] <\pi c+\type \bigl[\ln|g|\bigr]
$$
обращается в нуль на ${\mathrm Z}\leq \Zero_{f_s}\leq \Zero_f$ и удовлетворяет неравенствам 
$$
\ln\bigl| f(z)\bigr|\equiv \ln\bigl|f_s(z) \bigr|+\ln\bigl|g(z) \bigr|\leq 
\ln\bigl|g(z)\bigr|\quad\text{при всех $z\in \overline \strip_s$}, 
$$
что даёт утверждение \ref{Ieg0BM} и завершает доказательство теоремы \ref{thBM1}.
\end{proof}

\section{Выметание  из  правой полуплоскости}\label{BalCrh}

\subsection{Выметания рода $0$ и $1$ распределений зарядов}
Для распределения заряда $\nu$ используем \cite[формула (1.9)]{KhaShm19}   его {\it функцию  распределения на\/}  $\RR$, 
обозначаемую как  $\nu_{\RR}\colon \RR\to  \RR$ и определённую  равенствами 
\begin{equation}\label{nuR} 
\nu_{\RR}(x_2)-\nu_{\RR}(x_1):=\nu\bigl((x_1,x_2]\bigr), \quad
-\infty <x_1<x_2<+\infty, 
\end{equation}
и   {\it функцию распределения $\nu_{i\RR}\colon \RR\to \RR$    на\/} $i\RR$, определённую равенствами
\begin{equation}\label{nuiR} 
\nu_{i\RR}(y_2)-\nu_{i\RR}(y_1):=\nu\bigl(i(y_1,y_2]\bigr), \quad -\infty <y_1<y_2<+\infty. 
\end{equation}
Поскольку эти функции распределения определены лишь с точностью до аддитивной постоянной, 
при необходимости используем их {\it нормировки в нуле}
\begin{equation}\label{nuo}
\nu_{\RR}(0):=0, \quad \nu_{i\RR}(0):=0.
\end{equation}
По построению \eqref{nuR} и \eqref{nuiR}    функции  $\nu_{\RR}$ и $\nu_{i\RR}$  локально ограниченной вариации  на $\RR$, а в случае {\it распределения масс\/} $\nu$ обе эти функции возрастающие. 
Обратно, любая функция локально ограниченной вариации на $\RR$ или $i\RR$ однозначно определяет распределение зарядов с носителем соответственно на  $\RR$ или $i\RR$. 

Мы напоминаем и адаптируем основные понятия и утверждения из \cite{KhaShm19} и \cite{KhaShmAbd20},
 а также частично из \cite{Kha91} и \cite{Kha91AA} о  выметании конечного рода $q\in \NN_0$ распределений зарядов, но пока применительно только к правой полуплоскости $\CC_{\rh}$ в случаях  $q:=0$ и $q:=1$. В 
\cite{KhaShm19} и \cite{KhaShmAbd20} в основном рассматривается {\it верхняя полуплоскость\/} $\CC^{\up}:=i\CC_{ \rh}$, что переносится на $\CC_{\rh}$ поворотом на прямой угол.

{\it Характеристическую функцию множества\/}  $S$ обозначаем через 
\begin{equation}\label{SdrS}
\boldsymbol{1}_S\colon z\underset{z\in \mathbb C}{\longmapsto} \begin{cases}
1&\text{ если $z\in S$},\\
0&\text{ если $z\notin S$}.
\end{cases}
\end{equation}
 
{\it Гармоническая  мера  для\/ $\CC_{\rh}$ в точке\/ $z\in \CC_{\rh}$} на интервалах $i(y_1,y_2]\subset i\overline\RR$
\begin{equation}\label{omega}
\omega_{\rh} \bigl(z,i(y_1,y_2]\bigr){\overset{\text{\cite[3.1]{KhaShm19}}}{:=}}\omega_{\CC_{\rh}}(z,i(y_1,y_2])
\underset{z\in \CC_{\rh}}{:=}\frac1{\pi}
\int_{y_1}^{y_2}\Re \frac{1}{z-iy} \dd y 
\end{equation} 
равна делённому на $\pi$ углу, под которым виден интервал $i(y_1,y_2]$ из точки $z\in \CC_{\rh}$ \cite[(3.1)]{Kha91AA}, \cite[1.2.1, 3.1]{KhaShm19},  а  в точках  мнимой оси $i\RR$ определяется как 
\begin{equation}\label{oiR}
\omega_{\rh} \bigl(iy,i(y_1,y_2]\bigr):=\boldsymbol{1}_{(y_1,y_2]}(y)
\quad\text{при $y\in  \RR$.}
\end{equation}
Для  распределения зарядов  $\nu $ при {\it классическом условии Бляшке\/} для  $\CC_{\rh}$
\begin{equation}\label{Blcl}
\ell_{|\nu|}^{\rh}(1,+\infty)\overset{\eqref{df:dDlLm}}{=} \int_{\CC_{\rh}\setminus  \DD} \Re \frac{1}{z}\dd |\nu| (z)<+\infty
\end{equation}
определено \cite[следствие 4.1, теорема 4]{KhaShm19}  
его классическое выметание из $\CC_{\rh}$ на  $ \CC_{\overline \lh}$ с носителем на  $  \CC_{\overline \lh}$, которое  в более широких рамках  \cite[определение 3.1]{KhaShmAbd20} 
представляет собой {\it выметание рода\/ $0$,\/} обозначавшееся в \cite{KhaShmAbd20} как
$\nu^{\bal[0]}_{\CC_{\overline \lh}}$. 
Здесь используется чуть более компактная  запись  
 $\nu^{\bal^0_{\rh}}:=\nu^{\bal[0]}_{\CC_{\overline \lh}}$. По определению {\it распределение зарядов\/}  $\nu^{\bal^0_{\rh}}$ --- это {\it сумма сужения\/}  $\nu{\lfloor}_{\CC_{\lh}}$ на $\CC_{\lh}$  с {\it распределением зарядов на $i\RR$,} определяемым в обозначениях \eqref{nuiR}  {\it функцией распределения\/}
\begin{equation}\label{mubal}
\nu^{\bal^0_{\rh}}_{i\RR}(y_2)-\nu^{\bal^0_{\rh}}_{i\RR}(y_1)\overset{\eqref{omega},\eqref{oiR}}{:=}
\int\limits_{\CC_{\overline \rh}} \omega_{\rh}\bigl(z, i(y_1,y_2]\bigr) \dd \nu(z)
\end{equation}
с нормировкой вида \eqref{nuo} при необходимости. 
Классическое выметание  рода $0$ не увеличивает полную  меру полной вариации  распределения зарядов, 
поскольку гармоническая мера \eqref{omega} вероятностная и 
\begin{equation}\label{omega1}
\bigl|\nu^{\bal^0_{\rh}}\bigr|(S)\underset{S\subset \CC}{\overset{\eqref{mubal}}{\leq}} |\nu|(S).
\end{equation}  

В \cite[определение 2.1]{KhaShmAbd20} вводилось понятие гармонического   заряда  рода\/ $1$ для верхней полуплоскости \/ $\CC^{\up}$ в точке\/ $z\in \CC^{\up}$,  обозначавшегося  в \cite[формула (2.1)]{KhaShmAbd20} через  $\Omega^{[1]}_{\CC^{\up}}$. Здесь используем 
поворот на прямой угол с переходом от $\CC^{\up}$ к $\CC_{\rh}$ и определим 
 {\it гармонический    заряд рода\/ $1$ для правой  полуплоскости \/ $\CC_{\rh}$} 
 как функцию $\Omega_{\rh}$ ограниченных  интервалов $i(y_1,y_2]\subset i\RR$ по правилу 
\begin{equation}\label{Ocr}
\Omega_{\rh}\bigl(z,i(y_1,y_2]\bigr)
\overset{\eqref{omega},\eqref{oiR}}{:=}\omega_{\rh}\bigl(z,i(y_1,y_2]\bigr)-\frac{y_2-y_1}{\pi}\Re\frac{1}{z}
\quad\text{в $z\in \CC_{\overline \rh}\setminus 0$}.
\end{equation}
Для  распределения зарядов  $\nu$ в  \cite[определение 3.1, теорема 1, замечание 3.3]{KhaShmAbd20} определялось 
{\it выметание\/ $\nu_{\CC_{\overline \lh}}^{\bal[1]}$ рода\/ $1$ распределения зарядов\/  $\nu$ из\/  $\CC_{\rh}$ на \/} $ \CC_{\overline \lh}$
  при $0\notin \supp \nu$. Здесь используется  более компактная запись  для такого выметания $\nu^{\bal^1_{\rh}}:=\nu^{\bal[1]}_{ \CC_{\overline \lh}}$.  По определению {распределение зарядов\/} $\nu^{\bal^1_{\rh}}$ --- это 
{\it сумма  сужения\/} $\nu{\lfloor}_{\CC_{\lh}}$ с {\it  распределением зарядов на\/} $i\RR$, определяемым в обозначениях \eqref{nuiR}  {\it функцией распределения} 
\begin{equation}\label{df:nurh}
\nu^{\bal^1_{\rh}}_{i\RR}(y_2)-\nu^{\bal^1_{\rh}}_{i\RR}(y_1)\overset{\eqref{Ocr}}{=}
\int_{\CC_{\rh}} \Omega_{\rh} \bigl(z, i(y_1,y_2]\bigr)\dd \nu (z)
\end{equation}
с нормировкой вида \eqref{nuo} при необходимости. 
\begin{remark}\label{remAA} Выметание распределения зарядов  рода $1$ из $\CC_{\rh}$ на  $ \CC_{\overline \lh}$ --- это часть глобального, или двухстороннего, выметания распределения зарядов  из $\CC\setminus i\RR$ на $i\RR$, рассмотренного  в \cite[\S~3]{Kha91AA} и сыгравшего там одну из ключевых ролей. Двустороннее выметание на мнимую ось можно рассматривать как последовательное выметание рода $1$ распределения зарядов  сначала из $\CC_{\rh}$ на $ \CC_{\overline \lh}$, а затем зеркально симметричной относительно $i\RR$ процедуры выметания рода $1$ получившегося распределения зарядов  из  $\CC_{\lh}$ на  $\CC_{\overline \rh}$. 
\end{remark}

Ограничение $0\notin \supp \nu$ для выметания рода $1$ легко преодолевается, если скомбинировать выметание рода $0$ части $\nu$  около нуля с выметанием рода $1$ для оставшейся части $\nu$. Для этого определяем {\it комбинированное выметание рода\/ $01$ 
распределения зарядов\/ $\nu$ из\/ $\CC_{\rh}$ на\/ $\CC_{\overline \lh}$} \cite[замечание 3.3, (3.43), (4.1)]{KhaShmAbd20}
\begin{equation}\label{bal01}
\nu^{\bal_{\rh}^{01}}:=\bigl(\nu{\lfloor}_{r_0\DD}\bigr)^{\bal_{\rh}^{0}}+\bigl(\nu{\lfloor}_{\CC\setminus r_0\DD}\bigr)^{\bal_{\rh}^{1}}
\end{equation}
при каком-нибудь фиксированном радиусе  $r_0\in \RR^+\setminus 0$. 
\begin{remark}\label{rem0nu}
Круг $r_0\DD$  в правой части \eqref{bal01}  можно заменить на любое ограниченное борелевское множество, содержащее полукруг $r_0\DD\setminus \CC_{\overline \lh}$, или даже на пустое множество.  Другими словами, можно  обойтись совсем без 
выметания рода $0$ и положить $\nu^{\bal_{\rh}^{01}}:=\nu^{\bal_{\rh}^{1}}$
если $0\notin \supp \nu$ или, более общ\'о,
\begin{equation*}
\int_{r_0\DD} \Re^+\frac{1}{z} \dd |\nu|(z)<+\infty, 
\end{equation*}
что, очевидно, выполнено, если для некоторого $r_0\in \RR^+\setminus 0$ имеем 
\begin{equation}\label{nur0+}
|\nu|\bigl(r_0\DD\cap \CC_{ \rh}\bigr)=0.
\end{equation}
\end{remark}
Теперь вопрос существования выметания  $\nu^{\bal_{\rh}^{01}}$  упирается лишь в поведение 
распределения зарядов $\nu$ около бесконечности.

Распределение зарядов $\nu$ принадлежит  {\it классу сходимости при порядке\/} роста  $p\in \NN_0$, если
\cite[определение 4.1]{HK}, \cite[\S~2, 2.1, (2.3)]{KhaShm19} 
\begin{equation}\label{sufc}
\int_1^{+\infty}\frac{|\nu|^{\rad}(t)}{t^{p+1}}\dd t<+\infty. 
\end{equation}
В \cite[3]{SalKha21}  использовались различные виды условий Линделёфа. 
Распределение зарядов $\nu$ удовлетворяет  {\it $\RR$-условию Линделёфа\/}  (рода $1$), если  
\begin{equation}
\sup_{r\geq 1} \biggl| \int_{1<|z|\leq r}\Re\frac{1}{z}\dd \nu(z)\biggr|<+\infty, 
\label{con:LpZR}
\end{equation}
что по определениям \eqref{df:dDlm+}--\eqref{df:dDlm-} эквивалентно соотношению
\begin{equation}
\sup_{r\geq 1} \bigl| \ell^{\rh}_{\nu}(1,r)-\ell^{\lh}_{\nu}(1,r)\bigr|<+\infty, 
\label{con:LpZRl}
\end{equation}
удовлетворяет {\it $i\RR$-условию Линделёфа\/} (рода $1$), если
\begin{equation}
\sup_{r\geq 1}\biggl| \int_{1<|z|\leq r}\Im\frac{1}{z}\dd \nu(z)
\biggr|<+\infty, 
\label{con:LpZiR}
\end{equation}
и удовлетворяет  {\it условию Линделёфа\/} (рода $1$), если
\begin{equation}
\sup_{r\geq 1}\biggl|\int_{1<|z|\leq r}\frac{1}{z}\dd \nu(z)
\biggr|<+\infty. 
\label{con:LpZ}
\end{equation}

Ключевая роль условий Линделёфа отражает следующая классическая  
\begin{theo}[{Вейерштрасса\,--\,Адамара\,--\,Линделёфа\,--\,Брело (\cite[3, Теорема  12]{Arsove53p},  \cite[4.1, 4.2]{HK}, \cite[2.9.3]{Az}, \cite[6.1]{KhaShm19})}]\label{pr:rep}
Если   $u\not\equiv -\infty$ --- субгармоническая функция конечного типа, то её   распределение масс Рисса  $\frac{1}{2\pi}{\bigtriangleup}u$
конечной верхней плотности и удовлетворяет условию Линделёфа \eqref{con:LpZ}. 

Обратно, если  распределение масс    $\nu$ конечной верхней плотности, то существует  субгармоническая функция $u_{\nu}$ с  
$\frac{1}{2\pi}{\bigtriangleup}u_{\nu}=\nu$  порядка $\ord[u_{\nu}]
\leq 1$, которая при выполнении условия Линделёфа \eqref{con:LpZ} для\/  $\nu$   будет уже  функцией  конечного типа. 
При этом  любая  субгармоническая   функция $u$ с  $\frac{1}{2\pi}{\bigtriangleup}u=\nu$
представляется в виде суммы $u=u_{\nu}+H$, где $H$ --- гармоническая функция на\/ $\CC$, которая при условии\/ $\type_2[u]=0$ является  гармоническим многочленом степени\/  $\deg H\leq 1$, а функция $u$ становится функцией порядка\/  $\ord[u]\leq 1$.
\end{theo}

В дальнейшем нам потребуется следующее

\begin{propos}[{\cite[основная теорема]{Kha22Lind}}]\label{theoB1} Пусть $\nu$  ---  распределение зарядов, для которого сужение $\nu{\lfloor}_{\CC_{ \rh}}$ при\-н\-а\-д\-л\-ежит   классу сходимости при порядке $p\overset{\eqref{sufc}}{=}2$.  
Тогда существует выметание $\nu^{\bal^{01}_{\rh}}$ рода $01$ из $\CC_{\rh}$ на $ \CC_{\overline \lh}$. В частности, если $\ord[\nu]<2$, то  $\nu$ из  класса сходимости при порядке $<2$ и   $\ord[\nu^{\bal^{01}_{\rh}}]\leq \ord[\nu]$.  
Если  $\nu$ конечной верхней плотности и 
\begin{equation}\label{cB2l}
\sup_{r\geq 1}\bigl|\ell_{\nu}^{\rh}(1,r)\bigr|\overset{\eqref{df:dDlm+}}{:=}
<+\infty,
\end{equation}
то $\nu^{\bal^{01}_{\rh}}$ ---  распределение зарядов  конечной верхней плотности  c разностью $\nu-\nu^{\bal^{01}_{\rh}}$,  удовлетворяющей условию Линделёфа. При этом  если   носитель $\supp \nu$ не пересекается с замкнутым углом раствора строго больше, чем  $\pi$, содержащим  $i\RR$,  с вершиной в нуле и биссектрисой $\RR$, т.е.
\begin{equation}\label{eA1}
 \bigl\{z\in \CC \bigm| \Re z\leq a |z|\bigr\}  \bigcap \supp \nu=\varnothing 
\quad\text{для  некоторого $a \in (0,1)$},
\end{equation}
то в обозначении \eqref{df:nup} для радиальных считающих функций $|\nu^{\bal^{01}_{\rh}}\bigr|_{iy}^{\rad}$ полной вариации $|\nu^{\bal^{01}_{\rh}}|$ распределения зарядов  $\nu^{\bal^{01}_{\rh}}$ с центрами  $iy\in i\RR$  имеем
\begin{equation}\label{trnuair}
\sup_{y\in \RR}\sup_{t\in (0,1]}\frac{|\nu^{\bal^{01}_{\rh}}\bigr|_{iy}^{\rad}(t)}{t}<+\infty.
\end{equation}

\end{propos}

\subsection{Выметания  разностей  субгармонических  функций}\label{6_2}
Пусть $\mathcal U=u-v$ --- разность субгармонических на $\CC$ функций $u$ и $v$, или $\delta$-суб\-г\-а\-р\-м\-о\-н\-и\-ч\-е\-с\-к\-ая функция, для которой при   $u\not\equiv -\infty$ и $v\not\equiv -\infty$ пишем  ${\mathcal U}\not\equiv \pm\infty$.  Значения такой функции ${\mathcal U}\not\equiv \pm\infty$ определены во всех точках, в которых одна из функций  $u$ или $v$ принимает значение из $\RR$, т.е. вне некоторого полярного множества, а её   {\it распределение зарядов Рисса\/}
\begin{equation}\label{Delta}
\varDelta_{\mathcal U}:=\frac{1}{2\pi}{\bigtriangleup}{\mathcal U}\overset{\eqref{Riesz}}{:=}
\frac{1}{2\pi}{\bigtriangleup}u-\frac{1}{2\pi}{\bigtriangleup}v\overset{\eqref{Riesz}}{=}
\varDelta_u-\varDelta_v
\end{equation} 
--- разность распределений масс Рисса  $u$ и $v$. Следуя  \cite[определение 4.1]{KhaShmAbd20}, 
{\it  $\delta$-субгармоническим выметанием $\delta$-субгармонической функции  ${\mathcal U}\not\equiv \pm\infty$  из\/ $\CC_{\rh}$ на $ \CC_{\overline \lh}$}  называем 
каждую  {\it $\delta$-субгар\-м\-о\-н\-и\-ч\-е\-с\-к\-ую  функцию,\/} обозначаемую как ${\mathcal U}^{{\Bal_{\rh}}}$, которая  равна   функции ${\mathcal U}$ на замкнутой левой полуплоскости\/ $\CC_{\overline \lh}$ вне некоторого полярного множества и одновременно гармоническая  на     $\CC_{\rh}$. 

\begin{propos}[{\rm  \cite[теорема 2]{Kha22bal}
}]\label{Balv}
Пусть   ${\mathcal U}\not\equiv \pm \infty$  --- $\delta$-суб\-г\-а\-р\-м\-о\-н\-и\-ч\-е\-с\-к\-ая  функция c распределением зарядов
 Рисса  \eqref{Delta} конечной верхней плотности представима  в виде разности субгармонических функций порядка $\leq 1$.
 
Тогда существует $\delta$-субгармоническое  выметание\/ $\mathcal U^{\Bal_{\rh}}\not\equiv \pm\infty$ из\/ $\CC_{\rh}$ на $ \CC_{\overline \lh}$
c распределением зарядов Рисса 
\begin{equation}\label{112}
\frac{1}{2\pi}{\bigtriangleup}{\mathcal U}^{\Bal_{\rh}}\overset{\eqref{bal01}}{=}\varDelta_{\mathcal U}^{\bal^{01}_{\rh}}, 
\end{equation}  
представимое вне некоторого полярного множества   в виде разности 
\begin{equation}\label{reprvB}
{\mathcal U}^{\Bal_{\rh}}:=u_+-u_-, \quad u_{\pm}\not\equiv -\infty, \quad \ord[u_{\pm}]\overset{\eqref{senu0:a}}{\leq}  1,
\end{equation} 
 субгармонических функций $u_{\pm}\not\equiv -\infty$. Если при этом функция $\mathcal U$ 
гармоническая в открытом полукруге $r_0\DD\cap \CC_{ \rh}$ при некотором $r_0>0$, то правую часть в \eqref{112} можно заменить на выметание $\varDelta_{\mathcal U}^{\bal^{1}_{\rh}}$ рода $1$.  
\end{propos}

\begin{proof} По \cite[теорема 6]{KhaShmAbd20} для любой $\delta$-субгармонической функции $\mathcal U$ с распределением зарядов Рисса конечного типа  существует выметание ${\mathcal U}^{\Bal_{\rh}}$ с распределением зарядов Рисса  \eqref{112}, представимое в виде 
 \begin{equation}\label{v+u}
{\mathcal U}^{\Bal_{\rh}}=v_+-u_-+H,  \quad v_+\not\equiv -\infty, \quad u_-\not\equiv -\infty,
\end{equation} 
где $v_+$ и $u_-$--- субгармонические функции порядка не выше $1$ , а $H$ --- гармоническая функция на $\CC$.  
При этом если функция $\mathcal U$ представима в виде разности субгармонических функций  порядка не выше $1$, 
то в заключительной части  \cite[теорема 6]{KhaShmAbd20} отмечено, что в качестве  $H$ можно выбрать  гармонический многочлен степени $\deg H\leq 1$. Таким образом, $u_+:=v_++H\not\equiv -\infty$ --- субгармоническая функция порядка не выше $1$ и из 
\eqref{v+u} получаем \eqref{reprvB}. Возможность замены правой части в \eqref{112} на  $\varDelta_{\mathcal U}^{\bal^{1}_{\rh}}$ следует из замечания \ref{rem0nu} в части \eqref{nur0+}. 
\end{proof}


\section{Две  конструкции с  распределениями зарядов и их выметанием, связанные с  логарифмическими функциями интервалов}
\setcounter{equation}{0} 

Если   функция $z \mapsto \Re \dfrac{1}{z}$ суммируема по полной вариации $|\nu|$ распределения зарядов $\nu $ в правой окрестности нуля 
\begin{equation}\label{{klsx0}r}
\int_{\DD\cap \CC_{\rh}}\Re \frac{1}{z}\dd |\nu|(z)<+\infty
\quad \Longleftrightarrow\quad  
\lim_{0<r\to 0} \ell_{|\nu|}^{\rh} (r,1) <+\infty,
\end{equation}
или функция $z \mapsto \Re \dfrac{-1}{z}$ суммируема по  $|\nu|$  в левой  окрестности нуля
\begin{equation}\label{{klsx0}l}
\int_{\DD\cap \CC_{\lh}}\Re \frac{-1}{z}\dd |\nu|(z)<+\infty
\quad \Longleftrightarrow\quad  
\lim_{0<r\to 0} \ell_{|\nu|}^{\lh} (r,1) <+\infty, 
\end{equation}
то  по непрерывности при $0<r\to 0$ определены и вводившиеся   ранее в \cite[(2.2)]{MR}, \cite[гл.~22, Определение]{RC} лишь для  положительных распределений точек   соответственно \textit{правый\/} и {\it левый характеристические логарифмы заряда\/} $\nu$
\begin{align}
\ell_{\nu}^{\rh}(R)&:=\ell_{\nu}^{\rh}(0,R)\overset{\eqref{{klsx0}r}}{:=}
\lim_{0<r\to 0}\ell_{\nu}^{\rh}(r,R),
\label{{df:dDlL1}r}
\\ 
\ell_{\nu}^{\lh}(R)&:=\ell_{\nu}^{\lh}(0,R)
\overset{\eqref{{klsx0}l}}{:=}\lim_{0<r\to 0}\ell_{\nu}^{\lh}(r,R), 
\notag
\end{align}
а для {\it меры\/} $\mu$ --- двусторонний {\it характеристический логарифм меры\/} $\mu$
\begin{equation*}
\ell_{\mu}(R):=\ell_{\mu}(0,R):=\lim_{0<r\to 0}\ell_{\mu}(r,R).
\end{equation*}

Близкие версии следующего предложения для распределений точек содержатся в статьях  
\cite[лемма 3.1]{MR}, \cite[лемма 22.2]{RC}, \cite[лемма 1.1.]{Kha89},  \cite[лемма 1]{Kha91AA}. 
\begin{propos}\label{addsec} Для  любого  распределения зарядов    $\eta$ на $\CC$ при\/  $0\notin \supp \eta$ 
можно построить  распределение масс $\alpha$ с $\supp \alpha\subset \RR^+\setminus 0$, для которого
\begin{equation}\label{|l|}
\sup_{0\leq r<R<+\infty} \bigl|\ell_{\eta+\alpha}^{\rh}(r,R)\bigr|\leq 2
\sup_{0\leq  r<R<+\infty} \ell_{\eta}^{\rh}(r,R),
\end{equation}
где в случае   распределения зарядов $\eta$ конечной верхней плотности 
построенное  распределение масс $\alpha$ будет также  конечной верхней плотности. 
\end{propos}
\begin{proof} 
Возрастающая по построению функция 
\begin{equation}\label{ainf}
a\colon t \overset{\eqref{df:dDlm+}}{\underset{t\in \RR^+ }{\longmapsto}}-\sup_{s\geq t} \ell_{\eta}^{\rh}(s)
=\inf_{s\geq t} \bigl(-\ell_{\eta}^{\rh}(s)\bigr),
\end{equation}
 однозначно определяет  распределение масс  $\alpha$ 
через её  возрастающую функцию распределения \eqref{nuR}, построенную как  
\begin{equation}\label{{adist}r}
\alpha_{\RR}\colon x\underset{x\in \RR^+ }{\longmapsto} \int_0^x t \dd a(t)
=xa(x)-\int_0^xa(t)\dd t.
\end{equation}
Ввиду $ 0\notin \supp \eta$ из построения \eqref{ainf} функция $a$ постоянна на некотором интервале $[0,r_0)$, а из \eqref{{adist}r} из функция распределения $\alpha_{\RR}\equiv 0$ на $[0,r_0)$  и $\supp \alpha\subset \RR^+\setminus 0$.  По построению 
\eqref{ainf}  также имеем 
\begin{align}
a(t)&\overset{\eqref{ainf}}{=}\inf_{s\geq t}\bigl( -\ell_{\eta}^{\rh}(s)\bigr)\leq 
 -\ell_{\eta}^{\rh}(t),
\quad\text{при всех $t\in \RR^+$, откуда $a(0)\leq 0$},  
\label{aup}\\
 a(t)&\overset{\eqref{ainf}}{=}  -\sup_{s\geq t} \ell_{\eta}^{\rh}(s)
\geq  -\sup_{0\leq  r<R<+\infty} \ell_{\eta}^{\rh}(r,R)
\quad\text{при всех $t\in \RR^+$.}
\label{a}
\end{align}
Кроме того,  по определению \eqref{ainf} при всех $t\in \RR^+$ имеем 
\begin{equation*}
\ell_{\alpha}^{\rh} (t)\overset{\eqref{{df:dDlL1}r}}{=}
\int_0^t\frac{1}{x}\dd \alpha^{\RR}(x)\overset{\eqref{{adist}r}}{:=}\int_0^t  \dd a(t)
=a(t)-a(0),
\end{equation*}
откуда для любых $t\in \RR^+$ получаем 
\begin{equation}\label{nua}
\ell_{\eta+\alpha}^{\rh}(t)
=\ell_{\eta}^{\rh}(t)+a(t)-a(0)\overset{\eqref{aup}}{\leq}-a(0)
\overset{\eqref{a}}{\leq}\sup_{0\leq  r<R<+\infty} \ell_{\eta}^{\rh}(r,R), 
\end{equation}
а также
\begin{multline*}
\ell_{\eta+\alpha}^{\rh}(t) \overset{\eqref{nua}}{=} \ell_{\eta}^{\rh}(t)+a(t)-a(0) \overset{\eqref{ainf}}{=}
\inf_{s\geq t}\bigl(\ell_{\eta}^{\rh}(t) -\ell_{\eta}^{\rh}(s)\bigr)-a(0)\\
=-\sup_{s\geq t}\bigl(\ell_{\eta}^{\rh}(s)-\ell_{\eta}^{\rh}(t) \bigl)-a(0)
\overset{\eqref{{df:dDlL1}r}}{=}-\sup_{s\geq t}\bigl(\ell_{\eta}^{\rh}(s,t) \bigl)-a(0)
\\
\overset{\eqref{aup}}{\geq}    -\sup_{0\leq r<R<+\infty} \ell_{\eta}^{\rh}(r,R)
\quad\text{для любых $t\in \RR^+$}.
\end{multline*}
Отсюда и из \eqref{nua} сразу следует
\begin{equation*}
\sup_{t\geq 0} \bigl|\ell_{\eta+\alpha}^{\rh}(t)\bigr|\leq  \sup_{0\leq r<R<+\infty} \ell_{\eta}^{\rh}(r,R).
\end{equation*}
что влечёт за собой \eqref{|l|}, так как 
\begin{equation*}
\sup_{0<r<R<+\infty} \bigl|\ell_{\eta+\alpha}^{\rh}(r,R)\bigr|\leq 
\sup_{R\in \RR^+} \bigl|\ell_{\eta+\alpha}^{\rh}(R)\bigr|+\sup_{r\in \RR^+} \bigl|\ell_{\eta+\alpha}^{\rh}(r)\bigr|\leq 
2 \sup_{0\leq r<R<+\infty} \ell_{\eta}^{\rh}(r,R).
\end{equation*}
Если $\eta$ --- распределение зарядов  конечной верхней плотности, то 
\begin{equation}\label{l2}
\bigl|\ell_{\eta}^{\rh}(r,2r)\bigr|\leq \int_r^{2r}\frac{1}{t}\dd |\eta|^{\rad}(t)\leq
\frac{1}{r}|\eta|^{\rad}(2r)\underset{r\to +\infty}{=}O(1). 
\end{equation}
Отсюда и из \eqref{|l|} получаем 
\begin{multline*}
\frac{1}{2r}\bigl(\alpha^{\rad}(2r)-\alpha^{\rad}(r)\bigr)
\leq \int_r^{2r} \frac{1}{t}\dd \alpha^{\rad}(t)
=\ell_{\alpha}^{\rh}(r,2r) \\
\leq \bigl|\ell_{\eta+\alpha}^{\rh}(r,2r)\bigr|+\bigl|\ell_{\eta}^{\rh}(r,2r)\bigr|
\overset{\eqref{|l|}}{\leq} 
2\sup_{0\leq r<R<+\infty} \ell_{\eta}^{\rh}(r,R)
+\bigl|\ell_{\eta}^{\rh}(r,2r)\bigr|
\underset{r\to +\infty}{\overset{\eqref{l2}}{=}}O(1). 
\end{multline*}
Это  означает, что  распределение масс $\alpha$ конечной верхней плотности.
\end{proof}

\begin{propos}\label{pr52} Пусть  $ a \in (0,1)$ и носители распределений  масс  $\nu$
и $\mu$ конечного типа    содержаться в угле   $\bigl\{z\in \CC\bigm|  \Re z >  a |z|\bigr\}$, а также  
\begin{equation}\label{|l|+}
\sup_{0\leq r<R<+\infty} \ell_{\nu-\mu}^{\rh}(r,R) <+\infty . 
\end{equation}
Тогда найдутся  распределение масс $\alpha$ конечной верхней плотности  
с носителем на $\RR^+\setminus 0$, для которого существует выметание 
$(\nu+\alpha-\mu)^{\bal^1_{\rh}}$ 
рода $1$ из $\CC_{\rh}$
на $ \CC_{\overline \lh}$ конечной верхней плотности с носителем на $i\RR$,
а также распределение масс $\beta$ с носителем на $i\RR$ и число $c\in \RR^+$, для которых 
\begin{equation}\label{bm_1}
(\nu+\alpha+\beta -\mu)^{\bal^1_{\rh}}=
(\nu+\alpha -\mu)^{\bal^1_{\rh}}+\beta =c\mathfrak m_1{\lfloor}_{i\RR}
\end{equation}
---  линейная мера Лебега на мнимой оси, домноженная на $c$,
 а распределение зарядов  $\nu+\alpha+\beta-\mu$ удовлетворяет условию Линделёфа.   
 \end{propos}
\begin{proof} Для распределения зарядов $\eta:=\nu-\mu$ при условии \eqref{|l|+}  
по предложению  \ref{addsec} существует  распределение масс $\alpha$   конечной верхней плотности 
 с носителем на $\supp \alpha \subset  \RR^+\setminus 0$, для которых 
\begin{equation*}
\sup_{0\leq r<R<+\infty} \bigl|\ell_{\nu+\alpha-\mu}^{\rh}(r,R)\bigr|
=\sup_{0\leq r<R<+\infty} \bigl|\ell_{\eta +\alpha}^{\rh}(r,R)\bigr|
\overset{\eqref{|l|}}{<}+\infty.
\end{equation*}
Вместе с расположением  носителей $\nu$ и $\mu$ в  $\bigl\{z\in \CC\bigm| \Re z>  a |z|\bigr\}$
это означает выполнение условий \eqref{cB2l} и \eqref{eA1}  предложения \ref{theoB1} 
для  $\nu+\alpha-\mu$ в роли $\nu$. По предложению \ref{theoB1} существует выметание $(\nu+\alpha-\mu)^{\bal^1_{\rh}}$ конечной верхней плотности, сосредоточенное в данном случае исключительно на  $i\RR$, 
разность  $(\nu+\alpha-\mu)-(\nu+\alpha-\mu)^{\bal^1_{\rh}}$ удовлетворяет условию Линделёфа и 
\begin{equation}\label{trnuair+}
\sup_{y\in \RR}\sup_{t\in (0,1]}\frac{\bigl|(\nu+\alpha-\mu)^{\bal^1_{\rh}}\bigr|_{iy}(t)}{t}\overset{\eqref{trnuair}}{<}+\infty.
\end{equation}
Положим 
\begin{equation}\label{varth}
(\nu +\alpha -\mu)^{\bal^1_{\rh}}=: \vartheta=\vartheta^+-\vartheta^- ,
\end{equation}  
где  $\vartheta^+$ и $\vartheta^-$ --- это верхняя и нижняя вариации распределения зарядов $\vartheta$. По построению   распределения масс $\vartheta^\pm$ конечной верхней плотности с носителями на $i\RR$ и ввиду 
\eqref{trnuair+} при некотором $c\in \RR^+$ удовлетворяют ограничениям  
\begin{equation}\label{trnuair+-}
\vartheta^{\pm}_{iy}(t)\leq 2ct \quad\text{при всех $y\in \RR$ и $t\in (0,1]$.}
\end{equation}
Отсюда для линейной меры Лебега $\mathfrak m_1{\lfloor}_{i\RR}$ на $i\RR$
разность $c\mathfrak m_1{\lfloor}_{i\RR}-\vartheta^+$ --- это распределение {\it масс\/} конечной верхней плотности с носителем на $i\RR$, а из  \eqref{varth} при этом получаем 
\begin{equation}\label{vt}
(\nu +\alpha -\mu)^{\Bal}+\underset{\beta}{\underbrace{\vartheta^-+
(c\mathfrak m_1{\lfloor}_{i\RR}-\vartheta^+)}}
\overset{\eqref{varth}}{=}\vartheta^++c\mathfrak m_1{\lfloor}_{i\RR}-\vartheta^+=c\mathfrak m_1{\lfloor}_{i\RR}.
\end{equation}
При  таком  выборе распределения {\it масс\/} $\beta:=\vartheta^-+
(c\mathfrak m_1{\lfloor}_{i\RR}-\vartheta^+)$ конечной верхней плотности равенство \eqref{vt} означает, что  выполнено  
второе равенство в  \eqref{bm_1}, а поскольку  $\supp \beta\subset  i\RR$, то по определению \eqref{df:nurh} выметания рода $1$
из $\CC_{ \rh}$ имеем $\beta^{\bal^1_{\rh}}=\beta$, что влечёт за собой  и первое равенство в \eqref{bm_1}. 
Наконец, как отмечалось выше перед \eqref{trnuair+},  $(\nu+\alpha-\mu)-(\nu+\alpha-\mu)^{\bal^1_{\rh}}$
удовлетворяет условию Линделёфа,   а из явного вида распределения масс $c\mathfrak m_1|_{i\RR}$ правой части  в \eqref{bm_1}, очевидно, удовлетворяющего условию Линделёфа, и из второго равенства в \eqref{bm_1} также следует, что и 
распределение зарядов   $(\nu+\alpha -\mu)^{\bal^1_{\rh}}+\beta $ удовлетворяет условию Линделёфа. Следовательно, и их сумма, равная  $(\nu+\alpha-\mu)+\beta=\nu+\alpha+\beta -\mu$,  удовлетворяет условию Линделёфа. 
\end{proof}

\section{Выметание на вертикальную полосу}

\subsection{Сдвиги и двустороннее  выметание распределения зарядов}\label{5_1Ss}
{\it Зе\-ркальная симметрия $z\underset{z\in \CC}{\longmapsto} -\Bar z$ относительно мнимой оси\/} позволяет все результаты о выметании рода  $1$ из $\CC_{ \rh}$ на $ \CC_{\overline \lh}$ переформулировать для выметания 
из левой полуплоскости $ \CC_{\lh}$ на $\CC_{\overline \rh}$ с заменой, где необходимо, правой логарифмической  функции интервалов \eqref{df:dDlm+} на  левую \eqref{df:dDlm-},  а также с переобозначением  верхнего индекса $^{\bal^1_{\rh}}$  как  $^{\bal^1_{\lh}}$ при выметании из левой полуплоскости $\CC_{ \lh}$.

Для распределения зарядов $\nu$ и точки $w\in \CC$ через 
$\nu_{\vec{w}}$ обозначаем {\it $w$-сдвиг  распределения зарядов\/ $\nu$,} определяемый 
как 
\begin{equation}\label{vecz}
\nu_{\vec{w}}(K):=\nu(K-w) \quad\text{на компактах $K\subset \CC$.}
\end{equation}

\begin{propos}\label{prb2} Пусть $\nu$ --- распределение зарядов конечной верхней пло\-т\-н\-о\-с\-ти. Тогда  
для любых $w\in \CC$ и   $r_0\in \RR^+\setminus 0$  имеем 
\begin{equation}\label{K1}
\sup_{r\geq r_0} \bigl|\ell^{\rh}_{\nu-\nu_{\vec{w}}}(r_0, r)\bigr|
+\sup_{r\geq r_0} \bigl|\ell^{\lh}_{\nu-\nu_{\vec{w}}}(r_0, r)\bigr|<+\infty,
\end{equation}
a  $\nu$ и $\nu_{\vec{w}}$ могут удовлетворять какому-либо одному из трёх видов 
условий Линделёфа \eqref{con:LpZR}--\eqref{con:LpZRl}, \eqref{con:LpZiR} или \eqref{con:LpZ}
только одновременно.    
\end{propos}
Доказательство предложения \ref{prb2}, легко следующее   из определений $\ell^{\rh}$ и $\ell^{\lh}$ в \eqref{df:dDlm+}--\eqref{df:dDlm-} и условий Линделёфа \eqref{con:LpZR}--\eqref{con:LpZ}, опускаем. 
 
{\it Выметание\/} рода $01$ распределения зарядов $\nu$  {\it на  замкнутую вертикальную полосу\/}  $\overline \strip_b$ ширины $2b\geq 0$ из \eqref{{strip}c} опишем  в пять шагов [b\ref{i1}]--[b\ref{i5}], применяя каждый шаг к распределению зарядов, полученному на предыдущем шаге:
\begin{enumerate}[{[b1]}]
\item\label{i1} $(-b)$-сдвиг $\nu_{\vec{-b}}$ распределения зарядов $\nu$;
\item\label{i2} выметание $\nu_{\vec{-b}}^{\bal^{01}_{\rh}}$  рода  $01$ из правой полуплоскости  $\CC_{ \rh}$ на $ \CC_{\overline \lh}$;
\item\label{i3} $2b$-сдвиг $\Bigl(\nu_{\vec{-b}}^{\bal^{01}_{\rh}}\Bigr)_{\vec{2b}}$
 распределения зарядов $\nu_{\vec{-b}}^{\bal^{01}_{\rh}}$;
\item\label{i4} выметание $\Bigl(\nu_{\vec{-b}}^{\bal^{01}_{\rh}}\Bigr)_{\vec{2b}}^{{\bal^{01}_{\lh}}}$
рода $01$ из левой полуплоскости $ \CC_{\lh}$ на $\CC_{\overline \rh}$;
\item\label{i5} $(-b)$-сдвиг $\biggl(\Bigl(\nu_{\vec{-b}}^{\bal^{01}_{\rh}}\Bigr)_{\vec{2b}}^{{\bal^{01}_{\lh}}}\biggr)_{\vec{-b}}$ распределения зарядов $\Bigl(\nu_{\vec{-b}}^{\bal^{01}_{\rh}}\Bigr)_{\vec{2b}}^{{\bal^{01}_{\lh}}}$. 
\end{enumerate} 
Полученное  в [b\ref{i5}]  распределение зарядов для краткости обозначаем как  
\begin{equation}\label{Balb1}
\nu^{\Bal^{01}_b}:=\biggl(\Bigl(\nu_{\vec{-b}}^{\bal^{01}_{\rh}}\Bigr)_{\vec{2b}}^{{\bal^{01}_{\lh}}}\biggr)_{\vec{-b}}
\end{equation}
и называем {\it выметанием рода\/ $01$ на\/ $\overline \strip_b$ распределения зарядов\/} $\nu$, если  шаги  
[b\ref{i2}] и [b\ref{i4}] реализуемы.  Для последнего  по предложению \ref{theoB1} достаточно, чтобы распределение зарядов $\nu$ было 
из класса сходимости при порядке $p\overset{\eqref{sufc}}{=}2$. 

\begin{remark}\label{remB01}
По замечанию \ref{rem0nu} при $\pm b\notin \supp \nu$
или, более общ\'о,  при 
\begin{equation}\label{Reznub}
\int_{b+r_0\DD} \Re^+\frac{1}{z-b} \dd |\nu|(z)+ 
\int_{-b+r_0\DD} \Re^-\frac{1}{z+b} \dd |\nu|(z)
<+\infty  
\end{equation}
для некоторого $r_0>0$ можно  в [b\ref{i2}] и [b\ref{i4}], а в итоге и в  \eqref{Balb1}  обойтись выметанием рода $1$, результат чего  в  \eqref{Balb1} будем обозначать   через $\nu^{\Bal^1_b}$. В частности, \eqref{Reznub} выполнено, если 
для некоторого числа $r_0>0$ имеем равенство
\begin{equation}\label{nur0+b}
|\nu|\bigl(b+r_0\DD\cap \CC_{ \rh}\bigr)+|\nu|\bigl(-b+r_0\DD\cap \CC_{ \lh}\bigr)=0.
\end{equation}

\end{remark}

\begin{propos}\label{pr52s} 
Пусть распределение масс   $\mu$ конечной верхней пло\-т\-н\-о\-с\-ти удовлетворяет условию Линделёфа и для распределения масс $\nu$ имеем
\begin{gather}
\sup_{1\leq r<R<+\infty} \Bigl(\ell_{\nu}(r,R)-\ell_{\mu}(r,R)\Bigr)<+\infty,
\label{lJMmull}
\\
\supp (\nu+\mu) \subset \CC\setminus \Bigl(\overline{\mathrm X}_{ a}\bigcup \overline\strip_b\Bigr) 
\quad\text{при  некоторых $b\in \RR^+$ и  $ a \in (0,1)$}.
\label{bst}
\end{gather}
 Тогда найдутся  распределение масс $\alpha$ конечной верхней плотности  
с носителем на $\RR\setminus [-b,b]$, для которого существует выметание 
$(\nu+\alpha-\mu)^{\Bal_b^1}$ 
рода $1$  на $\overline \strip_b$
конечной верхней плотности с носителем на паре  вертикальных прямых   $\pm b+i\RR$,
проходящих через точки $\pm b\in \RR$,
а также пара распределений масс $\beta_{\pm}$ с носителями  на $\pm b+i\RR$
 и число $c\in \RR^+$, для которых 
\begin{equation}\label{bm_1b}
(\nu+\alpha+\beta_++\beta_- -\mu)^{\Bal^1_b} =
(\nu+\alpha -\mu)^{\Bal^1_{b}}+\beta_++\beta_- =c\mathfrak m_1{\lfloor}_{b+i\RR}
+c\mathfrak m_1{\lfloor}_{-b+i\RR}
\end{equation}
---  пара линейных  мер Лебега на прямых $\pm b+i\RR$, домноженных на $c$, 
а распределение масс $\nu+\alpha+\beta_++\beta_-$ удовлетворяет условию Линделёфа.
 \end{propos}
\begin{proof} 
Убедимся,  что $\nu$ конечной верхней плотности.  Действительно, 
из соотношения \eqref{lJMmull} для некоторого   $C\in \RR$ ввиду  $\type[\mu]<+\infty$ имеем
\begin{multline*}
\ell_{\nu}(r,2r)\leq \ell_{\mu}(r,2r)+C\leq 
\int_r^{2r} \frac{1}{|z|}\dd \mu(z)\\
\leq \frac{1}{r}\bigl(\mu^{\rad}(2r)-\mu^{\rad}(2r)\bigr)\leq
2\type[\mu]+1+C
\end{multline*} 
при достаточно больших $r\in \RR^+$, а также  ввиду \eqref{bst} имеем оценку снизу для 
\begin{multline}\label{ellsn}
\ell_{\nu}(r,2r) \overset{\eqref{df:dDlLm}}{\geq}\ell_\nu^{\rh}(r, 2r)
\overset{\eqref{df:dDlm+}}{=} \int_{r}^{2r}\frac{\Re z}{|z|^2} \dd \nu{\lfloor}_{\CC_{\overline \rh}}(z)
\\
\overset{\eqref{bst}}{\geq} \int_{r}^{2r}\frac{a|z|}{|z|^2} \dd \nu{\lfloor}_{\CC_{\overline\rh}}(z)
\geq \frac{a}{2r}\bigl(\nu{\lfloor}_{\CC_{\overline\rh}}^{\rad}(2r)-\nu{\lfloor}_{\CC_{\overline\rh}}^{\rad}(r)\bigr),
\end{multline}
откуда согласно предыдущей оценке сверху  получаем
\begin{equation*}
\nu{\lfloor}_{\CC_{\overline\rh}}^{\rad}(2r)-\nu{\lfloor}_{\CC_{\overline\rh}}^{\rad}(r)
\leq  \bigl(2\type[\mu]+1+C\bigr) \frac{2}{a}r
\end{equation*}
при достаточно больших $r\in\RR^+$. Отсюда следует, что распределение масс 
$\nu{\lfloor}_{\CC_{\overline\rh}}$ конечной верхней плотности. 
Аналогично устанавливается,  что  $\nu{\lfloor}_{\CC_{\overline\lh}}$ конечной верхней плотности, 
откуда и  $\nu$ конечной верхней плотности.

Поскольку по условию распределение  масс $\mu$ удовлетворяет   условию Линделёфа \eqref{con:LpZ}, то  $\mu$ удовлетворяет 
{\it $\RR$-условию Линделёфа\/} \eqref{con:LpZR}--\eqref{con:LpZRl} и 
\begin{multline*}
\sup_{1\leq r<R<+\infty} \bigl| \ell^{\rh}_{\mu}(r,R)-\ell^{\lh}_{\mu}(r,R)\bigr|
\overset{\eqref{con:LpZR}}{=}\sup_{1\leq r<R<+\infty} \biggl| \int_{r<|z|\leq R}\Re\frac{1}{z}\dd \mu(z)\biggr|\\
\leq
\sup_{r\geq 1} \bigl| \ell^{\rh}_{\mu}(1,r)-\ell^{\lh}_{\mu}(1,r)\bigr|+
\sup_{R\geq 1} \bigl| \ell^{\rh}_{\mu}(1,R)-\ell^{\lh}_{\mu}(1,R)\bigr|\overset{\eqref{con:LpZRl}}{<}+\infty, 
\end{multline*} 
откуда по определению \eqref{df:dDlLm} логарифмической субмеры интервалов для $\mu$
\begin{equation}\label{lmu+}
\sup_{1\leq r<R<+\infty}  \bigl| \ell_{\mu}(r,R)-\ell^{\rh}_{\mu}(r,R)\bigr|+
\sup_{1\leq r<R<+\infty}  \bigl|\ell_{\mu}(r,R) -\ell^{\lh}_{\mu}(r,R)\bigr|
<+\infty. 
\end{equation} 
Из условия \eqref{lJMmull} вновь по определению  \eqref{df:dDlLm} логарифмической субмеры интервалов, но уже  для $\nu$, 
получаем соответственно
\begin{multline}\label{{l11}r}
\sup_{1\leq r<R<+\infty} \ell^{\rh}_{\nu-\mu}(r,R)
\overset{\eqref{df:dDlLm}}{\leq} \sup_{1\leq r<R<+\infty} \bigl(\ell_{\nu}(r,R)-\ell^{\rh}_{\mu}(r,R)\bigr)\\
\leq \sup_{1\leq r<R<+\infty} \Bigl(\bigl(\ell_{\nu}(r,R)-\ell_{\mu}(r,R)\bigr)
+  \bigl| \ell_{\mu}(r,R)-\ell^{\rh}_{\mu}(r,R)\bigr|\Bigr)
\overset{\eqref{lJMmull}}{<}+\infty,
\end{multline}
где в конце использовано и \eqref{lmu+}. Аналогично, из \eqref{lJMmull} и \eqref{lmu+} следует 
\begin{equation}\label{{l11}l}
\sup_{1\leq r<R<+\infty} \ell^{\lh}_{\nu-\mu}(r,R)<+\infty. 
\end{equation}
В силу включения \eqref{bst} имеем $\pm b\notin \supp (\nu+\mu)$ и по замечанию \ref{remB01} при реализации шагов [b\ref{i2}], [b\ref{i4}] можно  будет обойтись выметанием рода $1$. 

Используем [b\ref{i1}]--[b\ref{i2}], где на шаге  [b\ref{i2}]  применяем 
предложение \ref{pr52} к распределениям масс  $\nu_{\vec{-b}}\!{\lfloor}_{\CC_{ \rh}}$ и $\mu_{\vec{-b}}\!{\lfloor}_{\CC_{ \rh}}$ в роли $\nu$ и $\mu$, для которых условия включения носителей в  $\bigl\{z\in \CC\bigm|  \Re z >  a |z|\bigr\}$ и \eqref{|l|+}
обеспечены включением \eqref{bst}, а выполнение условия \eqref{|l|+} ---  соотношением \eqref{{l11}r}. Таким путём [b\ref{i2}] возникают распределение масс 
$\alpha^{\rh}$ конечной верхней плотности с носителем на $\RR^+\setminus [0,r_0)$ при некотором $r_0>0$,
а также  распределение масс $\beta^{\rh}$ с носителем на $i\RR$ и число $c^{\rh}\in \RR^+$, для которых 
\begin{gather}
\bigl(\nu_{\vec{-b}}
{\lfloor}_{\CC_{ \rh}}+\alpha^{\rh} -\mu_{\vec{-b}}
{\lfloor}_{\CC_{ \rh}}\bigr)^{\bal^1_{\rh}}
+\beta^{\rh} \overset{\eqref{bm_1}}{=}c^{\rh}\mathfrak m_1{\lfloor}_{i\RR},
\label{bm_1+}\\
\nu_{\vec{-b}}{\lfloor}_{\CC_{ \rh}}+\alpha^{\rh} +\beta^{\rh}
-\mu_{\vec{-b}}{\lfloor}_{\CC_{ \rh}}\quad\text{удовлетворяет условию Линделёфа}.
\label{uuL+}
\end{gather}

После этого используем  [b\ref{i3}]--[b\ref{i4}], где на шаге  [b\ref{i4}]  снова применяем 
предложение \ref{pr52} в зеркально симметричной относительно $i\RR$ форме 
к распределениями масс  $\nu_{\vec{b}}\!{\lfloor}_{\CC_{ \lh}}$ и $\mu_{\vec{b}}\!{\lfloor}_{\CC_{ \lh}}$ в  роли $\nu$ и $\mu$, для которых 
условия включения их носителей в угол $\bigl\{z\in \CC\bigm|  \Re z <-  a |z|\bigr\}$ и \eqref{|l|+} с $\ell^{\lh}$ вместо $\ell^{\rh}$ 
обеспечены включением \eqref{bst} и соотношением \eqref{{l11}l}. 
Таким путём возникают распределение масс 
$\alpha^{\lh}$ конечной верхней плотности с  $\supp \alpha^{\lh}\subset -\RR^+\setminus [0,r_0)$ при некотором $r_0>0$,
а также  распределение масс $\beta^{\lh}$ с  $\supp \beta^{\lh}\subset i\RR$ и число $c^{\lh}\in \RR^+$, для которых 
\begin{gather}
\bigl(\nu_{\vec{b}}{\lfloor}_{\CC_{\lh}}+\alpha^{\lh} -\mu_{\vec{b}}{\lfloor}_{\CC_{\lh}}\bigr)^{\bal^1_{\lh}}
+\beta^{\lh} \overset{\eqref{bm_1}}{=}c^{\lh}\mathfrak m_1{\lfloor}_{i\RR},
\label{bm_1-}\\
\nu_{\vec{b}}{\lfloor}_{\CC_{\lh}}+\alpha^{\lh} +\beta^{\lh}
-\mu_{\vec{b}}{\lfloor}_{\CC_{\lh}}\quad\text{удовлетворяет условию Линделёфа}.
\label{uuL-}
\end{gather}
При этом если $c^{\rh}\geq c^{\lh}$, то можем добавить распределение масс $(c^{\rh}-c^{\lh})\mathfrak m_1{\lfloor}_{i\RR}$, очевидно, конечной верхней плотности, удовлетворяющее условию Линделёфа, к правой части \eqref{bm_1-} и к $\beta^{\lh}$, сохраняя прежнее обозначение $\beta^{\lh}$ для суммы $\beta^{\lh}+(c^{\rh}-c^{\lh})\mathfrak m_1{\lfloor}_{i\RR}$. Тогда, очевидно, 
\eqref{uuL-} сохраняется, а в правых частях \eqref{bm_1+} и \eqref{bm_1-} окажется $c^{\lh}=c^{\rh}$. Аналогично поступаем при $c^{\rh}< c^{\lh}$ по отношению к \eqref{bm_1+}--\eqref{uuL+} и получим $c^{\rh}=c^{\lh}$, что всегда позволяет выбрать 
\begin{equation}\label{crl}
c:=c^{\lh}=c^{\rh}.
\end{equation}

На шаге [b\ref{i5}]  получаем требуемое распределение масс $\alpha:=\alpha^{\rh}_{\vec{b}}+\alpha^{\lh}_{\vec{-b}}$ конечной верхней плотности с носителем на $\RR\setminus [-b,b]$, 
 распределения масс $\beta_+:=\beta_{\vec{b}}^{\rh}$ и  $\beta_-:=\beta_{\vec{-b}}^{\lh}$ конечной верхней плотности с  носителями 
соответственно на прямых $b+i\RR$ и $-b+i\RR$. При этом выметания на шагах [b\ref{i2}] и [b\ref{i4}] дают выметание  
\begin{equation}\label{numuab}
(\nu+\alpha -\mu)^{\Bal^1_{b}}+\beta_++\beta_-=(\nu+\alpha+\beta_++\beta_- -\mu)^{\Bal^1_b}
\end{equation}
конечного порядка из \eqref{Balb1} с носителем на паре прямых $\pm b+i\RR$, где равенство в \eqref{numuab} следует из 
равенства $(\beta_++\beta_-)^{\Bal^1_{b}}=\beta_++\beta_-$  для распределений с носителем на $\overline \strip_b$. Равенство 
\eqref{numuab} вместе с  \eqref{bm_1+}, \eqref{bm_1-} и \eqref{crl} даёт  \eqref{bm_1b}. 
Из свойств \eqref{uuL+} и \eqref{uuL-} следует, что распределение зарядов  $\nu+\alpha+\beta_++\beta_--\mu$ удовлетворяет условию Линделёфа, а сложение этого распределения зарядов с распределением масс $\mu$, удовлетворяющим условию Линделёфа,  даёт распределение масс
$\nu+\alpha+\beta_++\beta_-$, также удовлетворяющее условию Линделёфа. 
\end{proof}

\subsection{Сдвиги и выметание $\delta$-субгармонической функции}
Для точек $w\in \CC$ аналогично $w$-сдвигу \eqref{vecz} распределений зарядов определяем {\it $w$-сдвиг $u_{\vec{w}}$ 
функции\/} $u$ на $\CC$, задаваемый как 
\begin{equation}\label{veczu}
u_{\vec{w}}\colon z\underset{z\in \CC}{\longmapsto} u(z-w) .
\end{equation}
При $w$-сдвиге  \eqref{veczu}  $\delta$-субгармонической на $\CC$ функции $\mathcal U\not\equiv \pm \infty$ она остаётся такой же,  
а распределение её зарядов Рисса претерпевает $w$-сдвиг
\begin{equation}\label{zsh}
\frac{1}{2\pi}{\bigtriangleup}(\mathcal U_{\vec{w}})\overset{\eqref{vecz}}{=}
\Bigl(\frac{1}{2\pi}{\bigtriangleup}\mathcal U\Bigr)_{\vec{w}}.
\end{equation}  
Для $\delta$-субгармонической функции  $\mathcal U\not\equiv \pm\infty$ и $b\in \RR^+$ каждую  {\it $\delta$-субгар\-м\-о\-н\-и\-ч\-е\-с\-к\-ую  функцию\/} $\mathcal U^{\Bal_b}$,  равную  функции $\mathcal U$ на  вертикальной полосе $\overline \strip_b$ ширины $2b$ из \eqref{{strip}c} вне некоторого полярного множества   и
в то же время    гармоническую  на    $\CC\setminus \overline \strip_b$, называем 
{\it  $\delta$-субгармоническим выметанием на $\overline \strip_b$ функции\/  $\mathcal U$}.

\begin{propos}[{\cite[теорема 3]{Kha22bal}}]\label{Balvs}
 Пусть $b\in \RR^+$ и  $\delta$-суб\-г\-а\-р\-м\-о\-н\-и\-ч\-е\-с\-к\-ая  функция $\mathcal U\not\equiv \pm \infty$  c распределением зарядов
 Рисса  \eqref{Delta} ко\-н\-е\-ч\-н\-ой верхней плотности 
представима  в виде разности субгармонических функций не более чем первого порядка.
 Тогда существует $\delta$-субгармоническое  выметание\/ $\mathcal U^{\Bal_b}$ на $\overline \strip_b$
функции $\mathcal U$ c распределением зарядов Рисса 
\begin{equation}\label{112s}
\frac{1}{2\pi}{\bigtriangleup}(\mathcal U^{\Bal_b})\overset{\eqref{Balb1}}{=}\varDelta_\mathcal U^{\Bal^{01}_b},
\end{equation}  
представимое вне некоторого полярного множества  в виде разности 
\eqref{reprvB} субгармонических функций $u_{\pm}\not\equiv -\infty$ не более чем первого порядка. 
\end{propos}

\begin{remark}\label{rem010}
Если в условиях предложения  \ref{Balvs}
 функция $\mathcal U$  гармоническая в двух открытых полукругах $b+r_0\DD\cap \CC_{ \rh}$ и $-b+r_0\DD\cap \CC_{ \lh}$
для некоторого $r_0>0$, то по варианту  \eqref{nur0+b} замечания \ref{remB01}   правую часть 
в \eqref{112s} можно заменить на выметание  $\varDelta_\mathcal U^{\Bal^{1}_b}$ рода $1$ на $\overline \strip_{b}$, т.е. 
 \begin{equation}\label{112s0}
\frac{1}{2\pi}{\bigtriangleup}\bigl(\mathcal U^{\Bal_b}\bigr)\overset{\eqref{112s}}{=}\varDelta_\mathcal U^{\Bal^{1}_b}.
\end{equation}  
\end{remark}

\subsection{Выметание 
 на объединение вертикальной полосы и вещественной оси}\label{bRu}

\begin{propos}[{\cite[теорема 3]{Kha22bal}}]\label{prop4_1}
Пусть  $M\not\equiv -\infty$ --- субгармоническая функция конечного типа с распределением масс  Рисса $\varDelta_M$.
Тогда для любого  $s\in \RR^+$ существуют субгармоническая   функция $M_{\RR}$ конечного типа 
с распределением масс Рисса  $\varDelta_{M_{\RR}}=\frac{1}{2\pi}{\bigtriangleup}M_{\RR}$, сосредоточенным на паре лучей  $\RR\setminus [-s,s]$, 
со свойством 
\begin{equation}\label{MMR}
\sup_{1\leq r<R<+\infty}\Bigl|\ell_{\varDelta_M}(r,R) -\ell_{\varDelta_{M_{\RR}}}(r,R)\Bigr|<+\infty, 
\end{equation} 
и субгармоническая функция $M_{s}$  конечного типа с носителем  $\supp \varDelta_{M_{s}}\subset \overline \strip_s$ 
 распределения масс Рисса $\varDelta_{M_s}=\frac{1}{2\pi}{\bigtriangleup}M_s$, для которых  
\begin{align}
M(x) &\equiv M_{\RR}(x)+M_{s}(x)\quad\text{при всех  $x\in \RR$},
\label{{MRab}r}\\
M(z) &\equiv M_{\RR}(z)+M_{s}(z)\quad\text{при всех  $z\in \overline \strip_s$}, 
\label{{MRab}i}\\
M(z) &\leq M_{\RR}(z)+M_{s}(z)\quad\text{при всех  $z\in \CC$}. 
\label{{MRab}leq}
\end{align}
\end{propos}

\section{Доказательство импликации \ref{II_2}$\Rightarrow$\ref{I} 
основной теоремы}\label{S7}

Несколько раз будет использована 
\begin{lemma}[{\cite[предложение 4.1, (4.19)]{KhaShmAbd20}}]\label{lemJl} 
Для  любой субгармонической функции $u\not\equiv -\infty$ конечного типа 
с распределением масс  Рисса $\varDelta_u$ в обозначении 
\begin{equation}\label{JiR}
J_{i\RR}(r,R; u):=\frac{1}{2\pi}\int_r^R \frac{u(-iy)+u(iy)}{y^2} \dd y, \quad 0<r<R\leq +\infty
\end{equation}
при любом $r_0\in \RR^+\setminus 0$ выполнены  соотношения 
\begin{align}
\sup_{r_0\leq r<R<+\infty} &\Bigl|J_{i\RR}(r,R;u)-\ell_{\varDelta_u}^{\rh}(r,R)\Bigr|
<+\infty,
\label{{Jll}r}
 \\
\sup_{r_0\leq r<R<+\infty} &\Bigl|J_{i\RR}(r,R;u)-\ell_{\varDelta_u}^{\lh}(r,R)\Bigr|
<+\infty,
\label{{Jll}l}
\\
\sup_{r_0\leq r<R<+\infty} &\Bigl|J_{i\RR}(r,R;u)-\ell_{\varDelta_u}(r,R)\Bigr|
<+\infty.
\label{{Jll}m} 
\end{align}
\end{lemma}

\begin{propos}\label{llJ} Пусть $\mu$ и $\nu$ --- распределения масс  конечной верхней плотности. Тогда следующие два утверждения эквивалентны: 
\begin{enumerate}
\item\label{1l} Существует   неограниченная последовательность  $(r_n)_{n\in \NN}$ в $\RR^+\setminus 0$,
возрастающая  не быстрее геометрической прогрессии в том смысле, что 
\begin{equation}
\limsup\limits_{n\to\infty}\frac{r_{n+1}}{r_n}<+\infty, 
\label{rn}
\end{equation}
для которой выполнено соотношение 
\begin{equation}
\limsup_{N\to  \infty}\sup\limits_{0\leq n<N}
\Bigl(\ell_{\nu}(r_n,r_N)-\ell_{\mu}(r_n,r_N)\Bigr)<+\infty.
\label{cprec}
\end{equation}
\item\label{2l} Для любого   $r_0\in \RR^+\setminus 0$ выполнено соотношение 
\begin{equation}\label{lJMmul+}
\sup_{r_0\leq r<R<+\infty} \Bigl(\ell_{\nu}(r,R)-\ell_{\mu}(r,R)\Bigr)<+\infty.
\end{equation} 

Кроме того, если $\mu$ --- распределение масс Рисса субгармонической функции $M\not\equiv -\infty$
конечного типа, то предыдущие утверждения \ref{1l}--\ref{2l} эквивалентны каждому из следующих двух   утверждений: 
\item\label{3l} Существует такая же, как в утверждении \ref{1l},  последовательность 
$(r_n)_{n\in \NN}$ со свойством \eqref{rn}, для которой  в обозначении 
\eqref{JiR} 
\begin{equation}\label{lMrn}
\limsup_{r_N\to  \infty}\sup\limits_{0\leq n<N}
\Bigl(\ell_{\nu}(r_n,r_N)-J_{i\RR}(r_n,r_N;M)\Bigr)<+\infty.
\end{equation}
\item\label{4l} Выполнено соотношение \eqref{lJ2r}.
\end{enumerate}
\end{propos}
\begin{proof}
Эквивалентности  \ref{1l}$\Leftrightarrow$\ref{3l} и \ref{1l}$\Leftrightarrow$\ref{4l} следуют из соотношения \eqref{{Jll}m}  леммы \ref{lemJl}. 
Импликация \ref{2l}$\Rightarrow$\ref{1l} очевидна. Если выполнено утверждение  \ref{1l}, то согласно  \eqref{rn} существует число 
$A>1$, для которого $r_{n+1}\leq Ar_n$ для всех $n\in \NN$,  где можем рассмотреть произвольное $r_0\in (0,r_1]$, а также согласно \eqref{cprec} существует  число $B>0$ и номер $N_0\in \NN$, для которых  
\begin{equation}\label{lnNA}
\ell_{\nu}(r_n,r_N)\leq \ell_{\mu}(r_n,r_N)+B
\end{equation}
при каждом  $N> N_0$ для  любых $n<N$.
При любых    $0\leq n<N\leq N_0$   
$$
\ell_{\nu}(r_n,r_N)\leq \int_{r_0}^{r_{N_0}}\Bigl|\Re\frac{1}{z}\Bigr|\dd \nu(z)\leq B_0,
$$
где  $B_0$ не зависит от $n<N$.
Таким образом, при достаточно большом  $B>0$  неравенства \eqref{lnNA} выполнены при любых целых $N>n\geq 0$.
При $r_0\leq r<R$ выберем  $n\in \NN$ и $N\geq n$ так, что  $r\in (r_{n},r_{n+1}]$ и $R\in (r_N,r_{N+1}]$. Тогда 
\begin{multline*}
\ell_{\nu}(r,R)\leq \ell_{\nu}(r_n,r_{n+1})+\ell_{\nu}(r_{n+1},r_N)+\ell_{\nu}(r_N, r_{N+1})\\
\overset{\eqref{lnNA}}{\leq} \ell_{\nu}(r_n,Ar_n) + \ell_{\mu}(r_n,r_N)+B+\ell_{\nu}(r_N,Ar_N)
\\
\leq \frac{\nu^{\rad}(Ar_n)}{r_n}+ \ell_{\mu}(r,R)+B+\frac{\nu^{\rad}(Ar_N)}{r_N},
\end{multline*}
откуда в силу конечной верхней плотности распределения масс $\nu$ имеем \eqref{lJMmul+}. 
\end{proof}

\begin{proof}[импликации \ref{II_2}$\Rightarrow$\ref{I}]
Из соотношения \eqref{lJ2} утверждения \ref{II_2} в обозначении \eqref{JiR} для некоторого   $C\in \RR$ имеем неравенства 
\begin{equation}\label{1112}
\ell_{\nu}(2^n,2^{n+1})\leq J_{i\RR}(2^n, 2^{n+1};M)+C\leq \type[M]+1+C
\end{equation} 
при всех достаточно больших $n\in \NN$. При этом из предельного соотношения в условиях 
\eqref{nubstr-} основной теоремы или из вытекающего из него по замечанию \ref{rem0} условия  [$\upnu$]  с \eqref{nubstr}
имеем, как в \eqref{ellsn},  оценки снизу для 
\begin{equation*}
\ell_{\nu}(2^n,2^{n+1})\geq \ell_\nu^{\rh}(2^n, 2^{n+1})
\overset{\eqref{ellsn}}{\geq} a 2^{-n-1}\bigl(\nu{\lfloor}_{\CC_{\overline\rh}}^{\rad}(2^{n+1})-\nu{\lfloor}_{\CC_{\overline\rh}}^{\rad}(2^{n})\bigr),
\end{equation*}
откуда согласно \eqref{1112} получаем 
\begin{equation*}
\nu{\lfloor}_{\CC_{\overline\rh}}^{\rad}(2^{n+1})-\nu{\lfloor}_{\CC_{\overline\rh}}^{\rad}(2^{n})
\leq  \bigl(\type[M]+1+C\bigr)\frac{1}{a} 2^{n+1}
\end{equation*}
при достаточно больших $n\in \NN$. Это означает, что $\nu{\lfloor}_{\CC_{\overline\rh}}$ --- распределение масс конечной верхней плотности. Аналогично то же самое  устанавливаем  и для сужения $\nu{\lfloor}_{\CC_{\lh}}$. Таким образом, 
из \eqref{lJ2} и предельного соотношения в \eqref{nubstr-} следует, что {\it распределение масс\/ $\nu$ конечной верхней плотности.\/} 

Из  эквивалентности \ref{3l}$\Leftrightarrow$\ref{2l}  предложения \ref{llJ}  
для случая двоичной последовательности из $r_n:=2^n$ в \eqref{rn} соотношение 
\eqref{lJ2} утверждения \ref{II_2} основной теоремы влечёт за собой соотношение 
\begin{equation}\label{lJMl}
\sup_{1\leq r<R<+\infty} \Bigl(\ell_{\nu}(r,R)-\ell_{\varDelta_M}(r,R)\Bigr)<+\infty
\end{equation} 
с распределением масс Рисса $\varDelta_M\overset{\eqref{Riesz}}{=}\frac{1}{2\pi}{\bigtriangleup}M$ функции $M$. 
Применим предложение \ref{prop4_1} к функции $M$ с выметанием на объединение $\RR$ с замкнутой вертикальной полосой 
 $\overline \strip_{s}$ ширины $2s$, где $s$ определено равенством в \eqref{nubstr-}
из основной теореме.   В обозначениях  предложения \ref{prop4_1} из  \eqref{lJMl} и \eqref{MMR} получаем 
\begin{multline*}
\sup_{1\leq r<R<+\infty} \Bigl(\ell_{\nu}(r,R)-\ell_{\varDelta_{M_{\RR}}}(r,R)\Bigr)
\leq \sup_{1\leq r<R<+\infty} \Bigl(\ell_{\nu}(r,R)-\ell_{\varDelta_M}(r,R)\Bigr)
\\+\sup_{1\leq r<R<+\infty}\Bigl|\ell_{\varDelta_M}(r,R) -\ell_{\varDelta_{M_{\RR}}}(r,R)\Bigr|
\overset{\eqref{lJMl},\eqref{MMR}}{<}+\infty. 
\end{multline*}
Крайние части этих неравенств можно записать как  
\begin{equation}\label{supmul}
\sup_{1\leq r<R<+\infty} \Bigl(\ell_{\nu}(r,R)-\ell_{\mu}(r,R)\Bigr)
<+\infty ,
\end{equation}
где для распределения масс Рисса функции $M_{\RR}$ использовано обозначение 
\begin{equation}\label{supmum}
\mu:=\varDelta_{M_{\RR}}, \quad\supp \mu=\supp \varDelta_{M_{\RR}}\subset \RR\setminus (-s,s).
\end{equation}

Рассмотрим произвольное фиксированное число 
\begin{equation}\label{b}
b\in (0,s).
\end{equation} 
Тогда, наряду с   \eqref{supmul}, совпадающим с условием \eqref{lJMmull}
 предложения \ref{pr52s},  по \eqref{supmum} выполнено и условие \eqref{bst} 
 предложения \ref{pr52s} с $b\in (0,s)$, а значит, и его заключения,  обозначения для  объектов из которого и используем ниже.  
   
По теореме Вейерштрасса\,--\,Адамара\,--\,Лин\-д\-е\-л\-ё\-ф\-а\,--\,Бре\-ло 
для распределений  масс $\nu+\alpha+\beta_++\beta_-$ конечной верхней плотности, удовлетворяющего условию Линделёфа, 
 существует субгармоническая функция $u$ конечного типа  с распределением масс Рисса 
\begin{equation}\label{unu}
\varDelta_u=\nu+\alpha+\beta_++\beta_-\geq \nu,
\end{equation} 
для которого  по \eqref{supmum}, построению $\alpha$ и $\beta_{\pm}$ в предложении \ref{pr52s} и по условиям  
\eqref{nubstr-} на $\nu$ в виде условия [$\upnu$] с \eqref{nubstr} из замечания \ref{rem0} имеют место включения 

\begin{equation}\label{suppdu}
\supp (\nu+\alpha-\mu)\subset \CC\setminus \bigl(\overline {\mathrm X}_a\cup \overline \strip_{b}\bigr). 
\quad \supp \beta_{\pm}\subset \pm b+i\RR.
\end{equation}. 

Рассмотрим $\delta$-субгармоническую функцию 
\begin{equation}\label{UuvM}
{\mathcal U}:=u-M_{\RR}, \quad \varDelta_{\mathcal U}:=\frac{1}{2\pi}{\bigtriangleup}{\mathcal U}=\nu+\alpha+\beta_++\beta_-
 -\mu,
\end{equation}
представленную в виде разности субгармонических функций конечного типа. 
По предложению \ref{Balvs} существует  $\delta$-субгармоническое  выметание\/ $\mathcal U^{\Bal_b}$ на $\overline \strip_b$
функции $\mathcal U$ c распределением зарядов Рисса \eqref{112s}, 
представимое вне некоторого полярного множества  в виде разности 
\eqref{reprvB} субгармонических функций  порядка $\leq 1$. 
Поскольку ввиду \eqref{suppdu} и \eqref{UuvM} функция $\mathcal U$ гармоническая  
в двух открытых полукругах $b+r_0\DD\cap \CC_{ \rh}$ и $-b+r_0\DD\cap \CC_{ \lh}$
для  $r_0:=s-b>0$, то по замечанию \ref{rem010}
правую часть в \eqref{112s} можно заменить, как в \eqref{112s0}, на выметание  $\varDelta_\mathcal U^{\Bal^{1}_b}$ рода $1$ на $\overline \strip_{b}$, которое по построению имеет вид
\begin{align*}
\varDelta_{\mathcal U}^{\Bal^{1}_b}&=(\nu+\alpha+\beta_++\beta_--\mu)^{\Bal^{1}_b}\\
&\overset{\eqref{bm_1b}}{=}
(\nu+\alpha -\mu)^{\Bal^1_{b}}+\beta_++\beta_- 
\overset{\eqref{bm_1b}}{=}c\mathfrak m_1{\lfloor}_{b+i\RR}+c\mathfrak m_1{\lfloor}_{-b+i\RR}.
\end{align*} 
В правой части здесь явно выписанное  {\it распределение масс,\/} равное домноженной на $c\in \RR^+$ сумме линейных мер Лебега на паре прямых $\pm b+i\RR$, удовлетворяющее, очевидно, условию Линделёфа.
 Отсюда, учитывая вторую часть теоремы Вейерштрасса\,--\,Адамара\,--\,Линделёфа\,--\,Брело,  $\mathcal U^{\Bal_b}$ --- {\it субгармоническая функция конечного типа,\/} равная сумме субгармонической  функции 
\begin{equation}\label{h0}
2\pi c (\Re z-b)^++2\pi c (\Re z+b)^-\underset{z\in \CC}{=}
\begin{cases}
2\pi c (\Re z-b)^+&\text{ при $\Re z>b$},\\
0&\text{ при $|\Re  z|\leq b$},\\
2\pi c (\Re z+b)^-&\text{ при $\Re z<-b$}.
\end{cases}
\end{equation}
с некоторым {\it гармоническим многочленом\/} $h$ степени $\deg h\leq 1$.  
Из этих построений по определению $\delta$-субгармонического выметания на полосу $ \overline\strip_b$
 имеем 
\begin{equation*}
\mathcal U^{\Bal_b}(z)\overset{\eqref{UuvM}}{\equiv} u(z)-M_{\RR}(z)\quad\text{при всех $z\in \overline\strip_b$},
\end{equation*}
откуда, учитывая явный вид   $\mathcal U^{\Bal_b}$ как суммы функции \eqref{h0} и $h$, 
получаем 
\begin{equation}\label{UHM}
(u-M_{\RR}-h)(z)=
\mathcal U^{\Bal_b}(z)-h(z)\overset{\eqref{h0}}{\equiv}0 \quad\text{при всех $z\in \overline\strip_b$},
\end{equation}
но вне некоторого полярного множества. 
Таким образом, построена субгармоническая функция $u-h$ конечного типа, для которой 
\begin{equation}\label{uhM}
(u-h)(z)\overset{\eqref{UHM}}{\equiv} M_{\RR}(z)\quad\text{при всех $z\in \overline\strip_b$},
\end{equation} 
 с распределением масс Рисса 
\begin{equation}\label{ununu}
\varDelta_{u-h}=\frac{1}{2\pi}({\bigtriangleup}u-{\bigtriangleup}h)= \frac{1}{2\pi}{\bigtriangleup}u \overset{\eqref{unu}}{=}\nu+\alpha+\beta_++\beta_-\geq \nu,
\end{equation} 
Для субгармонической функцией $M_s$ конечного типа  из предложения \ref{prop4_1} рассмотрим субгармоническую функцию 
$U:=(u-h)+M_s$ также конечного типа с распределением масс Рисса $\varDelta_U\geq \varDelta_{u-h}\overset{\eqref{ununu}}{\geq}\nu$, для которой на полосе $\overline \strip_b\subset  \strip_s$ получаем тождества
\begin{equation}\label{Umstb}
U(z)\equiv (u-h+M_s)(z)\underset{z\in \overline \strip_b}{\overset{\eqref{uhM}}{\equiv}} M_{\RR}(z)+M_s(z)
\underset{z\in \overline \strip_s}{\overset{\eqref{{MRab}i}}{\equiv}} M(z)
\end{equation}
вне некоторого полярного множества. Две субгармонические функции, совпадающие на открытом множестве вне полярного множества, совпадают всюду на этом открытом множестве.  Таким образом, построена  субгармоническая функция $U$ конечного типа с тождеством 
$U(z)\equiv M(z)$ для всех $z\in \strip_b$, где число $b$ выбиралось в \eqref{b} произвольным из открытого интервала $(0,s)$. Следовательно, тождество $U(z)\equiv M(z)$ можно считать выполненным  для всех $z$
из замкнутой полосы $\overline \strip_b$, и импликация \ref{II_2}$\Rightarrow$\ref{I} основной теоремы доказана. 
\end{proof}

\section{Доказательства импликаций 
\ref{I}$\Rightarrow$\ref{I_2}$\Rightarrow$\ref{I_3} 
основной   теоремы}\label{S9}

\begin{proof}[импликации \ref{I}$\Rightarrow$\ref{I_2}] 
Применение утверждения \ref{I}  с  произвольным $b'\in (b,s)$ вместо $b\in [0,s)$
обеспечивает существование субгармонической функции $U\not\equiv -\infty$ 
конечного типа с распределением масс Рисса $\varDelta_U\geq \nu$, для которой выполнено тождество      
\begin{equation}\label{Uequi}
U(z)\overset{\eqref{UeqM}}{\underset{z\in \overline\strip_{b'}}{\equiv}} M(z).
\end{equation}
Функцию $U$  можно представить в виде суммы 
трёх субгармонических функций 
\begin{equation}\label{UvarU}
U=v+m+u, \quad v\not\equiv -\infty, \quad m\not\equiv -\infty, \quad u\not\equiv -\infty,
\end{equation}
с распределениями масс Рисса конечной верхней плотности соответственно
\begin{equation}\label{Uvarvar}
\varDelta_v=\nu, \quad \varDelta_m=\varDelta_M{\lfloor}_{\strip_{s}}, \quad 
\varDelta_u=\varDelta_U-\varDelta_M{\lfloor}_{\strip_{s}}-\nu.
\end{equation}
По  тождеству \eqref{Uequi}  имеем $\varDelta_U{\lfloor}_{\strip_{b'}}=\varDelta_M{\lfloor}_{\strip_{b'}}$, и   $ \strip_{b'}\cap \supp \nu =\varnothing$ по \eqref{nubstr-}. Отсюда  $\strip_{b'} \cap \supp\varDelta_u\overset{\eqref{Uvarvar}}{=}\varnothing$ и субгармоническая функция $u$ {\it гармоническая  на открытой вертикальной полосе\/} $\strip_{b'}$ ширины $2b'>2b$. 

Неоднократно будет использована следующая 
\begin{theorem}[{\cite[основная теорема]{Kha22minorant}}]\label{th8_1}
Пусть $u\not\equiv -\infty$ --- субгармоническая функция на $\CC$, а функция 
 $r\colon \CC\to (0,1]$ удовлетворяет условию \eqref{qr}, эквивалентному\/ \eqref{nubstrr} из замечания\/ 
{\rm \ref{rem3}}.  Тогда существует целая функция $f\not\equiv 0$, для которой 
\begin{align}
\ln |f(z)|\overset{\eqref{vpD}}{\leq} u^{\bullet r}(z) &\overset{\eqref{vbc}}{\leq} u^{\circ r}(z) 
\overset{\eqref{vbc}}{\leq}\sup_{D_z(r(z))}u \quad\text{при всех $z\in \CC$},
\label{ubc}
\\
{\mathrm M}_{\ln |f|}(R)&\overset{\eqref{u}}{\leq}  {\mathrm M}_u\bigl(R+{\mathrm M}_r(R) \bigr)\quad\text{при всех $R\in \RR^+$}.
\label{Mulnf}
\end{align}
Если для  этой субгармонической  функции $u\not\equiv -\infty$ её распределение масс Рисса $\varDelta_u$  конечного порядка $\ord[\varDelta_u]<+\infty$,  то для  каждого  $d\in (0,2]$ целую функцию $f\not\equiv 0$, удовлетворяющую\/  \eqref{ubc}--\eqref{Mulnf}, можно подобрать так, что 
\begin{align}
\ln |f(z)|&\leq u(z) \quad\text{при всех $z\in \CC\setminus E$, где}
\label{{uE1}u}\\
{\mathfrak m}_d^{r}(E\cap S)&\leq \sup_{z\in S} r(z)
\quad\text{для  любого  $S\subset  \CC$}.
\label{{uE1}E}
\end{align}
\end{theorem}

Постоянная функция 
\begin{equation}\label{drfr}
r\colon z\underset{z\in \CC}{\longmapsto}\min\Bigl\{\frac12(b'-b), 1\Bigr\}>0, 
\end{equation}
очевидно, удовлетворяет условию \eqref{qr} теоремы \ref{th8_1}.  По теореме \ref{th8_1} найдётся целая функция 
$h\not\equiv 0$, для которой  с функцией $r$ из  \eqref{drfr} имеем 
\begin{equation}\label{hub}
\ln\bigl|h(z)\bigr|\overset{\eqref{ubc}}{\leq} u^{\circ r}(z)\quad\text{при всех $z\in \CC$}.
\end{equation}
Отсюда  по  представлению \eqref{UvarU} на $\CC$ получаем неравенство 
\begin{equation}\label{eqvmr}
v+m+\ln|h|\overset{\eqref{hub}}{\leq} v+m+u^{\circ r} \quad\text{на $\CC$}, 
\end{equation}
которое может быть продолжено как 
$$
v+m+\ln|h|\overset{\eqref{vbc}}{\leq} v^{\circ r}+m^{\circ r}+u^{\circ r}
\overset{\eqref{vpC}}{=}U^{\circ r}\overset{\eqref{drfr}}{\leq}U^{\circ 1},
$$
где справа субгармоническая функция  $U^{\circ 1}$ конечного типа. Следовательно,  и $v+m+\ln|h|$ --- субгармоническая функция конечного типа. 

Кроме того,  согласно выбору \eqref{drfr} постоянной функции $r$ для функции $u$, гармонической на полосе $\strip_{b'}$ ширины $2b'>2b$, 
на $\overline \strip_b$ имеем тождество 
\begin{equation}\label{ucirc}
u^{\circ r}(z)\underset{z\in \overline \strip_b}{\equiv}u(z),
\end{equation}
откуда согласно  \eqref{eqvmr} получаем
$$
v(z)+m(z)+\ln|h(z)|\overset{\eqref{eqvmr}}{\leq} v(z)+m(z)+u^{\circ r}(z)
\underset{z\in \overline \strip_b}{\overset{\eqref{ucirc}}{\equiv}}
v(z)+m(z)+u(z).
$$
Правая часть здесь по представлению \eqref{UvarU}  тождественно равна $U(z)$
для всех $z\in \CC$, а тождество  \eqref{Uequi} влечёт за собой  \eqref{umM}, что  даёт утверждение  \ref{I_2}.
\end{proof}

\begin{proof}[импликации \ref{I_2}$\Rightarrow$\ref{I_3}]  Для значения $b\in [0,s)$ выберем  $b'\in (b,s)$
и заменим функцию $r$ из \eqref{qr} на м\'еньшую функцию 
\begin{equation}\label{brr}
r_*(z)\underset{z\in \CC}{:=}\min\Bigl\{r(z), \frac12(b'-b)\Bigr\}\leq r(z)\leq 1,
\end{equation}
для которой, очевидно, по-прежнему выполнено условие \eqref{qr} и 
\begin{equation}\label{r*D}
D_z\bigl(r_*(z)\bigr)\subset \strip_{b'}\quad\text{при всех $z\in \overline \strip_{b}$.}
\end{equation}

Пусть  выполнено утверждения \ref{I_2} с числом $b'$ в роли $b$.

Применение  теоремы \ref{th8_1} с функцией $r_*$ вместо $r$ к  субгармонической функции $m$  с  распределением масс Рисса   $\frac{1}{2\pi}{\bigtriangleup}m=\frac{1}{2\pi}{\bigtriangleup}\!M{\lfloor}_{\strip_{s}}$ конечной верхней плотности 
 позволяет подобрать 
целую функцию $f\not\equiv 0$  так, что 
\begin{align}
\ln \bigl|f(z)\bigr|&\overset{\eqref{vpD}}{\leq} m^{\bullet r_*}(z) 
\quad\text{при всех $z\in \CC$},
\label{ubcm}\\
\ln \bigl|f(z)\bigr|&\leq m(z) \quad\text{при всех $z\in \CC\setminus E$,}
\label{{uE1m}u}
\end{align}
где для множества $E\subset \CC$ имеет место неравенство 
\begin{equation}\label{Er*}
{\mathfrak m}_d^{r_*}(E\cap S)\overset{\eqref{{uE1}E}}{\leq} \sup_{z\in S} r_*(z)
\quad\text{для  любого  $S\subset  \CC$}.
\end{equation}
Но из  неравенств $r(z)\underset{z\in \CC}{\overset{\eqref{brr}}{\geq}} r_*(z)$  следует  
$$
{\mathfrak m}_d^r(E\cap S)\overset{\eqref{mts}}{\leq} {\mathfrak m}_d^{r_*}(E\cap S)
\overset{\eqref{Er*}}{\leq}\sup_{z\in S} r_*(z)\leq \sup_{z\in S} r(z)
\quad\text{для  любого  $S\subset  \CC$},
$$
что для $E_b:=E$ даёт   \eqref{ubull}. При этом из неравенств  
\eqref{umM} и \eqref{{uE1m}u} получаем 
\begin{equation*}
v(z)+\ln \bigl|f(z)h(z)\bigr|\underset{z\in \CC}{\overset{\eqref{{uE1m}u}}{\leq}} 
v(z)+m(z)+\ln\bigl|h(z)\bigr|\overset{\eqref{umM}}{\leq}
M(z)\text{ при всех $z\in \overline \strip_{b'}\setminus E_b$},
\end{equation*}
что даст соотношения \eqref{ubull1}--\eqref{ubull} из утверждения \ref{I_3}, если переобозначить целую  функцию $fh\not\equiv 0$
как целую функцию $h\not\equiv 0$. 

Применяя интегральные средние \eqref{vpD} по кругам $D_z\bigl(r(z)\bigr)$  
к неравенству \eqref{umM} утверждения \ref{I_2},  получаем неравенства 
\begin{multline*}\label{umMV}
v(z)+\ln \bigl|f(z)\bigr|+\ln\bigl|h(z)\bigr|
\overset{\eqref{ubcm}}{\leq} 
v(z)+m^{\bullet r_*}(z)+\bigl(\ln|h|\bigr)(z)
\\
\overset{\eqref{vbc}}{\leq} v^{\bullet r_*}(z)+m^{\bullet r_*}(z)+\bigl(\ln|h|\bigr)^{\bullet r_*}(z)
\overset{\eqref{vpD}}{=} \bigl(v+m+\ln|h|\bigr)^{\bullet r_*}(z)
\text{ при  всех $z\in \CC$,} 
\end{multline*} 
откуда сразу следует, что   субгармоническая функция $v+\ln |fh|$ конечного типа, поскольку таковой по утверждению \ref{I_2}  является 
функция  $v+m+\ln|h|$, а функция  $r_*$ ограничена.  Кроме того, из крайних частей этих неравенств согласно   \eqref{umM} и \eqref{r*D} сразу следует 
\begin{equation*}\label{umMz}
v(z)+\ln\bigl|f(z)h(z)\bigr|\leq \bigl(v+m+\ln|h|\bigr)^{\bullet r_*}(z)
\overset{\eqref{umM},\eqref{r*D}}{\leq} M^{\bullet r_*}(z)\quad\text{при  всех $z\in \overline\strip_b$,} 
\end{equation*} 
откуда для целой функции $fh\not\equiv 0$ получаем 
\begin{equation*}
v(z)+\ln \bigl|(fh)(z)\bigr|\overset{\eqref{umM},\eqref{r*D}}{\leq} M^{\bullet r_*}(z)\overset{\eqref{brr}}{\leq} 
M^{\bullet r}(z) \quad\text{при  всех $z\in \overline\strip_b$,} 
\end{equation*} 
 что после переобозначения целой функции $fh$ как $h$ даёт \eqref{ubullet}. Таким образом, 
импликация  \ref{I_2}$\Rightarrow$\ref{I_3}, а значит,  и основная теорема доказаны.  
\end{proof}

\section{Дополнение распределения масс\\  до удовлетворяющего условиям Линделёфа}

В этом параграфе рассматриваем распределение масс $\mu$ конечной верхней плотности, 
которые удовлетворяют следующему условию: 
\begin{enumerate}
\item[{[$\upmu^{\rh}$]}] {\it существует неограниченная  последовательность  $(r_n)_{n\in \NN}$ в $\RR^+\setminus 0$, возрастающая   не быстрее геометрической прогрессии в смысле \eqref{rn}, с 
\begin{equation}\label{mu1}
\limsup\limits_{N\to  \infty}\sup\limits_{0\leq n<N}
\Bigl(\ell_{\mu}^{\lh}(r_n,r_N)-\ell_{\mu}^{\rh}(r_n,r_N)\Bigr)<+\infty. 
\end{equation}
}
\end{enumerate}
Для выполнения  [$\upmu^{\rh}$] с \eqref{mu1} достаточно любого из следующих пяти  условий:
\begin{enumerate}[{[$\upmu1$]}]
\item  $\ell_{\mu}^{\lh}(1,+\infty)<+\infty$;
\item $\supp \mu \subset \CC_{\overline \rh}$;
\item распределение масс  $\mu$ удовлетворяет  $\RR$-условию Линделёфа  \eqref{con:LpZR}--\eqref{con:LpZRl};
\item распределение масс  $\mu$ удовлетворяет  условию Линделёфа  \eqref{con:LpZ};
\item существует субгармоническая функция $M\not\equiv -\infty$ конечного типа с распределением масс Рисса $\frac{1}{2\pi}{\bigtriangleup}M=\mu$, 
\end{enumerate} 
где, очевидно,  [$\upmu1$] следует из [$\upmu2$],  [$\upmu3$] --- из [$\upmu4$], 
а по теореме  Вейерштрасса\,--\,Адамара\,--\,Ли\-н\-д\-е\-л\-ё\-фа\,--\,Брело условие 
[$\upmu5$] влечёт за собой [$\upmu4$]. 

Для подмножества $S\subset \CC$ зеркально симметричное ему относительно вещественной оси $\RR$ множество обозначаем 
через  $\bar{S}:=\bigl\{\bar z\bigm| z\in S\bigr\}$. 

Для распределения зарядов   $\nu$ на $\CC$  зеркально  симметричное 
ему относительно мнимой оси $i\RR$ распределение зарядов, обозначаем и определяем как 
\begin{equation}\label{mulhr-}
\nu^{\leftrightarrow}(S):=\nu(-\bar S) 
\quad\text{для всех $S\subset \CC$.}
\end{equation}

\begin{propos}\label{lemM} Для распределения масс $\mu$ конечной верхней плотности со свойством\/ {\rm [$\upmu^{\rh}$]}
существует распределение масс $\gamma$ конечной верхней плотности с носителем $\supp \gamma$ на отрицательной полуоси $-\RR^+$, для которого $\mu+\gamma$ удовлетворяет $\RR$-условию  Линделёфа \eqref{con:LpZR}--\eqref{con:LpZRl} и  
\begin{equation}\label{varDMmu}
0\leq \sup\limits_{1\leq r<R<+\infty}
\Bigl(\ell_{\mu+\gamma}(r,R)-\ell_{\mu}(r,R)\Bigr)<+\infty.
\end{equation}
\end{propos}
\begin{proof} Не умаляя общности, можно считать, что $0\notin \supp \mu$. 
Представим распределение масс $\mu$ в виде суммы двух его  сужений 
\begin{equation}\label{mupm}
\mu=\mu_{\overline\rh}+\mu_{\lh},  \quad  \mu_{\overline\rh}:=\mu{\lfloor}_{\CC_{\overline \rh}}, \quad 
\mu_{\lh}:=\mu{\lfloor}_{\CC_{\lh}},
\end{equation}
соответственно на замкнутую правую и открытую левую полуплоскости. Для распределения масс  $\mu_{\lh}$  рассмотрим  зеркально  симметричное ему относительно $i\RR$ распределение масс $\mu_{\lh}^{\leftrightarrow}$, определённое  
равенствами \eqref{mulhr-}. По построению 
\begin{equation}
0\notin \supp \mu_{\overline\rh}\bigcup \supp \mu_{\lh}^{\leftrightarrow}\subset \CC_{\overline\rh},
\label{mulhrs}
\end{equation}
и  для любых $0<r<R<+\infty$ по определениям \eqref{df:dDlm+}--\eqref{df:dDlm-} имеем 
\begin{multline*}
\ell_{\mu}^{\lh}(r,R)-\ell_{\mu}^{\rh}(r,R)\overset{\eqref{mupm}}{=}\ell_{\mu_{\lh}}^{\lh}(r,R)-\ell_{\mu_{\overline\rh}}^{\rh}(r,R)
\\
\overset{\eqref{mulhr-}}{=}\ell_{\mu_{\lh}^{\leftrightarrow}}^{\rh}(r,R)-\ell_{\mu_{\overline\rh}}^{\rh}(r,R)
\overset{\eqref{mulhrs}}{=}\ell_{\mu_{\lh}^{\leftrightarrow}}(r,R)-\ell_{\mu_{\overline\rh}}(r,R).
\end{multline*}
Отсюда  по условию \eqref{mu1} получаем  
\begin{equation}\label{mu1x}
\limsup\limits_{N\to  \infty}\sup\limits_{0\leq n<N}
\Bigl(\ell_{\mu_{\lh}^{\leftrightarrow}}(r_n,r_N)-\ell_{\mu_{\overline\rh}}(r_n,r_N)\Bigr)<+\infty, 
\end{equation}
а из импликации  \eqref{cprec}$\Rightarrow$\eqref{lJMmul+} предложения \ref{llJ} при 
\eqref{rn} и \eqref{mulhrs} имеем 
\begin{equation}\label{mu1x1}
\sup_{0< r<R<+\infty}
\Bigl(\ell_{\mu_{\lh}^{\leftrightarrow}}(r,R)-\ell_{\mu_{\overline\rh}}(r,R)\Bigr)<+\infty. 
\end{equation}
Рассмотрим  распределение зарядов 
\begin{equation}\label{etamu}
\eta:=\mu_{\lh}^{\leftrightarrow}-\mu_{\overline\rh},\quad 
 \supp \eta\overset{\eqref{mulhrs}}{\subset} \CC_{\overline  \rh},
\end{equation}
конечной верхней плотности, для которого  \eqref{mu1x1} означает, что 
\begin{equation}\label{slei}
\sup_{0<  r<R<+\infty} \ell_{\eta}^{\rh}(r,R)<+\infty. 
\end{equation}
Согласно конструкции предложения  \ref{addsec} можно построить  распределение масс $\alpha$ конечной верхней плотности с носителем   $\supp \alpha\subset \RR^+\setminus 0$, для которого
\begin{equation*}\label{|l|a}
\sup_{0\leq r<R<+\infty} \bigl|\ell_{\eta+\alpha}^{\rh}(r,R)\bigr|\overset{\eqref{|l|}}{\leq} 2
\sup_{0\leq  r<R<+\infty} \ell_{\eta}^{\rh}(r,R)\overset{\eqref{slei}}{<}+\infty,
\end{equation*}
что согласно \eqref{etamu} может быть записано как 
\begin{equation}\label{afrt}
\sup_{1\leq r<R<+\infty} \bigl|\ell_{\mu_{\lh}^{\leftrightarrow}+\alpha}^{\rh}(r,R)
-\ell_{\mu_{\overline\rh}}^{\rh}(r,R)\bigr|<+\infty.
\end{equation}
При этом, если положим $\gamma\overset{\eqref{mulhr-}}{:=}\alpha^{\leftrightarrow}$, то 
$\supp \gamma \subset -\RR^+$ ввиду $\supp \alpha\subset \RR^+$
и 
\begin{equation}\label{ellmg}
\begin{split}
\ell_{\mu+\gamma}^{\lh}(r,R)-\ell_{\mu+\gamma}^{\rh}(r,R) 
&=\ell_{\mu+\gamma}^{\lh}(r,R)-\ell_{\mu}^{\rh}(r,R) \\
=\ell_{\mu_{\lh}^{\leftrightarrow}+\gamma^{\leftrightarrow}}^{\rh}(r,R)-\ell_{\mu}^{\rh}(r,R) 
&=\ell_{\mu_{\lh}^{\leftrightarrow}+\alpha}^{\rh}(r,R)-\ell_{\mu}^{\rh}(r,R), 
\end{split}
\end{equation}
что вместе с   $\ell_{\mu+\gamma}^{\rh}(r,R)=\ell_{\mu}^{\rh}(r,R)\leq \ell_{\mu}(r,R)$ 
по определениям \eqref{df:dDlm+}--\eqref{df:dDlLm}  даёт 
\begin{multline*}
\sup_{1\leq r<R<+\infty} 
\bigl(\ell_{\mu+\gamma}(r,R)-\ell_{\mu}(r,R) \bigr)
\leq \sup_{1\leq r<R<+\infty}  \bigl(\ell_{\mu+\gamma}(r,R)-\ell_{\mu}^{\rh}(r,R) \bigr)
\\
\overset{\eqref{df:dDlLm}}{=}\sup_{1\leq r<R<+\infty} 
\bigl(\ell_{\mu+\gamma}^{\lh}(r,R)-\ell_{\mu}^{\rh}(r,R) \bigr)^+
\\
\overset{\eqref{ellmg}}{\leq}
\sup_{1\leq r<R<+\infty} \bigl|\ell_{\mu_{\lh}^{\leftrightarrow}+\alpha}^{\rh}(r,R)
-\ell_{\mu_{\overline\rh}}^{\rh}(r,R)\bigr|\overset{\eqref{afrt}}{<}+\infty.
\end{multline*}
 Отсюда и из  очевидного для $\mu+\gamma\geq \mu$  неравенства 
$\ell_{\mu+\gamma}(r,R)\geq\ell_{\mu}(r,R)$ получаем \eqref{varDMmu}. Кроме того, из 
\eqref{ellmg} следует 
\begin{multline*}
\sup_{1\leq r<R<+\infty}\Bigl|\ell_{\mu+\gamma}^{\lh}(r,R)-\ell_{\mu+\gamma}^{\rh}(r,R) \Bigr|\\
=\sup_{1\leq r<R<+\infty}\Bigl|\ell_{\mu_{\lh}^{\leftrightarrow}+\alpha}^{\rh}(r,R)-\ell_{\mu}^{\rh}(r,R)\Bigr|
 \overset{\eqref{afrt}}{<}+\infty,
\end{multline*}
откуда  распределение масс $\mu+\gamma$  удовлетворяет $\RR$-условию Линделёфа  \eqref{con:LpZRl}.
\end{proof}

\begin{propos}\label{propiR} Если  $\nu$ ---   распределения масс конечной верхней плотности, то 
найдётся  распределение масс $\beta$ конечной верхней плотности, для которого  $\supp \beta\subset i\RR$, а   $\nu+\beta$ удовлетворяет $i\RR$-условию Линделёфа  \eqref{con:LpZiR}.
\end{propos}
\begin{proof} Не умаляя общности, можем считать, что $0\notin \supp \nu$. 
Рассмотрим поворот на прямой угол по часовой стрелке  распределения масс $\nu$, обозначаемый и определяемый как
\begin{equation}\label{povor}
\nu^\circlearrowright(S):=\nu(-iS)\quad\text{для всех $S\subset \CC$.}
\end{equation}
Представим распределение масс $\nu^{\circlearrowright}$ в виде суммы двух его  сужений 
\begin{equation}\label{mupmc}
\nu^{\circlearrowright}=\nu^{\circlearrowright}_{\rh}+\nu^{\circlearrowright}_{\lh},  
\quad  \nu^{\circlearrowright}_{\rh}:=\nu^{\circlearrowright}{\lfloor}_{\CC_{\overline \rh}}, \quad 
\nu^{\circlearrowright}_{\lh}:=\mu{\lfloor}_{\CC_{\lh}},
\end{equation}
соответственно на замкнутую правую и открытую левую полуплоскости. 
Сопоставим левой составляющей  $\nu^{\circlearrowright}_{\lh}$ какое-нибудь распределение масс $\gamma_{\rh}$ конечной верхней плотности,  сосредоточенное на положительной полуоси, для которого 
 \begin{equation}\label{ellgk}
\ell_{\nu^{\circlearrowright}}^{\lh}(2^n,2^{n+1})
\overset{\eqref{df:dDlm-}}{=}\ell_{\nu^{\circlearrowright}_{\lh}}(2^n,2^{n+1})\leq \ell_{\gamma_{\rh}}(2^n,2^{n+1})\quad\text{при всех $n\in \NN_0$.}
\end{equation}
Сделать это можно, например, следующим поинтервально-поточечным способом. При каждом $n\in \NN_0$ и соответствующего интервала   $(2^n,2^{n+1}]$
расположим в  правом конце  $2^{n+1}$ этого интервала  массу, равную 
\begin{equation}\label{lgl1}
g_n:=2^{n+1}\ell_{\nu^{\circlearrowright}_{\lh}}^{\lh}(2^n,2^{n+1})\overset{\eqref{df:dDlm-}}{\leq} 2^{n+1}\frac{1}{2^n}\nu^{\rad}(2^{n+1})
\leq 2\nu^{\rad}(2^{n+1}),
\end{equation}  
а за распределение масс $\gamma_{\rh}$ примем  сумму всех этих масс $g_n$, сосредоточенных  в 
точках $2^{n+1}$. По построению и из неравенств в  \eqref{lgl1} следует, что 
$$
\gamma_{\rh}^{\rad}(2^{n+1})-\gamma_{\rh}^{\rad}(2^{n})=g_n\leq 2\nu^{\rad}(2^{n+1})\quad\text{для всех $n\in \NN$}, 
$$
откуда $\gamma_{\rh}$ --- распределение масс конечной верхней плотности с $\supp \gamma_{\rh} \subset \RR^+$. Кроме того, по построению  \eqref{lgl1} масс $g_n$ в точках $2^{n+1}$  имеем 
\begin{equation*}
\ell_{\gamma_{\rh}}^{\rh}(2^n,2^{n+1})\overset{\eqref{df:dDlm+}}{=}\frac{1}{2^{n+1}}g_n\overset{\eqref{lgl1}}{=}
\ell_{\nu^{\circlearrowright}_{\lh}}^{\lh}(2^n,2^{n+1}),
\end{equation*}
откуда, тем более, выполнено  \eqref{ellgk}. Рассмотрим распределение масс 
\begin{equation}\label{muhg}
\mu:=\nu^\circlearrowright+\gamma_{\rh}, 
\quad \mu{\lfloor}_{\CC_{ \lh}}=\nu^{\circlearrowright}_{\lh}, 
\quad \mu{\lfloor}_{\CC_{ \rh}}=\nu^{\circlearrowright}_{\lh}+\gamma_{\rh}, 
\end{equation} 
конечной верхней плотности, для которого согласно \eqref{ellgk} выполнено 
условие {[$\upmu^{\rh}$]} с соотношением \eqref{mu1} для последовательности чисел $r_n\underset{n\in \NN_0}{:=}2^n$.
По предложению \ref{lemM} для распределения масс $\mu$ 
существует распределение масс $\gamma_{\lh}$ конечной верхней плотности с $\supp \gamma\subset -\RR^+$, для которого 
$\mu+\gamma_{\lh}$ удовлетворяет $\RR$-условию  Линделёфа \eqref{con:LpZR}--\eqref{con:LpZRl}.
Если положим $\gamma:=\gamma_{\lh}+\gamma_{\rh}$, то по  построению $\gamma$ --- распределение масс конечной верхней плотности с 
$\supp \gamma\subset \RR$, а $\nu^\circlearrowright+\gamma\overset{\eqref{muhg}}{=}\mu+\gamma_{\lh}$
 удовлетворяет $\RR$-условию Линделёфа \eqref{con:LpZR}--\eqref{con:LpZRl}. Очевидно, существует распределение масс $\beta$ конечной верхней плотности  с $\supp \beta\subset i\RR$, поворот которого по часовой стрелке даёт $\gamma\overset{\eqref{povor}}{=}\beta^{\circlearrowright}$, откуда 
\begin{equation*}
(\nu+\beta)^\circlearrowright\overset{\eqref{povor}}{=}
\nu^\circlearrowright+\beta^\circlearrowright=
\nu^\circlearrowright+\gamma
\end{equation*}
--- это распределение масс конечной верхней плотности, которое удовлетворяет $\RR$-ус\-л\-о\-в\-ию Линделёфа   
\eqref{con:LpZR}--\eqref{con:LpZRl}. Соответственно по определению 
 \eqref{con:LpZiR} распределение масс $\nu+\beta$ конечной верхней плотности удовлетворяет  $i\RR$-условия Линделёфа  \eqref{con:LpZiR}, что завершает доказательство. 
\end{proof}

\begin{propos}\label{lemMiR} Для распределения масс $\mu$ конечной верхней плотности со свойством\/ {\rm [$\upmu^{\rh}$]} существует распределение масс $\varDelta\geq \mu$ конечной верхней плотности, удовлетворяющее условию  Линделёфа \eqref{con:LpZ}, для  которого 
\begin{gather}
\CC_{\rh}\bigcap \supp \varDelta=\CC_{\rh}\bigcap \supp \mu,
\label{vD+r}
\\
0\leq \sup\limits_{1\leq r<R<+\infty}
\Bigl(\ell_{\varDelta}(r,R)-\ell_{\mu}(r,R)\Bigr)<+\infty.
\label{vD+l}
\end{gather}
\end{propos}
\begin{proof} По предложению  \ref{lemM} найдётся распределение масс $\gamma$ конечной верхней плотности с носителем $\supp \gamma$ на отрицательной полуоси $-\RR^+$, для которого $\mu+\gamma$ удовлетворяет $\RR$-условию  Линделёфа \eqref{con:LpZR}--\eqref{con:LpZRl} и  
выполнено \eqref{varDMmu}. По предложению \ref{propiR} 
для  распределения масс  $\nu:=\mu+\gamma$ 
найдётся  распределение масс $\beta$ конечной верхней плотности, для которого   $\supp \beta\subset i\RR$, а   $\nu+\beta=\mu+\gamma+\beta$ удовлетворяет $i\RR$-условию Линделёфа  \eqref{con:LpZiR}.

Положим $\varDelta:=\mu+\gamma+\beta$. Тогда ввиду  $\supp \gamma \cup \supp \beta\subset (-\RR^+)\cup i\RR$ получаем \eqref{vD+r}, а из равенства $\ell_{\varDelta}=\ell_{\mu+\gamma}$ 
и соотношения \eqref{varDMmu} следует  \eqref{vD+l}, что  завершает доказательство предложения \ref{lemMiR}.  
\end{proof}
\begin{remark}\label{lemrl} Вместо условия [$\upmu^{\rh}$] можно было рассматривать зеркально симметричное относительно мнимой оси 
$i\RR$  условие 
\begin{enumerate}
\item[{[$\upmu^{\lh}$]}] {\it существует неограниченная  последовательность  $(r_n)_{n\in \NN}$ в $\RR^+\setminus 0$, возрастающая   не быстрее геометрической прогрессии в смысле \eqref{rn}, с 
\begin{equation}\label{mu1l}
\limsup\limits_{N\to  \infty}\sup\limits_{0\leq n<N}
\Bigl(\ell_{\mu}^{\rh}(r_n,r_N)-\ell_{\mu}^{\lh}(r_n,r_N)\Bigr)<+\infty. 
\end{equation}
}
\end{enumerate}
При выборе свойства\/ {\rm [$\upmu^{\lh}$]} вместо\/ {\rm [$\upmu^{\rh}$]} в предложении \ref{lemM} 
 распределение масс $\gamma$ конечной верхней плотности строится с носителем $\supp \gamma$ на положительной полуоси $\RR^+$, а в предложении \ref{lemMiR} следует поменять первое  заключение \eqref{vD+r}
на равенство  $\CC_{\lh}\bigcap \supp \varDelta=\CC_{\lh}\bigcap \supp \mu$ с сохранением \eqref{vD+l}. 

\end{remark}

\begin{remark} В недавней нашей с А.Е. Салимовой   статье приводится результат 
\cite[теорема 3.2]{SalKha21}  о дополнении распределения точек ${\mathrm Z}\subset \CC$ конечной верхней плотности до распределения точек конечной верхней плотности, удовлетворяющей условии Линделёфа. При этом в доказательстве этого результата  для 
 последовательности $(r_n)_{n\in \NN}$ из \eqref{rn} неявно использовано  условие
\begin{equation*}
\limsup\limits_{N\to  \infty}\sup\limits_{0\leq n<N}
\Bigl(\ell_{\mathrm Z}^{\lh}(r_n,r_N)-\ell_{\mathrm Z}^{\rh}(r_n,r_N)\Bigr)<+\infty 
\end{equation*} 
или зеркально симметричное ему условие такого же вида, в котором 
$\ell_{\mathrm Z}^{\lh}$ и $\ell_{\mathrm Z}^{\lh}$ меняются местами. Эти условия являются версиями соответственно условий {[$\upmu^{\rh}$]} или {[$\upmu^{\rh}$]} и какое-нибудь из них  следует включить в формулировку 
\cite[теорема 3.2]{SalKha21}, а также требовать выполнения одного  из них применительно к   
 распределению точек ${\mathrm W}\subset \CC$ вместо ${\mathrm Z}$ в \cite[следствие 3.1, теоремы 4.1--4.3]{SalKha21}.
\end{remark}

\section{Варианты основных результатов для пар распределений масс или точек}\label{mainresMR}

\subsection{Развития теоремы Мальявена\,--\,Рубела}\label{dMR}
В исходной для настоящей статьи теореме Мальявена\,--\,Рубела и теореме  \ref{theKh}, сформулированных в   п.~\ref{prr1_2},  в условиях рассматриваются распределения точек ${\mathrm Z}\subset \CC$ и ${\mathrm W}\subset \CC_{\rh}$ конечной верхней плотности.  Развивая такую постановку, в этом п.~\ref{dMR}   далее   вместо  распределения точек ${\mathrm W}\subset \CC_{\rh}$ рассматриваем распределение масс $\mu$.

\begin{theorem}\label{th_mu} Пусть $\mu$ --- распределение масс конечной верхней плотности со свойством\/  {\rm [$\upmu^{\rh}$],}
а распределение масс $\nu$ удовлетворяет условиям \eqref{nubstr-}. 
Тогда следующие пять  утверждений\/ {\rm \ref{Imu}--\ref{II_3mu}} эквивалентны:  
\begin{enumerate}[{\rm I.}]
\item\label{Imu}   Для  любого  $b\in [0,s)$ и любой субгармонической функции 
$M\not\equiv -\infty$ конечного типа с распределением масс Рисса $\frac{1}{2\pi}{\bigtriangleup}M\geq\mu$
существует субгармоническая функция $U\not\equiv -\infty$ конечного типа 
с  распределением масс Рисса 
$\frac{1}{2\pi}{\bigtriangleup}U\geq \nu$, для которой имеет место тождество \eqref{UeqM}.

\item\label{II_2mu} При  значениях $n$ и $N$, пробегающих   соответственно\/ $\NN_0$ и\/ $\NN$,  имеем 
\begin{equation}\label{lJ2mu}
\limsup\limits_{N\to  \infty}\sup\limits_{0\leq n<N}
\Bigl(\ell_{\nu}(2^n,2^N)-\ell_{\mu}(2^n,2^N)\Bigr)<+\infty.
\end{equation}

\item\label{I_2mu}  
Для  любой субгармонической функции $M\not\equiv -\infty$ конечного типа с распределением масс Рисса $\frac{1}{2\pi}{\bigtriangleup}M\geq\mu$ и любой пары субгармонических функций $v$ и $m$ с распределениями масс Рисса соответственно 
$\frac{1}{2\pi}{\bigtriangleup} v=\nu$ и  $\frac{1}{2\pi}{\bigtriangleup}m=\frac{1}{2\pi}{\bigtriangleup}M{\lfloor}_{\strip_{s}}$
при   каждом $b\in [0,s)$ найдётся целая функция $h\not\equiv 0$, с которой  сумма $v+m+\ln |h|$ ---  субгармоническая функция
  конечного типа и выполнены неравенства  \eqref{umM}.

\item\label{I_3mu} Для  любой субгармонической функции $M\not\equiv -\infty$ конечного типа с распределением масс Рисса $\frac{1}{2\pi}{\bigtriangleup}M\geq\mu$ и для  произвольной субгармонической функции   $v$ с распределением масс Рисса $\frac{1}{2\pi}{\bigtriangleup} v=\nu$ при любых  $b\in [0,s)$, $d\in (0,2]$ и функции $r\colon \CC\to (0,1]$ с ограничением \eqref{qr} найдутся целая функция $h\not\equiv 0$ и подмножество $E_b\subset \CC$, для которых   $v+\ln |h|$ ---  субгармоническая функция  конечного типа  и имеет место \eqref{ubullet}--\eqref{ubull}.

\item\label{II_3mu} Существуют  субгармоническая  функция $M\not\equiv -\infty$ конечного типа с 
 \begin{equation}\label{rsM}
\varDelta_M\overset{\eqref{Riesz}}{:=}\frac{1}{2\pi}{\bigtriangleup}M\geq\mu, \quad 
\varDelta_M{\lfloor}_{\CC_{\rh}}=\mu{\lfloor}_{\CC_{\rh}}, 
\end{equation}
и  функции $q_0$,  $q$, множество $E\subset \CC$ и субгармоническая функция  $U\not\equiv -\infty$ такие же, как 
в утверждении\/ {\rm \ref{II_3}} основной теоремы,  с \eqref{{UM0}M}--\eqref{{UM0}E}.
\end{enumerate}
\end{theorem}
\begin{proof}
Нетрудно понять, что импликации  \ref{Imu}$\Rightarrow$\ref{I_2mu}$\Rightarrow$\ref{I_3mu}  --- это в точности
импликации  \ref{I}$\Rightarrow$\ref{I_2}$\Rightarrow$\ref{I_3} основной теоремы.
Установим \ref{I_3mu}$\Rightarrow$\ref{II_3mu}$\Rightarrow$\ref{II_2mu}$\Rightarrow$\ref{Imu}.

Выведем из  утверждения \ref{I_3mu} утверждение  \ref{II_3mu}.
По  предложению \ref{lemMiR} существует распределение масс $\varDelta\geq \mu$ конечной верхней плотности, удовлетворяющее условию  Линделёфа \eqref{con:LpZ} со  свойствами  \eqref{vD+r}--\eqref{vD+l}. По теореме Вейерштрасса\,--\,Адамара\,--\,Ли\-н\-д\-е\-л\-ё\-фа\,--\,Брело найдётся субгармоническая функция $M\not\equiv -\infty$ с распределением масс Рисса $\frac{1}{2\pi}{\bigtriangleup}M=\varDelta\geq \mu$.  По утверждению  
\ref{I_3mu}  при  $b:=0$ существует субгармоническая  функция  конечного типа 
 $U:=v+\ln|h|\not\equiv -\infty$ из \eqref{ubullet}--\eqref{ubull} с распределением масс Рисса $\frac{1}{2\pi}{\bigtriangleup}U\geq \nu$, удовлетворяющая неравенству $U(iy)\leq M(iy)$  для всех $iy\in (\CC\setminus E_0)\cap i\RR$, где при $d:=1$ и  выборе $r\equiv 1$ в \eqref{qr} имеем ${\mathfrak m}_1^1(E_0) <+\infty$.  Таким образом,  из утверждения   \ref{I_3mu}   имеем $U(iy)+U(-iy)\leq M(iy)+M(-iy)$ для всех $y\in \RR^+\setminus E$, где $E:=\bigl(iE_0)\cup (-iE_0)\bigr)\cap \RR^+$ конечной лебеговой меры ${\mathfrak m}_1(E) <+\infty$. Для такого подмножества $E\subset \RR^+$ имеет место \eqref{qEr1}, 
откуда  при выборе $q_0=q=0$ получаем как конечность интеграла из \eqref{{UM0}E}, так и неравенства  \eqref{{UM0}M} при всех $y\in \RR^+\setminus E$. Это доказывает истинность 
импликации \ref{I_3mu}$\Rightarrow$\ref{II_3mu}.

Если выполнено \ref{II_3mu}, то из импликации \ref{II_3}$\Rightarrow$\ref{II_2} основной теоремы следует 
соотношение \eqref{lJ2}. Это  в сочетании с  соотношением \eqref{{Jll}m} леммы \ref{lemJl}, записанным в обозначении  \eqref{JiR},  даёт неравенства 
\begin{multline}\label{lsmM}
\limsup\limits_{N\to  \infty}\sup\limits_{0\leq n<N}\Bigl(\ell_{\nu}(2^n,2^N)-\ell_{\varDelta_M}(2^n,2^N)\Bigr)
\\
\leq \limsup\limits_{N\to  \infty}\sup\limits_{0\leq n<N}\Bigl(\ell_{\nu}(2^n,2^N)-J_{i\RR}(2^n,2^N;M)\Bigr)\\
+\limsup\limits_{N\to  \infty}\sup_{0\leq n<N} \bigl|J_{i\RR}(2^n,2^N;M)-\ell_{\varDelta_M}(2^n,2^N)\bigr|
\overset{\eqref{lJ2},\eqref{{Jll}m}}{<}+\infty.
\end{multline}
 Но для распределения масс $\varDelta_M$, удовлетворяющего условию Линделёфа и, тем более, $\RR$-условию  Линделёфа в форме \eqref{con:LpZRl}, выполнено
$$
\sup_{1\leq r<R<+\infty}\bigl| \ell_{\varDelta_M}(r,R)-\ell_{\varDelta_M}^{\rh}(r,R)\bigr|<+\infty.
$$
Отсюда   согласно \eqref{lsmM} получаем 
\begin{equation}\label{nNmuM}
\limsup\limits_{N\to  \infty}\sup\limits_{0\leq n<N}\Bigl(\ell_{\nu}(2^n,2^N)-\ell_{\varDelta_M}^{\rh}(2^n,2^N)\Bigr)
<+\infty.
\end{equation}
Ввиду  $\varDelta_M{\lfloor}_{\CC_{\rh}}\overset{\eqref{rsM}}{=}\mu{\lfloor}_{\CC_{\rh}}$, 
 имеем равенства 
$$
\ell_{\varDelta_M}^{\rh}=\ell_{{\varDelta_M}{\lfloor}_{\CC_{\rh}}}=\ell_{\mu{\lfloor}_{\CC_{\rh}}}=\ell_{\mu}^{\rh}\leq \ell_{\mu},
$$ 
что  позволяет из  соотношения  \eqref{nNmuM} получить \eqref{lJ2mu} и утверждение \ref{II_2mu}. 

Из соотношения \eqref{lJ2mu}  утверждения \ref{II_2mu} следует, что для любой 
субгармонической функции $M\not\equiv -\infty$ конечного типа с распределением масс Рисса $\varDelta_M=\frac{1}{2\pi}{\bigtriangleup}M\geq\mu$ тем более выполнено 
соотношение 
\begin{equation}\label{limbm}
\limsup\limits_{N\to  \infty}\sup\limits_{0\leq n<N}
\Bigl(\ell_{\nu}(2^n,2^N)-\ell_{\varDelta_M}(2^n,2^N)\Bigr)<+\infty.
\end{equation}
Отсюда по соотношению \eqref{{Jll}m} леммы \ref{lemJl}, записанным в обозначении  \eqref{JiR}, 
\begin{multline*}
\limsup\limits_{N\to  \infty}\sup\limits_{0\leq n<N}\Bigl(\ell_{\nu}(2^n,2^N)-J_{i\RR}(2^n,2^N;M)\Bigr)
\\
\leq \limsup\limits_{N\to  \infty}\sup\limits_{0\leq n<N}
\Bigl(\ell_{\nu}(2^n,2^N)-\ell_{\varDelta_M}(2^n,2^N)\Bigr)
\\
+\limsup\limits_{N\to  \infty}\sup_{0\leq n<N} \bigl|\ell_{\varDelta_M}(2^n,2^N)-J_{i\RR}(2^n,2^N;M)\bigr|
\overset{\eqref{limbm},\eqref{{Jll}m}}{<}+\infty,
\end{multline*}
где крайние части и есть соотношение \eqref{lJ2} из  утверждения \ref{II_2}
основной теоремы. Из импликации \ref{II_2}$\Rightarrow$\ref{I} основной теоремы 
для  любого  $b\in [0,s)$ существует субгармоническая функция $U\not\equiv -\infty$ конечного типа 
с  распределением масс Рисса $\frac{1}{2\pi}{\bigtriangleup}U\geq \nu$, для которой 
выполнено тождество \eqref{UeqM}. Таким образом, импликация 
\ref{II_2mu}$\Rightarrow$\ref{Imu} истинна и теорема доказана. 
 \end{proof}

\begin{theorem}\label{th2_1Z} Пусть $\mu$ --- распределение масс конечной верхней плотности со свойством\/  {\rm [$\upmu^{\rh}$],}  а   распределение точек   ${\mathrm Z}$ такое же, как в теореме\/ {\rm \ref{th2_1}}.
Тогда следующие четыре  утверждения\/ {\rm \ref{IeZ}--\ref{IIZ}} эквивалентны:
\begin{enumerate}[{\rm I.}]
\item \label{IeZ}  
При любых\/  $0\leq b<s\in \RR^+$ для  произвольной   субгармонической функции $M\not\equiv -\infty$ конечного типа с  $\frac{1}{2\pi}{\bigtriangleup}M\geq\mu$   и любой  субгармонической функции $m$ с распределением масс Рисса 
$\frac{1}{2\pi}{\bigtriangleup}m=\frac{1}{2\pi}{\bigtriangleup}M{\lfloor}_{\strip_s}$
найдётся  целая функция $f\not\equiv 0$ с  $f({\mathrm Z})=0$,  для которой  субгармоническая функция 
$\ln |f|+m$ конечного типа и выполняются неравенства \eqref{umMe}.

\item\label{II2Z} При  значениях $n$ и $N$, пробегающих   соответственно\/ $\NN_0$ и\/ $\NN$,  имеем 
\begin{equation}\label{ellZMmu}
\limsup\limits_{N\to  \infty}\sup\limits_{0\leq n< N}
\Bigl(\ell_{\mathrm Z}(2^n,2^N)-\ell_{\mu}(2^n,2^N)\Bigr)<+\infty.
\end{equation}

\item \label{IIeZ}  
 При любых  $b\in \RR^+$, $d\in (0,2]$ и функции $r\colon \CC\to (0,1]$ из  \eqref{qr}
для  любой  субгармонической функции $M\not\equiv -\infty$ конечного типа с  $\frac{1}{2\pi}{\bigtriangleup}M\geq\mu$ 
найдутся целая функция $f\not\equiv 0$ экспоненциального типа   с  $f({\mathrm Z})=0$ и   $E_b\subset \CC$, для которых 
 $\ln\bigl|f(z)\bigr|\leq M^{\bullet r}(z)$ {при всех $z\in \overline\strip_b$, а также}  $\ln\bigl|f(z)\bigr|\leq M(z)$ при всех $z\in \overline\strip_b\setminus E_b$, где для $E_b$ выполнено \eqref{ubull}.

\item\label{IIZ} 
Для распределения масс $\nu:={\mathrm Z}$ 
выполнено утверждение\/ {\rm \ref{II_3mu}}    теоремы\/ {\rm \ref{th_mu}}. 
\end{enumerate}
 \end{theorem}
Вывод  теоремы \ref{th2_1Z} из теоремы \ref{th_mu} опускаем, поскольку он во многом почти дословно повторяет 
вывод теоремы  \ref{th2_1} из основной теоремы в п.~\ref{Ss2_2}. То же самое относится и к следующему следствию 
\ref{corefZ}, которое можно вывести из  теоремы \ref{th2_1Z} по схеме, аналогичной последовательному выводу 
следствий \ref{coref} и \ref{th2_2} из теоремы   \ref{th2_1}, также  изложенному в п.~\ref{Ss2_2}.
 \begin{corollary}\label{corefZ} 
Если в условиях теоремы\/ {\rm \ref{th2_1Z}} распределение масс $\mu$  целочисленное, т.е. является распределением точек ${\mathrm W}=\mu$, то  свойство\/  {\rm [$\upmu^{\rh}$]} с \eqref{mu1} при $r_n\underset{n\in \NN_0}{:=}2^n$ 
эквивалентно   свойству \eqref{ellZMgMR}, а каждое из  трёх  утверждений {\rm\ref{Ieg0MR}}--{\rm \ref{IIg0MR}}    теоремы\/  {\rm \ref{th1_5MR}} эквивалентно утверждению\/ {\rm \ref{IIZ}} теоремы\/ 
{\rm \ref{th2_1Z}}. 
\end{corollary}
\begin{remark} Если в условиях теорем \ref{th_mu}  и \ref{th2_1Z} свойство 
[$\upmu^{\rh}$] заменить на  зеркально симметричное относительно мнимой оси свойство  
[$\upmu^{\lh}$], то, с учётом замечания \ref{lemrl}, единственное изменение, которое необходимо внести в формулировки теорем 
\ref{th_mu}  и \ref{th2_1Z} --- это заменить в \eqref{rsM} второе <<правостороннее>> равенство 
$\varDelta_M{\lfloor}_{\CC_{\rh}}=\mu{\lfloor}_{\CC_{\rh}}$ на <<левостороннее>>
$\varDelta_M{\lfloor}_{\CC_{\lh}}=\mu{\lfloor}_{\CC_{\lh}}$.
Соответственно и в теореме \ref{th1_5MR} условие \eqref{ellZMgMR}
на ${\mathrm W}$ можно заменить на зеркально симметричное относительно мнимой оси $i\RR$ условие 
\begin{equation*}
\limsup\limits_{N\to  \infty}\sup\limits_{0\leq n<N}
\Bigl(\ell_{\mathrm W}(2^n,2^N)-\ell_{{\mathrm W}{\lfloor}_{\CC_{ \lh}}}(2^n,2^N)\Bigr)<+\infty,
\end{equation*}
с сужением  ${\mathrm W}{\lfloor}_{\CC_{ \lh}}$  ${\mathrm W}$ на $\CC_{ \lh}$, которое, очевидно, выполнено,  при\/ ${\mathrm W}\subset \CC_{\overline\lh}$. При этом единственное дополнительное изменение в теореме \ref{th1_5MR}, которое необходимо, --- это замена    правостороннего равенства  $\Zero_g {\lfloor}_{\CC_{\rh}}={\mathrm W}{\lfloor}_{\CC_{\rh}}$ в утверждении  \ref{IIg0MR}  на левостороннее  $\Zero_g {\lfloor}_{\CC_{\lh}}={\mathrm W}{\lfloor}_{\CC_{\lh}}$.
\end{remark}

\subsection{Заключительные комментарии}\label{dBM}
Не касаясь применений результатов статьи, которые указаны в конце введения,  отметим некоторые возможные  пути дальнейшего естественного развития основных  результатов, не претендуя, впрочем, на полноту описания перспектив такого развития. 

Вместо внешней плотности Редхеффера от распределений точек в формулировках результатов, где она  встречается,  можно использовать многочисленные известные и равные ей внешние плотности Бёрлинга\,--\,Мальявена, Кахана и др. \cite{BM67}, \cite{Kah62}, \cite{Red77}, \cite{Koosis88}, \cite{Koosis92}.  Взаимосвязи между этими плотностями, дополненные и новыми,   наиболее детально исследованы   И.\,Ф.~Красичковым-Терновским  \cite{Kra89} c реализации и для распределений масс.  Кроме того, можно  использовать и иные, гораздо более тонкие, плотности и характеристики  распределений точек и масс, порождаемые специальными классами тестовых субгармонических функций    из работ \cite{KhaTalKha14} и \cite{BaiTalKha16}, являющихся аналогами классов основных функций теории обобщённых функций. Именно эти тестовые плотности и характеристики  представляются наиболее естественным развитием логарифмических функций интервалов и субмер для частей распределений точек или масс вблизи или на мнимой оси. Такой подход через тестовые функции уже неплохо зарекомендовал себя при доказательстве теорем Бёрлинга\,--\,Мальявена  \cite{Kha94}, \cite[гл.~III]{Koosis96}, теоремы \ref{thK1-1} и её более общих версий в \cite{Kha01l},  при решении  многих других задач в \cite{Kha99}, \cite{Kha09}, 
\cite{KhaKhaChe08I}, а из  недавних --- в \cite{KhaRoz18}, \ \cite{KhaKha19A}, \cite{KhaKha19},
\cite{MenKha20}, \cite{KhaKha21} и т.д. Метод  тестовых субгармонических функций, двойственный в теоретико-потенциальном смысле к методу огибающей из 
\cite{KhaRozKha19}, \cite{Kha21b},  позволит учитывать и значительно  более жёсткие, чем в теоремах \ref{thBM1},
\ref{th1_5}, \ref{th1_5MR} и в основных результатах из \S~\ref{mainres} и \S~\ref{mainresMR}, 
 ограничения на  целые функции  экспоненциального типа или субгармонические функции конечного типа,  учитывающие точные значения величины их типа или даже ограничения на индикаторы роста по направлениям.
 
Наконец, постановки задач \eqref{fgiRM}, \eqref{UM=}, \eqref{UMbbul}, \eqref{{fgiRMu}l}  
во введении   возможны не только для субгармонических функций $M$ конечного типа, но и для достаточно произвольных расширенных числовых функций $M$, определённых лишь на мнимой оси или на вертикальной полосе $\overline\strip_b$ из \eqref{{strip}c}. Такие версии постановок задач в духе и строго в рамках теоремы Бёрлинга\,--\,Мальявена о мультипликаторе затронуты в 
\cite{BM62},  \cite{Mal79}, \cite{HJ94}, \cite{MasNazHav05}  и наиболее детально исследованы в монографиях П. Кусиса   \cite{Koosis88}, \cite{Koosis92}, а также в ряде его последующих статей. Один из возможных подходов при таком развитии --- аппроксимация расширенных числовых функций $M$ на  мнимой оси или полосе субгармоническим функциями конечного типа или их разностями в поточечных,  интегральных или  равномерных  метриках вне малых исключительных множеств. Это уже иная задача, представляющая и самостоятельный интерес.

\end{fulltext}

\end{document}